\documentclass[11pt]{amsart}

\usepackage{mathtools}
\usepackage{amsmath, amsthm, amssymb, amsfonts, enumerate}
\usepackage{dsfont}
\usepackage{color}
\usepackage{geometry}
\usepackage{todonotes}
\usepackage{graphicx}
\usepackage{caption}
\usepackage{float}
\usepackage{soul}

\usepackage[colorlinks=true,linkcolor=blue,urlcolor=blue, citecolor=blue]{hyperref}

\mathtoolsset{showonlyrefs}

\def \cA{\mathcal{A}}
\def \cB{\mathcal{B}}
\def \cC{\mathcal{C}}
\def \cD{\mathcal{D}}

\def \cI{\mathcal{I}}
\def \cJ{\mathcal{J}}
\def \cK{\mathcal{K}}
\def \cL{\mathcal{L}}

\def \cO{\mathcal{O}}

\def \cT{\mathcal{T}}
\def \P{\mathsf P}

\def \E{\mathsf E}

\def \N{\mathbb{N}}
\def \R{\mathbb{R}}
\def \F{\mathbb F}

\def \ud{\mathrm{d}}

\newcommand{\eps}{\varepsilon}

\newtheorem{theorem}{Theorem}[section]
\newtheorem{lemma}[theorem]{Lemma}
\newtheorem{corollary}[theorem]{Corollary}
\newtheorem{proposition}[theorem]{Proposition}

\newtheorem{remark}[theorem]{Remark}
\newtheorem{assumption}[theorem]{Assumption}

\theoremstyle{definition}
\newtheorem{problem}{Problem}

\DeclareMathOperator*{\argmax}{arg\,max}

\makeatletter
\@namedef{subjclassname@2020}{\textup{2020}Mathematics Subject Classification}
\makeatother

\title[Variational inequalities for controller vs.\ stopper games]{Variational inequalities on unbounded domains\\ for zero-sum singular-controller vs.\ stopper games}
\author[Bovo]{Andrea Bovo}
\author[De Angelis]{Tiziano De Angelis}
\author[Issoglio]{Elena Issoglio}
\subjclass[2020]{91A05, 91A15, 60G40, 93E20, 49J40, 35K58}
\keywords{zero-sum stochastic games, singular control, optimal stopping, controlled diffusions, variational inequalities, obstacle problems, gradient constraints, penalisation methods, unbounded domains.}
\address{A.\ Bovo: School of Management and Economics, Dept.\ ESOMAS, University of Torino, Corso Unione Sovietica, 218 Bis, 10134, Torino, Italy.}
\email{\href{mailto:andrea.bovo@unito.it}{andrea.bovo@unito.it}}
\address{T.\ De Angelis: School of Management and Economics, Dept.\ ESOMAS, University of Torino, Corso Unione Sovietica, 218 Bis, 10134, Torino, Italy; Collegio Carlo Alberto, Piazza Arbarello 8, 10122, Torino, Italy.}
\email{\href{mailto:tiziano.deangelis@unito.it}{tiziano.deangelis@unito.it}}
\address{E.\ Issoglio: Dept.\ of Mathematics ``G.\ Peano'', University of Torino, Via Carlo Alberto, 10, 10123, Torino, Italy.}
\email{\href{mailto:elena.issoglio@unito.it}{elena.issoglio@unito.it}}
\date{\today}

\numberwithin{equation}{section}

\begin{document}

\begin{abstract}
We study a class of zero-sum games between a singular-controller and a stopper over finite-time horizon. The underlying process is a multi-dimensional (locally non-degenerate) controlled stochastic differential equation (SDE) evolving in an unbounded domain. We prove that such games admit a value and provide an optimal strategy for the stopper. The value of the game is shown to be the maximal solution, in a suitable Sobolev class, of a variational inequality of `min-max' type with obstacle constraint and gradient constraint. Although the variational inequality and the game are solved on an unbounded domain we do not require boundedness of either the coefficients of the controlled SDE or of the cost functions in the game. 
\end{abstract}

\maketitle

\section{Introduction}
We consider a class of zero-sum games on a finite-time horizon $[0,T]$ between a controller and a stopper. The underlying stochastic dynamics $X^{[n,\nu]}$ is given by a multi-dimensional, singularly controlled, stochastic differential equation of the form
\begin{equation}\label{eq:dyn}
\ud X^{[n,\nu]}_t= b(X^{[n,\nu]}_t)\ud t+\sigma(X^{[n,\nu]}_t)\ud W_t+n_t\ud \nu_t,
\end{equation}
where $W$ is a multi-dimensional Brownian motion and the control pair $(n_t,\nu_t)_{t\in[0,T]}$ is given by a unitary vector $n_t(\omega)\in\R^d$ and a real-valued, right-continuous, increasing process $\nu_t(\omega)$. The stopper decides when the game ends and receives from the controller an amount depending on a running payoff $h(t,x)$, a terminal payoff $g(t,x)$ and a cost $f(t,x)$ per unit of control exerted, see Section \ref{sec:setting} for details. Using a blend of analytical and probabilistic techniques we prove that the game admits a value $v$ which is the maximal solution in a suitable Sobolev space of the variational inequalities:
\begin{align}\label{eq:varineq}
\begin{array}{l}
\min\big\{\max\big\{\partial_tu+\mathcal{L}u-ru+h, g-u\big\},f-|\nabla u|_d\big\}=0,\\[+5pt]
\max\big\{\min\big\{\partial_t u+\mathcal{L}u -ru+h, f-|\nabla u|_d\big\},g-u\big\}=0,
\end{array}
\end{align}
a.e.\ in $[0,T)\times\R^d$ with terminal condition $u(T,x)=g(T,x)$. Here, $\cL$ is the infinitesimal generator of the uncontrolled SDE, $r\ge 0$ is a constant discount rate for the game's payoff {and $|\cdot|_d$ is the Euclidean norm in $\R^d$}. We also provide an optimal stopping rule for the stopper, which provides a best response against {\em any} admissible control for the controller.

The two variational problems in \eqref{eq:varineq} have not received much attention in the literature and they pose a number of challenges. The first obvious one is that swapping the order of `min' and `max' is non-trivial and it relates in some sense to proving the equivalence between the so-called upper and lower value of the game. Secondly, a solution of the variational problem is subject to two hard constraints: an obstacle constraint $u\ge g$ and a gradient constraint $|\nabla u|_d\le f$. Thirdly, we solve the problem on an {\em unbounded domain} but without imposing boundedness of the coefficients of the SDE or of the payoff functions, and without requiring uniform ellipticity of the matrix $\sigma\sigma^\top$ in the whole space. These seem important technical improvements even when compared to variational inequalities on unbounded domains for singular control problems (e.g., Chow et al.\ \cite{chow1985additive}, Soner and Shreve \cite{soner1989regularity,soner1991free}, Menaldi and Taksar \cite{menaldi1989optimal} and Zhu \cite{zhu1992generalized}) or optimal stopping games (e.g., Bensoussan and Friedman \cite{bensoussan1977nonzero}, Friedman \cite{friedman1973stochastic} and Stettner \cite{stettner2011penalty}). 

Our method of proof builds upon penalisation techniques that address simultaneously the two hard constraints in \eqref{eq:varineq}. We find bounds on the Sobolev norm of the solution of the penalised (semilinear) PDE problem, uniformly with respect to the penalisation parameters, thanks to analytical techniques rooted in early work by Evans \cite{evans1979second} and new probabilistic tricks developed {\em ad-hoc} in our framework. Indeed, it turns out that the co-existence of two hard constraints in \eqref{eq:varineq}, the `min-max' structure of the problem, its parabolic nature and unboundedness of the domain make the use of purely analytical ideas as in Evans \cite{evans1979second} not sufficient to provide the necessary bounds (see also the references given in the previous paragraph and more recent work by Hynd \cite{hynd2013analysis} and Kelbert and Moreno-Franco \cite{kelbert2019hjb}, for comparison). In the process of obtaining our main result (Theorem \ref{thm:usolvar}) we also contribute a detailed proof of the existence and uniqueness of the solution for the penalised problem (Theorem \ref{thm:exisolpenprb} for bounded domain and Theorem \ref{thm:highreguued}+Proposition \ref{lem:prbraprRD} for unbounded domain), which hopefully will serve as a useful reference for future work in the field. Finally, the existence of an optimal stopping time $\tau_*$ is interesting in its own right as it may enable free-boundary techniques for the study of the optimal strategy of the controller. Thanks to the structure of $\tau_*$ the game can be turned into a singular control problem with absorption along the boundary of the (unknown) contact set $\{v=g\}$. Then it may be possible to use ideas from, e.g., De Angelis and Ekstr\"om \cite{de2017dividend} to determine the optimal control and a saddle point in the game.

The probabilistic study of so-called controller-stopper zero-sum games originates from work by Maitra and Sudderth \cite{maitra1996gambler}. A `gambler' selects the distribution of a discrete-time process $(X_n)_{n\in\N}$ and a stopper ends the game at a stopping time $\tau\in\N$ of her choosing and pays an amount $u(X_\tau)$ to the gambler, for some bounded function $u$. The problem was then cast in a continuous-time framework by Karatzas and Sudderth \cite{karatzas2001controller} who consider one-dimensional It\^o diffusions whose drift and diffusion coefficients are chosen by the controller. They obtain (almost explicit) optimal strategies for both players using methods based on the general theory of one-dimensional linear diffusions. Weerasinghe \cite{weerasinghe2006zsg} studied a similar problem, in which the underlying dynamics is a one-dimensional SDE whose diffusion coefficient is controlled, and finds that the game admits a value which is not continuously differentiable as function of the initial state of the process.

Prior to \cite{maitra1996gambler}, Bensoussan and Friedman \cite{bensoussan1974nonlinear} solved, via penalisation methods, nonlinear variational inequalities linked to zero-sum games in which each player uses both a control and a stopping time.

Following those early contributions, the literature on zero-sum (and to some extent also nonzero-sum) controller-stopper games has grown steadily. A wide variety of methods has been deployed spanning, for example, martingale theory (Karatzas and Zamfirescu \cite{karatzas2008martingale}), backward stochastic differential equations (e.g., Hamad\`ene \cite{hamadene2006stochastic}, Choukroun et al.\ \cite{choukroun2015bsde}) and solution of variational problems via viscosity theory (e.g., Bayraktar and Huang \cite{bayraktar2013controller} and Bayraktar and Young \cite{bayraktar2011regularity}). A common denominator of those papers is that the controller uses so-called `classical' controls, i.e., progressively measurable maps $(t,\omega)\mapsto \alpha_t(\omega)$ that enter the drift and diffusion coefficient, $b$ and $\sigma$ of the controlled SDE in the form 
\[
\ud X^\alpha_t= b(X^\alpha_t,\alpha_t)\ud t+\sigma(X^\alpha_t,\alpha_t)\ud W_t.
\] 
From an analytical point of view, those games are connected to Hamilton-Jacobi-Bellman (HJB) equations with obstacle constraint but without gradient constraint, hence different from \eqref{eq:varineq}.

Much less attention instead has been devoted to the study of games in which the controller can adopt singular controls as in \eqref{eq:dyn}. The resulting dynamics is `singular' in the sense that the mapping $t\mapsto \nu_t(\omega)$ need not be absolutely continuous with respect to the Lebesgue measure (nor even continuous at all). A notable contribution to this strand of the literature was given by Hernandez-Hernandez et al.\ \cite{hernandez2015zero} who consider a zero-sum game in which $X^{[n,\nu]}$ is real-valued. They provide a general verification theorem and explicit optimal strategies for both players in some examples. They also show that the value function of the game need not be smooth if the stopping payoff is not continuously differentiable. The methods in \cite{hernandez2015zero} rely crucially on the one-dimensional set-up that allows to solve the variational problem with ordinary differential equations and an educated guess on the structure of the optimal strategies. Here instead we develop a general theory for multi-dimensional state-dynamics that leads us to consider PDEs. The introduction of \cite{hernandez2015zero} also explains numerous important applications of zero-sum games like the ones we consider here. Such applications include models for a central bank controlling exchange rates up to the time of a possible political veto and models for the control of inflation. We refer the interested reader to the Introduction of \cite{hernandez2015zero} for more insight on applications.

Our paper is organised as follows. We set out the notation in Section \ref{sec:pre}. Then, in Section \ref{sec:setting} we cast our problem, state our main assumptions and the main result (Theorem \ref{thm:usolvar}). In Section \ref{sec:penalised-a} we introduce and solve a penalised problem on bounded domain. In Section \ref{sec:penalised-b} we let the domain become unbounded and solve the corresponding penalised problem. Finally, in Section \ref{sec:final} we prove our main result. A technical appendix completes the paper.

\section{Notation}\label{sec:pre}

Fix $d,d'\in\N$, $d\le d'$ and $T\in(0,\infty)$. Given $u\in\R^d$ we let $|u|_d$ be its Euclidean norm. For vectors $u,v\in\R^d$ their scalar product is denoted by $\langle u,v \rangle$. Given a matrix $M\in\R^{d\times d'}$, with entries $M_{ij}$, $i=1,\ldots d$, $j=1,\ldots d'$, we denote its norm by
\begin{align*}
|M|_{d\times d'}\coloneqq \Big(\sum_{i=1}^d\sum_{j=1}^{d'}M_{ij}^2\Big)^{1/2}
\end{align*}
and, if $d=d'$, we let $\mathrm{tr}(M)\coloneqq \sum_{i=1}^d M_{ii}$. 
The $d$-dimensional {\em open} ball centred in $0$ with radius $m$ is denoted by $B_m$ and 
the state space in our problem is going to be
\[
\R^{d+1}_{0,T}:=[0,T]\times\R^d.
\]
Finally, given a bounded set $A$ we denote by $\overline{A}$ its closure. 

For a smooth function $f:\R^{d+1}_{0,T}\to \R$ we denote its partial derivatives by 
$\partial_t f$, $\partial_{x_i}f$, $\partial_{tx_j}f$, $\partial_{x_ix_j}f$, for $i,j=1,\ldots d$. We will also use $f_t=\partial_t f$, $f_{x_i}=\partial_{x_i}f$, $f_{tx_i}=\partial_{tx_i}f$ and $f_{x_i x_j}=\partial_{x_i x_j}f$ to simplify long expressions. By $\nabla f$ we intend the spatial gradient, i.e., $\nabla f=(\partial_{x_1}f,\ldots \partial_{x_d} f)$, and by $D^2 f$ the spatial Hessian matrix with entries $\partial_{x_i x_j}f$ for $i,j=1,\ldots d$.
As usual $C^\infty_c(\R^{d+1}_{0,T})$ is the space of functions with compact support and infinitely many continuous derivatives. Continuous functions on a domain $D$ are denoted by $C(D)$. For an open bounded set $\cO\subset\R^{d+1}_{0,T}$ we let $C^0(\overline{\cO})$ be the space of continuous functions $f:\overline\cO\to\R$ equipped with the supremum norm
\begin{align}\label{eq:supremum}
\|f\|_{C^0(\overline\cO)}\coloneqq \sup_{(t,x)\in\overline \cO}|f(t,x)|.
\end{align}
Analogously, $C^0(\R^{d+1}_{0,T})$ is the space of bounded and continuous functions $f:\R^{d+1}_{0,T}\to\R$ equipped with the norm $\|f\|_\infty\coloneqq \|f\|_{C^0(\R^{d+1}_{0,T})}$ as in \eqref{eq:supremum} but with $\overline\cO$ replaced by $\R^{d+1}_{0,T}$.  

For {bounded $\cO\subset\R^{d+1}_{0,T}$}, we consider the following function spaces:  
\begin{itemize}
\item $C^{0,1}(\overline\cO)$ be the class of continuous functions with $\partial_{x_i}f\in C(\overline \cO)$ for $i=1,\ldots, d$; 
\item $C^{1,2}(\overline\cO)$ be the class of continuous functions with $\partial_t f,\partial_{x_i}f,\partial_{x_i x_j}f\in C(\overline \cO)$ for $i,j=1,\ldots, d$; 
\item $C^{1,3}(\overline\cO)$ be the class of continuous functions with 
\[
\partial_t f,\,\partial_{x_i}f,\,\partial_{x_i x_j}f,\,\partial_{x_i x_j x_k}f,\, \partial_{t x_i}f\in C(\overline \cO)
\] 
for $i,j,k=1,\ldots, d$ (notice the mixed derivatives $\partial_{tx_i}f$).
\end{itemize}
The above definitions extend obviously to continuously differentiable functions on $\R^{d+1}_{0,T}$.

Let $\ud(z_1,z_2)= (|t-s|+|x-y|^2_d)^{\frac{1}{2}}$ be the parabolic distance between points $z_1= (t,x)$ and $z_2= (s,y)$ in $\R^{d+1}_{0,T}$. For a fixed $\alpha\in(0,1)$ and a continuous function $f:\overline\cO\to\R$ we set \cite[p.\ 61]{friedman2008partial}
\begin{align*}
\|f\|_{C^{\alpha}(\overline\cO)}\coloneqq &\,\|f\|_{C^0(\overline\cO)}+\sup_{\substack{z_1,z_2\in \overline\cO \\ z_1\neq z_2}}\frac{|f(z_1)-f(z_2)|}{\ud^\alpha(z_1,z_2)}.
\end{align*}
{We say that $f\in C^\alpha(\overline\cO)$ if $f\in C^0(\overline\cO)$ and $\|f\|_{C^\alpha(\overline\cO)}<\infty$}. We work with the following norms, defined for sufficiently smooth functions $f$: 
\begin{align*}
&\|f\|_{C^{0,1,\alpha}(\overline\cO)}\coloneqq \|f\|_{C^{\alpha}(\overline\cO)}+\sum_{i=1}^d\|\partial_{x_i}f\|_{C^{\alpha}(\overline\cO)};\\
&\|f\|_{C^{1,2,\alpha}(\overline\cO)}\coloneqq \|f\|_{C^{0,1,\alpha}(\overline\cO)}+\|\partial_t f\|_{C^{\alpha}(\overline\cO)}+\sum_{i,j=1}^d\|\partial_{x_ix_j}f\|_{C^{\alpha}(\overline\cO)};\\
&\|f\|_{C^{1,3,\alpha}(\overline\cO)}\coloneqq \|f\|_{C^{1,2,\alpha}(\overline\cO)}+\sum_{i=1}^d\|\partial_{tx_i}f\|_{C^{\alpha}(\overline\cO)}+\sum_{i,j,k=1}^d\|\partial_{x_i x_j x_k}f\|_{C^{\alpha}(\overline\cO)}.
\end{align*}
For $(j,k)\in\{(0,0);(0,1);(1,2);(1,3)\}$ and bounded $\cO$ let us define
\begin{align}
C^{j,k,\alpha}(\overline\cO)&\coloneqq \big\{ f\in C^{j,k}(\overline\cO)\big|\,\|f\|_{C^{j,k,\alpha}(\overline\cO)}<\infty\big\},\\
C^{j,k,\alpha}_{\ell oc}(\R^{d+1}_{0,T})&\coloneqq\{f\in C^{j,k}(\R^{d+1}_{0,T})\big|\,f\in C^{j,k,\alpha}(\overline\cO)\ \text{for all bounded $\cO\subset\R^{d+1}_{0,T}$}\}.
\end{align}
For $B$ and $B'$ open balls in $\R^d$ and $S\in[0,T)$, let $\cO_B\coloneqq[0,T)\times B$ and $\cO_{S,B'}\coloneqq[0,S)\times B'$. We denote $C^{j,k,\alpha}_{Loc}(\cO_B)$ the class of functions $f\in C(\overline\cO_B)$ such that $f\in C^{j,k,\alpha}(\overline \cO_{S,B'})$ for all $S<T$ and $B'$ such that $\overline{B'}\subset B$. Finally, we let
\[
C^{j,k,\alpha}_{Loc}(\R^{d+1}_{0,T})\coloneqq \big\{ f\in C(\R^{d+1}_{0,T})\,\big|\, f\in C^{j,k,\alpha}_{Loc}(\cO_B)\ \text{for all open balls $B\subset \R^d$}\big\}.
\]
Notice that, the derivatives of functions in $C^{j,k,\alpha}_{\ell oc}(\R^{d+1}_{0,T})$ are H\"older continuous on $\overline{\cO}_B$ for any ball $B\subset\R^d$. Instead, the derivatives of functions in $C^{j,k,\alpha}_{Loc}(\cO_B)$ need not be continuous along the parabolic boundary of $\cO_B$ and derivatives of functions in $C^{j,k,\alpha}_{Loc}(\R^{d+1}_{0,T})$ may be discontinuous at $T$. To simplify long formulae, sometimes we use the notations:
\begin{align}\label{eq:ngrad}
\|\nabla f\|_{C^0(\overline\cO)}:=\Big(\sum_{i=1}^d\|\partial_{x_i}f\|^2_{C^0(\overline \cO)}\Big)^{\frac{1}{2}}\quad\text{and}\quad \|D^2 f\|_{C^0(\overline\cO)}:=\Big(\sum_{i,j=1}^d\|\partial_{x_i x_j}f\|^2_{C^0(\overline \cO)}\Big)^{\frac{1}{2}}.
\end{align}

For $p\in[1,\infty]$ we say that $f\in L^p_{\ell oc}(\R^{d+1}_{0,T})$ if $f\in L^p(\cO)$ for any bounded $\cO\subset \R^{d+1}_{0,T}$. Denote by $\cD^{1,2}_{\cO}$ the class of functions $f\in L^p_{\ell oc}(\R^{d+1}_{0,T})$ whose partial derivatives $\partial_t f$, $\partial_{x_i} f$, $\partial_{x_i x_j} f$ exist in the weak sense on $\cO$, for $i,j=1,\ldots d$, and let
\begin{align*}
\|f\|_{W^{1,2,p}(\cO)}\coloneqq \,\|f\|_{L^p(\cO)}+\|\partial_t f\|_{L^p(\cO)}+\sum_{i=1}^d\|\partial_{x_i}f\|_{L^p(\cO)} +\sum_{i,j=1}^d\|\partial_{x_ix_j}f\|_{L^p(\cO)}.
\end{align*}
Then we define $W^{1,2,p}(\cO)\coloneqq \{ f\in\cD^{1,2}_{\cO}\,|\,\|f\|_{W^{1,2,p}(\cO)}<\infty \}$
and 
\begin{align}\label{eq:W12p}
W^{1,2,p}_{\ell oc}(\R^{d+1}_{0,T})&\,\coloneqq \big\{ f\in L^p_{\ell oc}(\R^{d+1}_{0,T}) \big|\,f\in W^{1,2,p}(\cO), \,\forall \cO\subseteq \R^{d+1}_{0,T}, \cO\text{ bounded}\big\}.
\end{align}
For $\alpha=1-\frac{d+2}{p}$ and $p>d+2$, and for any bounded $\cO\subset\R^{d+1}_{0,T}$, we have 
\begin{align}\label{eq:inclW12C01}
C^{1,2}(\overline{\cO})\subset W^{1,2,p}(\cO)\hookrightarrow C^{0,1,\alpha}(\overline{\cO}),
\end{align}
where the first inclusion is obvious and the second one is a compact Sobolev embedding (\cite[eq. (E.9)]{fleming2012deterministic} or \cite[Exercise 10.1.14]{krylov2008lectures}). 


\section{Setting and Main Results}\label{sec:setting}
Let $(\Omega,\mathcal{F},\P)$ be a complete probability space, $\F=(\mathcal{F}_s)_{s\in[0,\infty)}$ a right-continuous filtration completed by the $\P$-null sets and $(W_s)_{s\in[0,\infty)}$ a $\F$-adapted, $d'$-dimensional Brownian motion. Fix $T\in(0,\infty)$.
For $t\in[0,T]$, we denote 
\[
\cT_t:=\left\{\tau\,|\,\text{$\tau$ is $\F$-stopping time with $\tau\in[0,T-t]$, $\P$-a.s.}\right\}
\]
and we let $\mathcal{A}_t$ be the class of processes
\begin{align*}
\mathcal{A}_t\coloneqq \left\{(n,\nu)\left|
\begin{array}{l}
\text{$(n_s)_{s\in[0,\infty)}$ is progressively measurable, $\R^d$-valued,}\\ [+3pt]
\text{with $|n_s|_d=1$, $\P$-a.s.\ for all $s\in[0,\infty)$;}\\ [+3pt]
\text{$(\nu_s)_{s\in[0,\infty)}$ is $\F$-adapted, real-valued, non-decreasing and}\\ [+3pt] 
\text{right-continuous with $\nu_{0-}=0$, $\P$-a.s.,\ and $\E[|\nu_{T-t}|^2]<\infty$}
\end{array}
\right. \right\}.
\end{align*}
The notation $\nu_{0-}=0$ accounts for a possible jump of $\nu$ at time zero.
For a given pair $(n,\nu)\in\cA_t$ we consider the following (controlled) stochastic differential equation:
\begin{align}\label{eq:prcXcntrll}
X_s^{[n,\nu]}=x+\int_0^s b(X_u^{[n,\nu]})\ud u + \int_0^s \sigma(X_u^{[n,\nu]})\ud W_u +\int_{[0,s]} n_u\ud \nu_u, \quad 0\leq s\leq T-t,
\end{align}
where $b:\R^d\to\R^d$ and $\sigma:\R^{d}\to \R^{d\times d'}$ are continuous functions. For $\P$-a.e.\ $\omega$, the map $s\mapsto n_s(\omega)$ is Borel-measurable on $[0,T]$ and $s\mapsto \nu_s(\omega)$ defines a measure on $[0,T]$; thus the Lebesgue-Stieltjes integral $\int_{[0,s]}n_u(\omega)\ud \nu_u(\omega)$ is well-defined for $\P$-a.e.\ $\omega$. Under our Assumption \ref{ass:gen1} on $(b,\sigma)$ there is a unique $\F$-adapted solution of \eqref{eq:prcXcntrll} by, e.g., \cite[Thm.\ 2.5.7]{krylov1980controlled}. We denote 
\[
\P_x\big(\,\cdot\,\big)=\P\big(\,\cdot\,\big|X^{[n,\nu]}_{0-}=x\big)\quad\text{and}\quad\E_x\big[\,\cdot\,\big]=\E\big[\,\cdot\,\big|X^{[n,\nu]}_{0-}=x\big].
\] 

We study a class of 2-player zero-sum games (ZSGs) between a (singular) controller and a stopper. The stopper picks a stopping time $\tau\in\cT_t$ and the controller chooses a pair $(n,\nu)\in\cA_t$. At time $\tau$ the game ends and the controller pays to the stopper a random payoff depending on $\tau$ and on the path of $X^{[n,\nu]}$ up to time $\tau$. 
Given continuous functions $f,g,h:\R^{d+1}_{0,T}\to [0,\infty)$, a fixed discount rate $r\ge 0$ and $(t,x)\in\R^{d+1}_{0,T}$, the game's {\em expected} payoff reads
\begin{align}\label{eq:payoff}
\cJ_{t,x}(n,\nu,\tau)= \E_{x}\bigg[e^{-r\tau}g(t\!+\!\tau,X_\tau^{[n,\nu]})\!+\!\int_0^{\tau}\!\! e^{-rs}h(t\!+\!s,X_s^{[n,\nu]})\,\ud s\!+\!\int_{[0,\tau]}\!\! e^{-rs}f(t\!+\!s,X_{s}^{[n,\nu]})\circ\!\ud \nu_s \bigg],
\end{align}
where
\begin{align}\label{eq:zoudefnint}
\int_{[0,\tau]}\! e^{-rs}f(t\!+\!s,X_s^{[n,\nu]})\circ\!\ud \nu_s\coloneqq &\int_0^{\tau}\!e^{-rs}f(t\!+\!s,X_s^{[n,\nu]})\,\ud \nu_s^c\\
&+\!\sum_{0\leq s\leq\tau}e^{-rs}\!\int_{0}^{\Delta\nu_{s}}f(t\!+\!s,X_{s-}^{[n,\nu]}+\lambda n_s)\,\ud \lambda.\notag
\end{align}
Here $\nu^c$ is the continuous part of the process $\nu$ in the decomposition $\nu_s=\nu^c_s+\sum_{u\le s}\Delta\nu_u$, with $\Delta\nu_u=\nu_u-\nu_{u-}$.
If $f(t,x)=f(t)$ the integral \eqref{eq:zoudefnint} reduces to the standard Lebesgue-Stieltjes integral $\int_{[0,\tau]}f(s)\ud \nu_s$. In general, \eqref{eq:zoudefnint} gives a cost of exerting control that is consistent with the gradient constraint appearing in Hamilton-Jacobi-Bellman (HJB) equations for singular stochastic control (see, e.g., \cite{zhu1992generalized}). It can be interpreted as the limit of standard Lebesgue-Stieltjes integrals if each jump of the control $\nu$ is approximated by infinitesimally small sequential jumps (\cite[Cor.\ 1]{alvarez2000singular}).

The game admits {\em lower} and {\em upper} value, defined respectively by 
\begin{align}\label{eq:lowuppvfnc}
\underline{v}(t,x)\coloneqq \sup_{\tau\in \mathcal{T}_t}\inf_{(n,\nu)\in \mathcal{A}_t} \cJ_{t,x}(n,\nu,\tau)\quad\text{and}\quad
\overline{v}(t,x)\coloneqq \inf_{(n,\nu)\in \mathcal{A}_t}\sup_{\tau\in \mathcal{T}_t} \cJ_{t,x}(n,\nu,\tau),
\end{align}
so that $\underline{v}(t,x) \leq \overline{v}(t,x)$. If equality holds then we say that our game admits a value 
\begin{align}\label{eq:valfunc}
v(t,x)\coloneqq \underline{v}(t,x) = \overline{v}(t,x).
\end{align}

For $a(x)\coloneqq (\sigma\sigma^\top)(x)\in\R^{d\times d}$, the infinitesimal generator of $X^{[e_1,0]}$ (with $e_1$ the unit vector with 1 in the first entry) reads
\begin{align*}
(\mathcal{L}\varphi)(x)=\tfrac{1}{2}\mathrm{tr}\left(a(x)D^2\varphi(x)\right)+\langle b(x),\nabla \varphi(x)\rangle, \ \text{for any $\varphi\in C^{2}(\R^d)$. }
\end{align*}
By density arguments the linear operator $\cL$ admits a unique extension $\bar \cL$ to $W^{2,p}_{\ell oc}(\R^d)$ and, with a slight abuse of notation, we set $\bar\cL=\cL$.

A heuristic use of the dynamic programming principle, suggests that the value of the game $v$ should be solution of a free boundary problem of the following form:
\begin{problem}\label{prb:varineq}
Fix $p>d+2$. Find a function $u\in W^{1,2,p}_{\ell oc}(\R^{d+1}_{0,T})$ such that, letting
\begin{align*}
\mathcal{I}\coloneqq \big\{(t,x)\in\R^{d+1}_{0,T}\big|\:|\nabla u(t,x)|_d<f(t,x)\big\}\quad\text{and}\quad\mathcal{C}\coloneqq \big\{(t,x)\in\R^{d+1}_{0,T}\big|\:u(t,x)>g(t,x)\big\},
\end{align*}
$u$ satisfies:  
\begin{align}\label{eq:inipde}
\begin{cases}
(\partial_tu +\mathcal{L}u-ru)(t,x)=-h(t,x), &\qquad \text{{for all}}\ (t,x)\in\mathcal{C}\cap\mathcal{I}; \\
(\partial_tu +\mathcal{L}u-ru)(t,x)\geq -h(t,x), &\qquad \text{for a.e.}\ (t,x)\in\mathcal{C}; \\
(\partial_tu +\mathcal{L}u-ru)(t,x)\leq -h(t,x), &\qquad \text{for a.e.}\ (t,x)\in \mathcal{I}; \\
u(t,x)\geq g(t,x),&\qquad \text{for all}\ (t,x)\in \R^{d+1}_{0,T};\\
|\nabla u(t,x)|_d\leq f(t,x), &\qquad \text{for all}\ (t,x)\in\R^{d+1}_{0,T}; \\
u(T,x)=g(T,x), &\qquad \text{for all}\ x\in\R^d,
\end{cases}
\end{align}
with $|u(t,x)|\leq c(1+|x|_d^2)$ for all $(t,x)\in\R^{d+1}_{0,T}$ and a suitable $c>0$.\hfill$\blacksquare$
\end{problem}
Notice that the conditions $u\geq g$ and $|\nabla u|_d\leq f$ hold for all $(t,x)$ because of the embedding \eqref{eq:inclW12C01}. Thus, the two sets $\cI$ and $\cC$ are open in $[0,T)\times\R^d$.
It is not hard to check that a function $u\in W^{1,2,p}_{\ell oc}(\R^{d+1}_{0,T})$ solves Problem \ref{prb:varineq} if and only if it solves the variational inequalities in \eqref{eq:varineq} a.e.\ on $\R^{d+1}_{0,T}$ with quadratic growth. Next we give assumptions under which we obtain our main result (Theorem \ref{thm:usolvar}). 
\begin{assumption}[Controlled SDE]\label{ass:gen1}
The functions $b$ and $\sigma$ are continuously differentiable and locally Lipschitz on $\R^d$.
Moreover, there is $D_1>0$ such that
\begin{align}
|b(x)|_d+|\sigma(x)|_{d\times d'}\leq D_1(1+|x|_d), \text{ for all }x\in\R^d.
\end{align}
Recalling $a=\sigma\sigma^\top$, for any bounded set $B\subset\R^{d}$ there is $\theta_B>0$ such that $a(\cdot)$ is locally elliptic:
\begin{align}\label{eq:defnthetaEC}
\langle\zeta,a(x)\zeta\rangle\geq\theta_B |\zeta |_d^2\,\qquad\text{for any $\zeta\in\R^d$ and all $x\in \overline{B}$}.
\end{align}
\end{assumption}

\begin{assumption}[Functions $f$, $g$, $h$]\label{ass:gen2}
For the functions $f,g,h:\R^{d+1}_{0,T}\to[0,\infty)$ the following hold:
\begin{enumerate}
\item[  (i)] $f^2,g \in C^{1,2,\alpha}_{\ell oc}(\R^{d+1}_{0,T})$ and $h\in C^{0,1,\alpha}_{\ell oc}(\R^{d+1}_{0,T})$ for some $\alpha\in(0,1)$; 
\item[ (ii)] $t\mapsto f(t,x)$ is non-increasing for each $x\in\R^d$ and $0\le f(t,x)\le c(1+|x|_d^p)$ for some $c,p>0$;
\item[(iii)] there is $K_0\in(0,\infty)$ such that for all $0\le s<t\le T$ and all $x\in\R^{d+1}_{0,T}$
\begin{align}\label{eq:lipstimehg}
h(t,x)-h(s,x)\leq K_0(t-s)\quad\text{and}\quad g(t,x)-g(s,x)\leq K_0(t-s);
\end{align}
\item[(iv)] there is $K_1\in(0,\infty)$ such that
\begin{align}\label{eq:fghpolgrow}
0\le g(t,x)+h(t,x) \leq K_1(1+|x|_d^2), \quad \text{for $(t,x)\in\R^{d+1}_{0,T}$};
\end{align}
\item[(v)] $f$ and $g$ are such that
\begin{align}\label{eq:grdg<f}
|\nabla g(t,x)|_d\leq f(t,x),\quad \text{for all $(t,x)\in\R^{d+1}_{0,T}$};
\end{align}
\item[(vi)] there is $K_2\in(0,\infty)$ such that
\begin{align}\label{defn:infTheta}
\Theta(t,x)\coloneqq \big(h+\partial_tg+\mathcal{L}g-rg\big)(t,x)\ge -K_2,\quad\text{for all $(t,x)\in\R^{d+1}_{0,T}$}.
\end{align}
\end{enumerate}
\end{assumption}

Condition \eqref{eq:lipstimehg} is immediately satisfied if $h$ and $g$ are time-homogeneous (as it is often the case in investment problems, see, e.g., \cite{chiarolla1998optimal}, \cite{chiarolla2009irreversible}, \cite{deangelis2014investment})  or if the maps $t\mapsto \big(h(t,x),g(t,x)\big)$ are non-increasing for all $x\in\R^d$. Otherwise, that condition amounts to setting a maximum growth rate on $t\mapsto\big( h(t,x),g(t,x)\big)$ as time increases. 
Condition \eqref{defn:infTheta} guarantees that there is no region in the state space where the controller (minimiser) can push the process and obtain arbitrarily large (negative) running gains.
Condition \eqref{eq:grdg<f} is sufficient to guarantee that the stopping payoff satisfies the gradient constraint in \eqref{eq:inipde} and therefore, from a probabilistic point of view, the stopper can stop at any point in the state-space. On a more technical level, we use \eqref{eq:grdg<f} to obtain crucial bounds on the spatial gradient of the solution of a penalised problem introduced below (cf.\ proof of Lemma \ref{thm:gradpenbnd}) and on its time-derivative (cf.\ proof of Lemma \ref{lem:bndtimder}, where we use the lower bound on the Hamiltonian \eqref{rem:trickhamg} implied by \eqref{eq:grdg<f}). Finally, it will be shown in Lemma \ref{lem:jumptau} that \eqref{eq:grdg<f} implies that the controller should never exert a jump at the time the stopper ends the game. Here we are not concerned with the construction of equilibria and Lemma \ref{lem:jumptau} is not needed for the rest of our analysis. However, we notice that the lemma points to potential difficulties in the construction of equilibria without condition \eqref{eq:grdg<f}. Intuitively, when \eqref{eq:grdg<f} is violated, the controller could make arbitrarily large (negative) gains by performing a large jump at the time the stopper decides to end the game. This can be seen by choosing $n_\tau=-\nabla g/|\nabla g|_d$ in the third line of \eqref{eq:lemjump}, so that the controller gains $|\nabla g|_d-f>0$ per unit of control exerted. 

The next theorem is the main result of the paper and its proof is distilled in the following sections through a number of technical results and estimates.
\begin{theorem}\label{thm:usolvar}
The game described above admits a value (i.e., \eqref{eq:valfunc} holds) and the value function $v$ of the game is the maximal solution to Problem \ref{prb:varineq}. Moreover, for any given $(t,x)\in\R^{d+1}_{0,T}$ and any admissible control $(n,\nu)\in\cA_t$, the stopping time $\tau_*=\tau_*(t,x;n,\nu)\in\cT_t$ defined under $\P_x$ as
\begin{align}\label{eq:taustar}
\tau_*\coloneqq \inf\big\{s\geq 0\,\big|\, v(t+s,X_s^{[n,\nu]})=g(t+s,X_s^{[n,\nu]})\big\}\wedge(T-t),
\end{align}
is optimal for the stopper. That is
\[
v(t,x)=\inf_{(n,\nu)\in\cA_t}\cJ_{t,x}(n,\nu,\tau_*(n,\nu)).
\]
\end{theorem}
\begin{remark}
The optimal stopping rule $\tau_*$ is in so-called {\em feedback form}, because it only depends on the sample paths of the controlled dynamics $X^{[n,\nu]}$. This means that it characterises a best response against {\em any} choice of the control $(n,\nu)$ available to the controller. 
\end{remark}

\begin{remark}
Uniqueness of the solution to Problem \ref{prb:varineq} remains an open question. Methods used in, e.g., \cite{soner1991free}, do not apply due to the presence of obstacle {\em and} gradient constraints. Existence of an optimal control pair $(n^*,\nu^*)\in\cA_t$ is also subtle and cannot be addressed in the generality of our setting. Even in standard singular control problems (not games) abstract existence results rely on compactness arguments in the space of increasing processes under more stringent assumptions (e.g., convexity/concavity or independence from the state dynamics) on the functions $f$, $g$, $h$, $b$ and $\sigma$ (see, e.g., \cite{budhiraja2006optimal}, \cite{haussmann1995singular}, \cite{li2017existence}, \cite{taksar1992skorohod}). 
\end{remark}


\section{Penalised Problem and A Priori Estimates}\label{sec:penalised-a}

In this section we first introduce a class of penalised problems and illustrate their connection with a class of ZSGs of control (Section \ref{sec:pen-p}). Then we provide important {\em a priori} estimates on the growth and gradient of the solution of such penalised problems (Sections \ref{sec:growth} and \ref{sec:grd}) and, finally, we prove existence and uniqueness of the solution (Section \ref{sec:sol-pen}). 

\subsection{A penalised problem}\label{sec:pen-p}
For technical reasons related to solvability of the penalised problem and the probabilistic representation of its solution we choose to work on a sequence of bounded domains $(\cO_m)_{m\in\N}\subset \R^{d+1}_{0,T}$. Recall that $B_m\subset\R^d$ is the open ball of radius $m$ centred in the origin and set $\cO_m\coloneqq [0,T)\times B_m$ with parabolic boundary 
$\partial_P\cO_m= ([0,T)\times\partial B_m)\cup( \{T\}\times \overline{B}_m)$.

Let $(\xi_{m})_{m\in\N}\subset C^\infty_c(\R^d)$ be such that for each $m\in\N$ we have:
\begin{itemize}
\item[(i)] $0\leq \xi_{m}\leq 1$ on $\R^d$, with $\xi_m=1$ on $B_m$ and $\xi_m=0$ on $\R^{d}\setminus B_{m+1}$;
\item[(ii)] there is $C_0>0$ independent of $m\in\N$ such that
\begin{align}\label{rem:cutoffXI}
\text{ $|\nabla \xi_m|_d^2\leq C_0\xi_m$ on $\R^d$}.
\end{align}
\end{itemize}
An example of such functions is provided in Appendix \ref{app:xi} for completeness.
We define
\[
g_m(t,x)\coloneqq \xi_{m-1}(x)g(t,x)\quad\text{and}\quad h_m(t,x)\coloneqq \xi_{m-1}(x)h(t,x),\quad \text{for $(t,x)\in\R^{d+1}_{0,T}$}.
\] 

Clearly $g_m=h_m=0$ on $\R^{d+1}_{0,T}\setminus \cO_m$ while $g_m=g$ and $h_m=h$ on $\cO_{m-1}$. We also define a version $f_m$ of the function $f$ so that $f_m = f$ on $\cO_{m-1}$ and the condition 
\begin{align}\label{eq:grdgm<fm}
|\nabla g_m(t,x)|_d\leq f_m(t,x),\quad \text{for } (t,x)\in \overline{\cO}_m,
\end{align}
is preserved. For $(t,x)\in \R^{d+1}_{0,T}$ we let
\begin{align}\label{eq:def-fm}
f_m(t,x)\coloneqq&\, \Big(f^2(t,x)+\|g\|_{C^0(\overline \cO_m)}^2|\nabla \xi_{m-1}(x)|_d^2+2\big(g\xi_{m-1}\langle \nabla \xi_{m-1},\nabla g\rangle\big)(t,x)\Big)^{\frac{1}{2}}
\end{align}
and notice that on $ \overline{\cO}_{m}$
\begin{align*}
|\nabla g_m|^2_d=\sum_{i=1}^d\big(\xi_{m-1}\partial_{x_i}g+g\partial_{x_i}\xi_{m-1}\big)^2
=&\,\xi^2_{m-1}|\nabla g|^2_d\!+\!g^2|\nabla\xi_{m-1}|^2_d\!+\!2\xi_{m-1}g\langle\nabla \xi_{m-1},\nabla g\rangle\le f^2_m,
\end{align*}
where the inequality follows by the assumption $|\nabla g|_{d}\le f$ and $|\xi_m|\le 1$. 
Since $\nabla\xi_{m-1}=\bf 0$ on $B_{m-1}$, we have $f_m=f$ on $\cO_{m-1}$. Notice that $f^2_m\in C^{0,1,\alpha}(\overline{\cO}_m)$ by Assumption \ref{ass:gen2} but it does not vanish on the boundary of $\cO_m$. 
By construction, it is clear that 
\[
g_m\to g,\quad h_m\to h \quad\text{and}\quad f_m\to f,\quad \text{as $m\to\infty$},
\]
uniformly on any compact $\cK\subset\R^{d+1}_{0,T}$. 

Let us now state the penalised problem. Fix $(\eps,\delta)\in (0,1)^2$ and $m\in\N$. Let $\psi_\varepsilon\in C^2(\R)$ be a non-negative, convex function such that $\psi_\varepsilon(y)=0$ for $y\leq0$, $\psi_\varepsilon(y)>0$ for $y>0$, $\psi_\varepsilon'\geq0$ and $\psi_\varepsilon(y)=\frac{y-\varepsilon}{\varepsilon}$ for $y\geq 2\varepsilon$. We also denote $(y)^+\coloneqq \max\{0,y\}$ for $y\in\R$.
\begin{problem}\label{prb:penprob}
Find $u=u^{\eps,\delta}_m$ with $u\in C^{1,2,\alpha}(\overline{\cO}_m)$, for $\alpha\in(0,1)$ as in Assumption \ref{ass:gen2}, that solves:
\begin{align}\label{eq:penprob}
\begin{cases}\partial_tu+\mathcal{L}u-ru=-h_m-\frac{1}{\delta}\left( g_m-u\right)^++\psi_\varepsilon\left(|\nabla u|_d^2- f_m^2\right), &\text{on } \cO_m, \\
u(t,x)=g_m(t,x), & (t,x)\in\partial_P \cO_m.
\end{cases}
\end{align}
\hfill$\blacksquare$
\end{problem}
There are two useful probabilistic interpretations of a solution to Problem \ref{prb:penprob}, which we are going to illustrate next. Given $t\in[0,T]$, define the control classes 
\begin{align}\label{eq:cA-pen}
\cA^\circ_t
\coloneqq 
\left\{ 
(n,\nu)\left|
\begin{array}{l}
\text{$(n_s)_{s\in[0,\infty)}$ is progressively measurable, $\R^d$-valued},\\ [+3pt]
\text{with $|n_s|_d=1$, $\P$-a.s.\ for all $s\in[0,\infty)$}; \\ [+3pt]
\text{$(\nu_s)_{s\in[0,\infty)}$ is $\F$-adapted, real-valued, non-decreasing and}\\ [+3pt]
\text{absolutely continuous in time, $\P$-a.s., with $\E[|\nu_{T-t}|^2] <\infty$}
\end{array}
\right.
\right\},
\end{align}
and
\begin{align*}
\mathcal{T}^\delta_t\coloneqq 
\left\{
w\left|
\begin{array}{l}
\text{$(w_s)_{s\in[0,\infty)}$ is progressively measurable}, \\ [+3pt]
\text{with $0\leq w_s\leq\frac{1}{\delta}$, $\P$-a.s.\ for all $s\in[0,T-t]$}
\end{array}
\right.
\right\}.
\end{align*}
It is obvious that $\cA^\circ_t\subset \cA_t$. For $(t,x)\in\R^{d+1}_{0,T}$ and $y\in\R^d$ we define the Hamiltonian 
\begin{align}\label{eq:hmltn}
H^{\varepsilon}_m(t,x,y)\coloneqq \sup_{p\in\mathbb{R}^d}\left\{\langle y, p\rangle-\psi_\varepsilon\left(|p|^2_d- f_m^2(t,x)\right)\right\}.
\end{align}
The function $H^{\varepsilon}_m$ is non-negative (pick $p={\bf 0}$). Thanks to \eqref{eq:grdgm<fm}, choosing $p=-\nabla g_m(t,x)$ we have 
\begin{align}\label{rem:trickhamg}
H^{\varepsilon}_m(t,x,y)\geq -\langle y,\nabla g_m(t,x)\rangle,\quad\text{for all $(t,x,y)\in[0,T]\times\R^d\times\R^d$}.
\end{align}

For any admissible pair $(n,\nu)\in\cA^\circ_t$, we consider the controlled dynamics
\begin{align*}
 X_s^{[n,\nu]}=x+\int_0^s \big[b(X_u^{[n,\nu]})+ n_u\dot{\nu}_u\big]\ud u + \int_0^s \sigma(X_u^{[n,\nu]})\ud W_u, \qquad \text{for $0\leq s\leq T-t$},
\end{align*}
and define $\rho_{m} =\rho_{\cO_m} (t,x;n,\nu)$ as 
\begin{align}\label{eq:rom}
\rho_{m} = \inf\{s\geq0\,|\,X_{s}^{[n,\nu]}\notin B_m\}\wedge (T-t).
\end{align}
Instead, for $w\in \mathcal{T}^\delta_t$, we introduce a controlled discount factor $R^w_s\coloneqq \exp( -\int_0^s \left[r+w_\lambda\right] \ud\lambda)$.

For $(t,x)\in\overline\cO_m$ and a triple $[(n,\nu),w]\in\cA^\circ_t\times\cT^\delta_t$ let us consider an expected payoff:
\begin{align}\label{eq:Jpen}
&\cJ^{\varepsilon,\delta,m}_{t,x}(n,\nu,w)\\
&= \E_{x}\biggr[\int_0^{\rho_{m}}\!\!R^w_s\big[h_m(\cdot)\!+\!w_sg_m(\cdot)\!+\! H^{\varepsilon}_m(\cdot,n_s\dot{\nu}_s)\big](t\!+\!s,X_s^{[n,\nu]})\,\ud s\!+\!R^w_{\rho_{m}}g_m(t\!+\!\rho_{m} ,X_{\rho_{m}}^{[n,\nu]})\biggr].\notag
\end{align}
The associated upper and lower value read, respectively, 
\begin{align*}
\overline{v}^{\varepsilon,\delta}_m (t,x)= \inf_{(n,\nu)\in \cA^{\circ}_t}\sup_{w\in \mathcal{T}^\delta_t} \cJ^{\varepsilon,\delta,m}_{t,x}(n,\nu,w)\quad \text{and}\quad
\underline{v}^{\varepsilon,\delta}_m (t,x)= \sup_{w\in \mathcal{T}^\delta_t}\inf_{(n,\nu)\in \cA^{\circ}_t} \cJ^{\varepsilon,\delta,m}_{t,x}(n,\nu,w), 
\end{align*}
so that $\underline{v}^{\varepsilon,\delta}_m \leq \overline{v}^{\varepsilon,\delta}_m $. A solution of Problem \ref{prb:penprob} coincides with the value function of this ZSG.
\begin{proposition}\label{prp:probrap1}
Let $u^{\varepsilon,\delta}_m $ be a solution of Problem \ref{prb:penprob}. Then 
\begin{align}\label{eq:probrap}
u^{\varepsilon,\delta}_m (t,x)=\overline{v}^{\varepsilon,\delta}_m (t,x)=\underline{v}^{\varepsilon,\delta}_m (t,x),\quad\text{for all $(t,x)\in\overline{\cO}_m$}.
\end{align}
\end{proposition}
\begin{proof}
For simplicity denote $u=u^{\eps,\delta}_m$. By definition $\overline{v}^{\eps,\delta}_m=\underline{v}^{\eps,\delta}_m=u=g_m$ on $\partial_P\cO_m$. Fix $(t,x)\in\cO_m$ and an arbitrary triple $[(n,\nu),w]\in\cA^\circ_t\times\cT^\delta_t$. Applying Dynkin's formula to 
$R^w_{\rho_m}u(t+\rho_{m},X^{[n,\nu]}_{\rho_{m}})$ and using that $u$ solves \eqref{eq:penprob} we have
\begin{align}\label{eq:repensol}
u(t,x)=\E_x\biggr[&R^w_{\rho_m}g_m(t+\rho_m ,X_{\rho_m}^{[n,\nu]})\!+\!\int_0^{\rho_m}\!R^w_sh_m(t\!+\!s,X_s^{[n,\nu]})\,\ud s\\
&+\int_0^{\rho_{m}}\!R^w_s\big[\tfrac{1}{\delta}\left(g_m-u\right)^+\!-\!\psi_{\varepsilon}\left(|\nabla u|_d^2- f_m^2\right)\big](t+s,X_s^{[n,\nu]})\,\ud s\notag \\
&+\int_0^{\rho_{m}}\!R^w_s\big[w_s u(t+s,X_s^{[n,\nu]})\!-\!\big\langle n_s\dot{\nu}_s,\nabla u(t+s,X_s^{[n,\nu]})\big\rangle \big]\,\ud s\biggl]. \notag
\end{align}
By definition of the Hamiltonian we have (choosing $p=-\nabla u(t+s,X_s^{[n,\nu]})$ in \eqref{eq:hmltn})
\begin{align}\label{eq:Ham}
-\big\langle n_s\dot{\nu}_s,\nabla u(t+s,X_s^{[n,\nu]})\big\rangle-\psi_{\varepsilon}\left(|\nabla u|_d^2- f_m^2\right)(t+s,X_s^{[n,\nu]})\le H^{\eps}_m(t+s,X_s^{[n,\nu]},n_s\dot{\nu}_s).
\end{align}
Moreover, choosing $w=w^*\in \cT^{\delta}_t$ defined as
\begin{align}\label{eq:dfnoptw*}
w^*_s\coloneqq \begin{cases} 0 &\text{if }u(t+s,X_s^{[n,\nu]})>g_m(t+s,X_s^{[n,\nu]}),\\
\frac{1}{\delta}&\text{if }u(t+s,X_s^{[n,\nu]})\leq g_m(t+s,X_s^{[n,\nu]}),
\end{cases}
\end{align}
we also have
\begin{align}\label{eq:w*}
\tfrac{1}{\delta}\big(g_m-u\big)^+(t+s,X_s^{[n,\nu]})+w_s^*u(t+s,X_s^{[n,\nu]}) = w^*_s g_m(t+s,X_s^{[n,\nu]}).
\end{align}
Then, plugging \eqref{eq:Ham} and \eqref{eq:w*} into \eqref{eq:repensol} we arrive at $u(t,x)\le \cJ_{t,x}^{\varepsilon,\delta,m}(n,\nu,w^*)$. Since the pair $(n,\nu)$ was arbitrary, then we have 
\begin{align*}
u(t,x)\leq \inf_{(n,\nu)\in\cA^\circ_t}\cJ_{t,x}^{\varepsilon,\delta,m}(n,\nu,w^*),
\end{align*}
and therefore $u(t,x)\le \underline{v}^{\varepsilon,\delta}_m (t,x)$.
Next we are going to prove that $u\geq \overline{v}^{\varepsilon,\delta}_m$.

For any $w\in\cT^\delta_t$ it is not hard to see that 
\begin{align}\label{eq:ridofpnlobst}
\tfrac{1}{\delta}\big(g_m-u\big)^+(t+s,X_s^{[n,\nu]})+w_s u(t+s,X_s^{[n,\nu]})\geq w_s g_m(t+s,X_s^{[n,\nu]}),
\end{align}
since $0\le w_s\leq\frac{1}{\delta}$ for all $s\in[0,T-t]$. Assume, it is possible to find a pair $(n^*,\nu^*)\in\cA^\circ_t$ such that, putting $X^*=X^{[n^*,\nu^*]}$ we obtain
\begin{align}\label{eq:Ham*}
-\big\langle n^*_s\dot{\nu}^*_s,\nabla u(t+s,X_s^{*})\big\rangle-\psi_{\varepsilon}\left(|\nabla u|_d^2- f_m^2\right)(t+s,X_s^{*})= H^{\eps}_m(t+s,X_s^{*},n^*_s\dot{\nu}^*_s).
\end{align}
Then, plugging \eqref{eq:ridofpnlobst} and \eqref{eq:Ham*} into \eqref{eq:repensol}, we have $u(t,x)\ge \cJ^{\eps,\delta,m}_{t,x}(n^*,\nu^*,w)$. Since $w\in\cT^\delta_t$ is arbitrary we have
\[
u(t,x)\ge \sup_{w\in\cT^\delta_t}\cJ^{\eps,\delta,m}_{t,x}(n^*,\nu^*,w)\ge \overline v^{\eps,\delta}_m(t,x).
\]

It remains to find the pair $(n^*,\nu^*)$.
By concavity, the supremum in $H^{\eps}_m(t,x,y)$ is uniquely attained at a point $p=p(t,x,y)\in\R^d$ identified by the first-order condition
$y=2\psi_{\varepsilon}'(|p|^2_d- f_m^2(t,x))p$.
Taking 
\begin{equation}\label{eq:optcntr}
\begin{split}
n_s^*\coloneqq &\!
\begin{cases} -\frac{\nabla u^{\eps,\delta}_m(t+s,X_s^*)}{|\nabla u^{\eps,\delta}_m(t+s,X_s^*)|_d}, &\quad\text{if } \nabla u^{\eps,\delta}_m(t+s,X_s^*)\neq \mathbf{0},\\[+5pt] 
{\text{any unit vector}}, &\quad\text{if }\nabla u^{\eps,\delta}_m(t+s,X_s^*)=\mathbf{0},
\end{cases}\\
\dot{\nu}_s^*\coloneqq &\, 2\psi_{\varepsilon}'\left(|\nabla u^{\eps,\delta}_m(t+s,X_s^*)|_d^2- f_m^2(t+s,X_s^*)\right)|\nabla u^{\eps,\delta}_m(t+s,X_s^*)|_d,
\end{split}
\end{equation}
with
\begin{align}\label{eq:X*pen}
X^*_{s\wedge\rho_{m}}\!=\!x\!+\!\int_0^{s\wedge\rho_{m}}\!\!\big[b(X_\lambda^*)\!-\! 2\big(\psi_{\varepsilon}'(|\nabla u^{\eps,\delta}_m|_d^2\!-\!f_m^2)\nabla u^{\eps,\delta}_m\big)(t\!+\!\lambda,X_\lambda^*)\big]\ud \lambda\!+\!\int_0^{s\wedge\rho_{m}}\!\!\!\sigma(X_\lambda^*)\ud W_\lambda, 
\end{align}
for $s\in[0,T-t]$, we have that \eqref{eq:Ham*} holds (we restored the notation $u^{\eps,\delta}_m$ for future reference). It remains to check that $(X^*_{s\wedge \rho_{m}})_{s\in[0,T]}$ is actually well-defined. 

Since $u^{\eps,\delta}_m\in C^{1,2,\alpha}(\overline{\cO}_m)$ and $\psi_\eps\in C^2(\R)$, then both the drift and diffusion coefficients of the controlled SDE are Lipschitz in space (recall Assumption \ref{ass:gen1}). Hence \eqref{eq:X*pen} admits a unique strong solution. Moreover, both $n^*$ and $\nu^*$ are progressively measurable and since $0\le \psi'_\eps\le 1/\eps$ then also
$0\le \dot{\nu}_s^*\leq2\eps^{-1} \|\nabla u^{\eps,\delta}_m\|_{C^0(\overline{\cO}_m)}<\infty$.
Therefore $(n^*,\nu^*)\in\cA^\circ_t$.
\end{proof}
\begin{remark}\label{rem:unique}
From the proof of Proposition \ref{prp:probrap1} we see that $w^*$ defined in \eqref{eq:dfnoptw*} is optimal for the maximiser and $(n^*,\nu^*)$ defined in \eqref{eq:optcntr} is optimal for the minimiser in the ZSG with payoff \eqref{eq:Jpen}. 
\end{remark}

We also show that $u^{\eps,\delta}_m$ is the value function of a control problem with a recursive structure.
\begin{proposition}\label{cor:prbrap2edm}
Let  $u^{\varepsilon,\delta}_m$ be a solution of Problem \ref{prb:penprob}. Then, for $(t,x)\in\overline{\cO}_m$, 
\begin{align}\label{eq:prbrap2d}
u^{\varepsilon,\delta}_m (t,x)=\inf_{(n,\nu)\in \cA^\circ_t}\E_{x}\biggr[&e^{-(r+\delta^{-1})\rho_{m}} g_m(t\!+\!\rho_{m} ,X_{\rho_{m}}^{[n,\nu]})\notag\\
&\!+\!\int_0^{\rho_{m}}\!\!e^{-(r+\delta^{-1})s}\big[ h_m\!+\!\tfrac{1}{\delta}(g_m\vee u^{\varepsilon,\delta}_{m})+H^{\varepsilon}_m(\cdot,n_s\dot{\nu}_s)\big](t\!+\!s,X_s^{[n,\nu]})\,\ud s\biggr],
\end{align}
and the pair $(n^*,\nu^*)$ from \eqref{eq:optcntr} is optimal.
\end{proposition}
\begin{proof}
For simplicity denote $u=u^{\eps,\delta}_m$. Since $\frac{1}{\delta}(g_m-u)^++\frac{1}{\delta}u=\frac{1}{\delta} g_m\vee u$, then taking $w\equiv\frac{1}{\delta}$ in \eqref{eq:repensol} and using \eqref{eq:Ham} we get
\begin{align*}
u(t,x)\leq\inf_{(n,\nu)\in\cA^{\circ}_t}\E_x\biggr[&\,e^{-(r+\delta^{-1})\rho_{m}}g_m(t\!+\!\rho_{m} ,X_{\rho_{m}}^{[n,\nu]})\\
&+\!\int_0^{\rho_{m}}\!e^{-(r+\delta^{-1})s}\big[h_m+\tfrac{1}{\delta}g_m\vee u+H^{\varepsilon}_m(\cdot,n_s\dot{\nu}_s)\big](t\!+\!s,X_s^{[n,\nu]})\,\ud s\biggr].
\end{align*}
The equality is obtained by substituting in \eqref{eq:repensol} the controls $w\equiv\frac{1}{\delta}$ and $(n^*,\nu^*)$ defined in \eqref{eq:optcntr}. Recalling the notation $X^*=X^{[n^*,\nu^*]}$ and \eqref{eq:Ham*} we obtain \eqref{eq:prbrap2d} and optimality of $(n^*,\nu^*)$.
\end{proof}
From the probabilistic representation of $u^{\eps,\delta}_m$ in \eqref{eq:probrap} we establish uniqueness in Problem \ref{prb:penprob}.
\begin{corollary}\label{cor:unique-pen}
There is at most one solution to Problem \ref{prb:penprob}.
\end{corollary}

\subsection{Quadratic growth and stability}\label{sec:growth}
Here we establish growth and stability results for $u^{\eps,\delta}_m$.
\begin{lemma}\label{lem:polygrow}
Let $u^{\varepsilon,\delta}_m $ be a solution of Problem \ref{prb:penprob}. Then, there is a constant $K_3>0$ independent of $\eps,\delta,m$ such that
\begin{align}\label{eq:lem-ling}
0\leq u^{\varepsilon,\delta}_m (t,x)\leq K_3(1+|x|_d^2),\qquad\text{for all $(t,x)\in\overline{\cO}_m$}.
\end{align}
\end{lemma}
\begin{proof}
Since $f_m$, $g_m$ and $h_m$ are non-negative, then by \eqref{eq:probrap} we get $u^{\eps,\delta}_m\ge 0$. For the upper bound we pick the control pair $(n,\nu)=(e_1,0)$ (with $e_1$ the unit vector with $1$ in the first coordinate) in \eqref{eq:Jpen} and notice that $H^{\varepsilon}_m(t+s,X_s^{[e_1,0]},0)=0$. Then
\begin{align}\label{eq:ling0}
u^{\eps,\delta}_m(t,x)&\leq \sup_{w\in \mathcal{T}^\delta}\E_{x}\biggr[\int_0^{\rho_{m}}\!R^w_s\big[ h_m\!+\! w_sg_m\big](t\!+\!s,X_s^{[e_1,0]})\,\ud s\!+\!R^w_{\rho_m}g_m(t\!+\!\rho_{m} ,X_{\rho_{m}}^{[e_1,0]})\biggr]\notag\\
&\le K_1 \sup_{w\in \mathcal{T}^\delta}\E_{x}\biggr[\Big(1+\sup_{0\le s\le T-t} e^{-rs}\big|X^{[e_1,0]}_s\big|_d^2\Big)\Big(1+\int_0^{\rho_{m}}\!e^{-\int_0^s w_\lambda \ud \lambda}\big[1+w_s\big]\,\ud s\Big)\biggr],\notag
\end{align}
where the inequality is using the quadratic growth of $h_m$ and $g_m$ (see \eqref{eq:fghpolgrow}) and the definition of $R^w$. For $\P$-a.e.\ $\omega$ we have the simple bound
\begin{align*}
\int_0^{\rho_m(\omega)}\!e^{-\int_0^s w_\lambda(\omega) \ud \lambda}\big[1+w_s(\omega)\big]\,\ud s\le 1+T.
\end{align*}
Therefore
\[
u^{\eps,\delta}_m(t,x)\le K_1{(2+T)}\E_x\biggr[1+\sup_{0\le s\le T-t} e^{-rs}\big|X^{[e_1,0]}_s\big|_d^2\biggl]\le K_3(1+|x|_d^2),
\]
by standard estimates for SDEs with coefficients with linear growth (\cite[Cor.\ 2.5.10]{krylov1980controlled}) and the constant $K_3>0$ depends only on $T$, $D_1$ and $K_1$ in Assumptions \ref{ass:gen1} and \ref{ass:gen2}.  
\end{proof}

\begin{remark}\label{rem:locbndu}
In particular, \eqref{eq:lem-ling} implies that for any $m\geq m_0\in \N$ and $\eps,\delta\in(0,1)$
\begin{align*}
\big\|u^{\eps,\delta}_m\big\|_{C^0(\overline\cO_{m_0})}\leq K_3(1+|m_0|^2) \eqqcolon M_1(m_0).
\end{align*}
\end{remark}
The next result relies upon standard PDE arguments. Its proof is in Appendix for completeness.
\begin{lemma}\label{lem:stability}
Let $u^{\eps,\delta}_m$ be a solution of Problem \ref{prb:penprob}. Let $u^n\in C^{\infty}(\overline{\cO}_m)$ and $\chi_n\in C^{\infty}(\R)$ be such that $\chi_n(0)=0$, $\chi_n'\ge 0$, $(\chi_n)_{n\in\N}$ are equi-Lipschitz, and
\[
\|u^n-u^{\eps,\delta}_m\|_{C^{1,2,\gamma}(\overline{\cO}_m)}+\|\chi_n-(\cdot)^+\|_{C^0(\R)}\leq \tfrac{1}{n},\quad n\in\N,
\] 
for some $\gamma\in(0,\alpha)$.
Then, there exists a unique solution $w^n\in C^{1,2,\alpha}(\overline{\cO}_m)$ of
\begin{align}\label{eq:PDEsmth2}
\begin{cases}\partial_tw^n\!+\!\mathcal{L}w^n\!-\!rw^n=\!-h_m\!-\!\frac{1}{\delta}\chi_n\big(g_m\!-\!u^{\eps,\delta}_m\big)\!+\!\psi_{\varepsilon}(|\nabla u^n|^2_d\!-\!f^2_m\!-\!\tfrac{1}{n}), & \text{ on }\cO_m, \\
w^n(t,x)=g_m(t,x), & (t,x)\in\partial_P \cO_m.
\end{cases}
\end{align}
Moreover, $w^n\in C^{1,3,\alpha}_{Loc}(\cO_m)$ and $w^n\to u^{\eps,\delta}_m$ in $C^{1,2,\beta}(\overline{\cO}_m)$ as $n\to\infty$, for all $\beta\in(0,\alpha)$.
\end{lemma}

\subsection{Gradient bounds}\label{sec:grd}

Our next goal is to find a bound for the norm of the gradient of $u^{\eps,\delta}_m$ uniformly in $\eps,\delta$. We start by considering an estimate on the parabolic boundary $\partial_P\cO_m$ that will be later used to bound $u^{\eps,\delta}_m$ on the whole $\overline{\cO}_m$.
\begin{lemma}\label{lem:fntbndr}
Let $u^{\varepsilon,\delta}_m $ be a solution of Problem \ref{prb:penprob}. Then, there is $M_2=M_2(m)>0$ such that 
\begin{align}\label{eq:grdb}
\sup_{(t,x)\in \partial_P\cO_m}|\nabla u^{\varepsilon,\delta}_m (t,x)|_d\leq M_2,\quad\text{for all $\eps,\delta\in(0,1)$}.
\end{align}
\end{lemma}
\begin{proof}
For simplicity we denote $u=u^{\eps,\delta}_m$. If $t=T$ we have $u(T,x)=g_m(T,x)$ for $x\in\overline{B}_m$ and the bound is trivial because $\nabla u(T,x)=\nabla g_m(T,x)$. 

Next, let $t\in[0,T)$. Notice that $u|_{\partial B_m}=g_m|_{\partial B_m}=0$. Fix $x\in B_m$ and $y\in\partial B_m$. Then
\begin{align}\label{eq:lipscond1}
0\leq u(t,x)-u(t,y)=&\,u(t,x)-g_m(t,y)\le \|\nabla g_m\|_{C^0(\overline{\cO}_m)}|x-y|_d+u (t,x)-g_m(t,x). 
\end{align}
For arbitrary $(n,\nu)\in\cA^\circ_t$ and $w\in\cT^\delta_t$, Dynkin's formula gives
\begin{align}\label{eq:gmprobrep}
g_m(t,x)=\E_x\biggr[&\,R^w_{\rho_m}g_m(t+\rho_{m},X_{\rho_{m}}^{[n,\nu]})\!-\!\int_{0}^{\rho_{m}}\! R^w_s\langle n_s\dot{\nu}_s,\nabla g_m(t+s,X_s^{[n,\nu]})\rangle\,\ud s\notag\\
&\,-\!\int_{0}^{\rho_{m}}\!R^w_s\big[(\partial_tg_m\!+\!\mathcal{L}g_m\!-\!rg_m)(t\!+\!s,X_s^{[n,\nu]})\!-\!w_s g_m(t\!+\!s,X_s^{[n,\nu]})\big]\,\ud s\biggr].
\end{align}
Then, setting $\Theta_m=\partial_tg_m\!+\!\mathcal{L}g_m\!-\!rg_m\!+\!h_m$ and recalling \eqref{eq:probrap}, we can write
\begin{align*}
&u (t,x)-g_m(t,x)\\
&=\inf_{(n,\nu)\in\cA^\circ_t}\sup_{w\in\cT^\delta_t}\E_x\biggr[\int_{0}^{\rho_{m}}\!\!R^w_s\Big(\big[\Theta_m\!+\!\langle n_s\dot{\nu}_s,\nabla g_m\rangle\big]\!+\!H^{\varepsilon}_m(\cdot,n_s\dot{\nu}_s)\Big)(t\!+\!s,X_s^{[n,\nu]})\,\ud s\biggr].
\end{align*}
Picking $(n,\nu)=(e_1,0)$ (with $e_1$ the unit vector with $1$ in the first coordinate) and recalling that $H^{\eps}_m(\cdot,0)=0$ and $\rho_m=\rho_{\cO_m}$ we obtain the upper bound
\[
u(t,x)-g_{m}(t,x)\le \sup_{w\in\cT^\delta_t}\E_x\biggr[\int_{0}^{\rho_{m}}\!R^w_s\Theta_m(t\!+\!s,X_s^{[e_1,0]})\,\ud s\biggr]\le 
\big\| \Theta_m\big\|_{C^0(\overline\cO_m)} \E_x\left[\rho_{\cO_m} \right].
\]
Combining the latter with \eqref{eq:lipscond1} we obtain 
\[
0\le u (t,x)\!-\!u (t,y)\leq \|\nabla g_m\|_{C^0(\overline{\cO}_m)}|x\!-\!y|_d\!+\!\|\Theta_m\|_{C^0(\overline\cO_m)}\E_x\left[\rho_{\cO_m} \right].
\]

Since $\rho_{\cO_m}=\rho_{\cO_m}(t,x;e_1,0)$ is associated to the control pair $(n,\nu)=(e_1,0)$, then
\[
\rho_{\cO_m}=\inf\{s\ge 0\,|\, X^{[e_1,0]}_s\notin B_m \}\wedge(T-t)=:\tau_m\wedge(T-t),
\]
and clearly $\E_x[\rho_{\cO_m}]\le \E_x[\tau_m]=:\pi(x)$.
It is well-known that $\pi\in C^2(\overline B_m)$ and it solves 
\begin{align*}
\mathcal{L}\pi(x)=-1\quad \text{for $x\in B_m$}\quad\text{with}\quad\pi(x)=0 \quad \text{for $x\in\partial B_m$},
\end{align*}
by uniform ellipticity of $\cL$ on $B_m$ (see \cite[Thm.\ 6.14]{gilbarg2015elliptic}). That is sufficient to conclude 
\begin{align*}
\E_x\left[\rho_{\cO_m}\right]\le \pi(x)=\pi(x)-\pi(y)\leq L_{\pi,m}|x-y|_d,
\end{align*}
for some constant $L_{\pi,m}>0$ depending only on the coefficients of $\mathcal{L}$ and the radius $m$. Then, for all $t\in[0,T)$ we have
$0\le u(t,x)-u(t,y)\leq M_2|x-y|_d$,
with $M_2= \|\nabla g_m\|_{C^0(\overline{\cO}_m)}+\|\Theta_m\|_{C^0(\overline\cO_m)} L_{\pi,m}$. This implies \eqref{eq:grdb} because $u\in C^{1,2,\alpha}(\overline\cO_m)$.
\end{proof}

Using Lemma \ref{lem:fntbndr} we can also provide a bound on $|\nabla u^{\varepsilon,\delta}_m|_d$ in the whole domain $\overline{\cO}_m$. It is useful to recall that a function $\varphi \in C^{1,2}(\overline{\cO}_m)$ attaining a maximum at a point  $(t_0,x_0)\in\cO_m$ also satisfies 
\begin{align}\label{eq:MaxP}
\cL\varphi(t_0,x_0)+\partial_t\varphi(t_0,x_0)\leq0,
\end{align}
by the maximum principle (see \cite[Lemma 2.1]{friedman2008partial}). Since $(t_0,x_0)\in\cO_m$, then $\nabla \varphi(t_0,x_0)={\bf 0}$ and, when $t_0\in(0,T)$, also $\partial_t\varphi(t_0,x_0)=0$. 
\begin{proposition}\label{prop:locgradp}
Let $u^{\varepsilon,\delta}_m $ be a solution of Problem \ref{prb:penprob}. Then, there is $M_3=M_3(m)$ such that 
\begin{align}\label{eq:unifbndgrad}
\sup_{(t,x)\in\overline{\cO}_m}|\nabla u^{\varepsilon,\delta}_m(t,x)|_d\leq M_3,\quad\text{for all $\eps,\delta\in(0,1)$}. 
\end{align} 
\end{proposition}
\begin{proof}
This proof refines and extends arguments from \cite[Lemma A.2]{zhu1992generalized} (see also \cite[Lemma 3.4]{hynd2010partial} and \cite[Lemma 2.8]{kelbert2019hjb}). For simplicity we denote $u=u^{\eps,\delta}_m$.
{Let $\lambda_m\in(0,\infty)$ be a constant depending on $m$ but independent of $\eps,\delta$, which will be chosen later}. Let $v^\lambda\in C^{0,1,\alpha}(\overline \cO_m)$ be defined as 
\begin{align*}
v^\lambda(t,x)\coloneqq |\nabla u (t,x)|^2_d-\lambda u (t,x)
\end{align*}
for some $\lambda\in(0,\lambda_m]$. Recalling $M_1=M_1(m)$ from Remark \ref{rem:locbndu} we have for any $\lambda \in (0, \lambda_m]$
\begin{align}\label{eq:dradu<gradv}
\sup_{(t,x)\in\overline{\cO}_m}|\nabla u (t,x)|^2_d\leq \sup_{(t,x)\in\overline{\cO}_m} v^\lambda(t,x)+{\lambda_m} M_1(m). 
\end{align}
 
Let $(t^\lambda,x^\lambda)\in\overline{\cO}_m$ be a maximum point for $v^\lambda$. Two situations may arise: either $(t^\lambda,x^\lambda)\in\cO_m$ or $(t^\lambda,x^\lambda)\in\partial_P\cO_m$. If $(t^\lambda,x^\lambda)\in\partial_P\cO_m$, then by Lemmas \ref{lem:polygrow} and \ref{lem:fntbndr} we have
\begin{align*}
v^\lambda(t^\lambda,x^\lambda) \le |\nabla u (t^\lambda,x^\lambda)|^2_d\leq M_2^2\implies\sup_{(t,x)\in\overline{\cO}_m}|\nabla u (t,x)|^2_d\leq M_2^2+ {\lambda_m} M_1,
\end{align*}
where the implication follows by \eqref{eq:dradu<gradv}. Thus, if 
\begin{align}\label{eq:Lambdam}
\Lambda_m:=\{\lambda\in{(0,\lambda_m]}:\text{ there exists a maximiser of $v^\lambda$ in }\partial_P\cO_m\}\neq \varnothing, 
\end{align}
it is sufficient to pick $\lambda\in\Lambda_m$ and \eqref{eq:unifbndgrad} holds.

Let us now assume $\Lambda_m=\varnothing$ (i.e., $(t^\lambda,x^\lambda)\in\cO_m$ {for all $\lambda\in(0,\lambda_m]$}).
With no loss of generality:
\begin{align}\label{eq:assgradu1}
|\nabla u(t^\lambda,x^\lambda)|_d>1\quad\text{and}\quad \langle \nabla u(t^\lambda,x^\lambda),\nabla g_m(t^\lambda,x^\lambda)\rangle-|\nabla u(t^\lambda,x^\lambda)|^2_d < 0.
\end{align}
If either condition fails, then \eqref{eq:unifbndgrad} trivially holds. Likewise, we assume 
\begin{align}\label{eq:assgradfm}
\psi_\varepsilon'\big(|\nabla u (t^\lambda,x^\lambda)|^2_d- f_m^2(t^\lambda,x^\lambda)\big)>1
\end{align}
because $\psi_\varepsilon$ is convex and $\psi_\varepsilon'(r)=\frac{1}{\varepsilon}>1$ for $r\geq 2\varepsilon$. So, if \eqref{eq:assgradfm} fails, it must be $|\nabla u(t^\lambda,x^\lambda)|^2_d< f_m^2(t^\lambda,x^\lambda)+2\varepsilon$ and \eqref{eq:unifbndgrad} holds because $f_m$ is bounded on $\overline\cO_m$.

We would like to compute $\partial_t v+\cL v$ but the term containing $(\,\cdot\,)^+$ in the PDE for $u$ is not continuously differentiable and, therefore, it is not clear that $|\nabla u|^2_d$ admits classical derivatives. That is why we resort to an approximation procedure.
Let $u^n$, $\chi_n$ and $w^n$ be defined as in Lemma \ref{lem:stability}. Recall that $w^n\in C^{1,3,\alpha}_{Loc}(\cO_m)\cap C^{1,2,\alpha}(\overline{\cO}_m)$ and $w^n\to u$ in $C^{1,2,\gamma}(\overline{\cO}_m)$ for all $\gamma\in(0,\alpha)$. Define
\begin{align}\label{eq:v^ndefngrad}
v^{\lambda,n}(t,x)\coloneqq |\nabla w^n(t,x)|_d^2-\lambda u(t,x),
\end{align}
so that $v^{\lambda,n}\in C^{1,2,\alpha}_{Loc}(\cO_m)\cap C^{0,1,\alpha}(\overline \cO_m)$. Clearly, $v^{\lambda,n}\to v^\lambda$ uniformly on $\overline\cO_m$.

Let $(t_n^\lambda,x_n^\lambda)_{n\in\N}$ be a sequence with $(t_n^\lambda,x_n^\lambda)\in \argmax_{\overline{\cO}_m} v^{\lambda,n}$ for $n\in\N$. Since $\overline{\cO}_m$ is compact, the sequence admits a subsequence $(t_{n_k}^\lambda,x_{n_k}^\lambda)_{k\in\N}$ converging to some $(\tilde{t},\tilde{x})\in\overline{\cO}_m$. It is not hard to show that $(\tilde{t},\tilde{x})\in\argmax_{\overline{\cO}_m} v^\lambda$ (we provide the full argument in Appendix for completeness). Then, with no loss of generality we can assume $(t^\lambda,x^\lambda)=(\tilde{t},\tilde{x})$. That implies that we can choose $(t^{\lambda}_{n_k},x^{\lambda}_{n_k})_{k\in\N}\subset\cO_m$. More precisely, by continuity of $v^\lambda$, for any $\eta>0$ there exist bounded open sets $U_{\lambda,\eta}\subset\cO_m$ and $V_{\lambda,\eta}\subset B_m$ such that $(t^\lambda,x^\lambda)\in U_{\lambda,\eta}\cup\big(\{0\}\!\times\!V_{\lambda,\eta}\big)$ and $v^\lambda(t,x)> v^\lambda(t^\lambda,x^\lambda)-\eta$ for all $(t,x)\in U_{\lambda,\eta}\cup\big(\{0\}\!\times\!V_{\lambda,\eta}\big)$ (by convention, if $t^\lambda\neq 0$ we take $V_{\lambda,\eta}=\varnothing$). Moreover, for $k\in\N$ sufficiently large we have $(t^\lambda_{n_j},x^\lambda_{n_j})_{j\ge k}\subset U_{\lambda,\eta}\cup\big(\{0\}\!\times\!V_{\lambda,\eta}\big)$. 

From now on we simply relabel our subsequence by $(t^\lambda_n,x^\lambda_n)_{n\in\N}$ with a slight abuse of notation.
Since $(t^\lambda_n,x^\lambda_n)$ is a maximum point of $v^{\lambda,n}$ from the maximum principle (see \eqref{eq:MaxP}) we get
\begin{align}\label{eq:frstcndmax}
\mathcal{L} v^{\lambda,n}(t^\lambda_n,x^\lambda_n) +\partial_tv^{\lambda,n}(t^\lambda_n,x^\lambda_n) \leq 0.
\end{align}
Next we compute explicitly all terms in \eqref{eq:frstcndmax} and to simplify notation we drop the argument $(t^\lambda_n,x^\lambda_n)$. Denoting $\partial_t v=v_t$, $\partial_{x_i}v=v_{x_i}$ and $\partial_{x_i x_j}v=v_{x_i x_j}$, we obtain 
\begin{align}
\label{eq:bndgrd+var}
v_t^{\lambda,n}=&\,2\left\langle \nabla w^n,\nabla w^n_t \right\rangle-\lambda u_t;& \notag\\
v_{x_i}^{\lambda,n}=&\,2 \langle \nabla w^n, \nabla w^n_{x_i} \rangle-\lambda u_{x_i}, &1\leq i\leq d; \\
v_{x_ix_j}^{\lambda,n}=&\, 2\langle \nabla w^n_{x_i}, \nabla w^n_{x_j} \rangle+2 \langle \nabla w^n, \nabla w^n_{x_ix_j} \rangle-\lambda u_{x_ix_j}, &1\leq i,j\leq d. \notag
\end{align}
Substituting in \eqref{eq:frstcndmax} gives:
\begin{align}\label{eq:bndgrad1}
0\ge\sum_{i,j=1}^d a_{ij}\langle \nabla w^n_{x_i},\nabla w^n_{x_j} \rangle +2 \sum_{k=1}^d w^n_{x_k}(\partial_t w^n_{x_k} + \mathcal{L} w^n_{x_k})-\lambda(\partial_t u + \mathcal{L}u). 
\end{align}
From uniform ellipticity \eqref{eq:defnthetaEC} on $B_m$ and denoting {$\theta_{B_m}=\theta$} we have 
\begin{align}\label{eq:bndgrad2}
 \sum_{i,j=1}^da_{ij}\langle \nabla w^n_{x_i},\nabla w^n_{x_j} \rangle =&\,\sum_{k=1}^d\sum_{i,j=1}^da_{ij}w^n_{x_kx_i}w^n_{x_kx_j}\geq\sum_{k=1}^d \theta|\nabla w^n_{x_k}|^2_d\geq \theta |D^2 w^n|^2_{d\times d}.
\end{align}

To study the second term in \eqref{eq:bndgrad1} we introduce the differential operator $\mathcal{L}_{x_k}$: 
\begin{align}\label{eq:infgeneder}
(\cL_{x_k}\varphi)(x)=\tfrac{1}{2}\mathrm{tr}\big(a_{x_k}(x)D^2\varphi(x)\big)+\langle b_{x_k}(x),\nabla \varphi(x)\rangle,\quad\text{for $\varphi\in C^2(\R^d)$},
\end{align}
where $a_{x_k}\in \R^{d\times d}$ is the matrix with entries $(\partial_{x_k}a_{ij})^d_{i,j=1}$ and $b_{x_k}\in \R^d$ the vector with entries $(\partial_{x_k} b_i)_{i=1}^d$.
Differentiating with respect to $x_k$ the PDE in \eqref{eq:PDEsmth2} and rearranging terms we get
\begin{align}\label{eq:bndgrad3}
\partial_t w^n_{x_k} + \mathcal{L} w^n_{x_k}=&\, r w^n_{x_k}+ \psi_\varepsilon'(\bar\zeta_n)\cdot(|\nabla u^n|^2_d- f_m^2)_{x_k} \\
&\,- \tfrac{1}{\delta}\chi_n'(g_m-u)\cdot(g_m-u)_{x_k}-\partial_{x_k} h_m -\mathcal{L}_{x_k}w^n, \notag
\end{align}
where we set $\bar\zeta_n\coloneqq(|\nabla u^n|^2_d- f_m^2)(t^\lambda_n,x^\lambda_n)-\tfrac{1}{n}$ for the argument of $\psi_\varepsilon'$. 

In the third term of \eqref{eq:bndgrad1} we substitute \eqref{eq:penprob} and, combining with \eqref{eq:bndgrad2} and \eqref{eq:bndgrad3}, we obtain 
\begin{align}\label{eq:bndgrad5}
0\geq&\, \theta |D^2 w^n|^2_{d\times d}+2\Big[r|\nabla w^n|^2_d+\psi'_\eps(\bar \zeta_n)\langle \nabla w^n,\nabla(|\nabla u^n|^2_d- f_m^2)\rangle-\tfrac{1}{\delta}\chi_n'(g_m-u)\langle \nabla w^n,\nabla(g_m-u)\rangle\notag\\
&\qquad\qquad\qquad\quad-\langle\nabla w^n,\nabla h_m\rangle -\sum_{k=1}^d w^n_{x_k}\mathcal{L}_{x_k} w^n\Big]
-\lambda\big(r u +\psi_\varepsilon(\zeta_n)-\tfrac{1}{\delta}(g_m-u)^+-h_m\big),
\end{align}
where we set $\zeta_n\coloneqq(|\nabla u|^2_d- f_m^2)(t^\lambda_n,x^\lambda_n)$ for the argument of $\psi_\varepsilon$. 

Let us denote $\hat{w}^n=u-w^n$. Then $\|\hat{w}^n\|_{C^{1,2,\gamma}(\overline{\cO}_m)}\to 0$ as $n\to\infty$ because $w^n\to u$ in $C^{1,2,\gamma}(\overline{\cO}_m)$, {for all $\gamma\in(0,\alpha)$}. We claim that
\begin{align}\label{eq:bnduseproflate}
&2\Big[r|\nabla w^n|^2_d-\tfrac{1}{\delta}\chi_n'(g_m-u)\langle \nabla w^n,\nabla(g_m-u)\rangle-\langle\nabla w^n,\nabla h_m\rangle -\sum_{k=1}^d w^n_{x_k}\mathcal{L}_{x_k} w^n\Big]\\
&-\lambda\big(r u +\psi_\varepsilon(\zeta_n)-\tfrac{1}{\delta}(g_m-u)^+-h_m\big)\\
&\geq-C_1 |\nabla u|^2_d - \theta| D^2 w^n |^2_{d\times d}- C_2-\lambda r M_1-\lambda\psi_\varepsilon'(\bar\zeta_n)|\nabla u|^2_d-R_n,
\end{align}
for $M_1$ as in Remark \ref{rem:locbndu}, constants $C_1\!=\!C_1(m)\!>\!0$, $C_2\!=\!C_2(m)\!>\!0$ depending only on $m$ and
with 
\begin{align}\label{eq:defnRn1}
0\le R_n\le \kappa_{\delta,m}\Big(\|\hat{w}^n\|_{C^{0,1,\gamma}(\overline{\cO}_m)}+\lambda\big\|\psi_\varepsilon'(\zeta_n)-\psi_\varepsilon'(\bar\zeta_n)\big\|_{C^0(\overline{\cO}_m)}\Big),
\end{align}
where $\kappa_{\delta,m}>0$ depends on $\delta$, $\|\nabla u\|_{C^0(\overline\cO_m)}$ and $\|\nabla g_m\|_{C^0(\overline\cO_m)}$. Clearly $R_n\to 0$ as $n\to\infty$. For the ease of exposition the proof of \eqref{eq:bnduseproflate} is given separately at the end of this proof.

Plugging \eqref{eq:bnduseproflate} into \eqref{eq:bndgrad5} we obtain
\begin{align*}
0\ge 2 \psi_\varepsilon'(\bar\zeta_n)\langle \nabla w^n,\nabla(|\nabla u^n|^2_d- f_m^2)\rangle -C_1 |\nabla u|^2_d-C_2 -\lambda r M_1 -\lambda\psi_\varepsilon'(\bar\zeta_n)|\nabla u|^2_d-R_n.
\end{align*}
By \eqref{eq:assgradfm}, $\psi_\varepsilon'(\bar\zeta_n)\geq 1$ for large $n$. Then, multiplying both sides of the inequality by $-1$ we obtain
\begin{align*}
0\leq \psi_\varepsilon'(\bar\zeta_n)\bigg(\lambda|\nabla u|^2_d\!-\!2 \langle \nabla w^n,\nabla (|\nabla u^n|^2_d \!-\! f_m^2)\rangle \!+C_1|\nabla u|^2_d\!+\!C_2\! +\!\lambda r M_1\!+\!R_n\bigg).
\end{align*}
That implies
\begin{align}\label{eq:bndgrad6ea}
0\leq(C_1+\lambda)|\nabla u|^2_d\!-\!2\langle \nabla w^n,\nabla (|\nabla u^n|^2_d- f_m^2)\rangle\!+\!C_2\! +\!\lambda r M_1\!+\!R_n.
\end{align}

We claim that 
\begin{align}\label{eq:claimtoprove}
-2\langle \nabla w^n,\nabla (|\nabla u^n|^2_d- f_m^2)\rangle\leq -2\lambda|\nabla u|^2_d+2|\nabla u|_d|\nabla f_m^2|_d+\tilde{R}_n,
\end{align}
where $\tilde{R}_n=\tilde R_n(m,\eps,\delta)$ goes to zero as $n\to\infty$. Again, for ease of exposition the proof of \eqref{eq:claimtoprove} is given separately at end of this proof, after the one for \eqref{eq:bnduseproflate}. Thus \eqref{eq:bndgrad6ea} becomes
\begin{align*}
0\leq (C_1-\lambda)|\nabla u|^2_d+2|\nabla u|_d|\nabla f_m^2|_d +C_2 +\lambda r M_1+R_n+\tilde{R}_n.
\end{align*}
By definition of $f_m$ it is not hard to verify that $|\nabla f_m^2|_d\le C_3$ for a constant $C_3 = C_3 (m)>0$ independent of $\eps$ and $\delta$. Then from the inequality above we obtain
\begin{align*}
(\lambda-C_1)|\nabla u|^2_d\leq 2 C_3|\nabla u|_d+C_2 +\lambda r M_1+R_n+\tilde{R}_n.
\end{align*}
Choosing $\lambda=\bar \lambda \coloneqq C_1+1$ and recalling our shorthand notation $\nabla u=\nabla u (t^\lambda_n,x^\lambda_n)$, we have
\begin{align}\label{eq:bndgrad7baa}
|\nabla u(t^{\bar\lambda}_n,x^{\bar\lambda}_n)|_d^2 \leq&\, 2 C_3|\nabla u(t^{\bar\lambda}_n,x^{\bar\lambda}_n)|_d+C_2 + \bar \lambda r M_1+R_n+\tilde{R}_n.
\end{align}

Thanks to \eqref{eq:assgradu1} we have $|\nabla u(t^{\bar \lambda},x^{\bar \lambda})|_d> 1$. Thus, with no loss of generality we can assume that $|\nabla u(t,x)|_d\geq1$ for all $(t,x)\in\overline{U}_{\bar \lambda,\eta}\cup\big(\{0\}\!\times\!\overline{V}_{\bar \lambda,\eta}\big)$ and recall also that $(t^{\bar \lambda}_n,x^{\bar \lambda}_n)_{n\in\N}\subset U_{\bar\lambda,\eta}\cup\big(\{0\}\!\times\!V_{\bar \lambda,\eta}\big)$. Thus, dividing by $|\nabla u(t^{\bar \lambda}_n,x^{\bar \lambda}_n)|_d$ in \eqref{eq:bndgrad7baa} we get
\begin{align}\label{eq:bndgrad7b}
|\nabla u(t^{\bar \lambda}_n,x^{\bar \lambda}_n)|_d \leq&\, 2 C_3+C_2 + \bar \lambda r M_1+R_n+\tilde{R}_n.
\end{align}
From \eqref{eq:bndgrad7b} and the definition of $U_{\bar \lambda,\eta}\cup(\{0\}\times V_{\bar \lambda,\eta})$
\begin{align}\label{eq:vlbar}
\sup_{(t,x)\in\overline{\cO}_m} v^{\bar \lambda}(t,x)=&\,v^{\bar \lambda}(t^{\bar \lambda},x^{\bar \lambda})\leq v(t^{\bar \lambda}_n,x^{\bar \lambda}_n)\!+\!\eta \leq (2 C_3\!+\!C_2\! + \!\bar \lambda r M_1+R_n+\tilde{R}_n)^2+\eta.
\end{align}
Letting $n\to\infty$ we obtain 
\begin{align*}
\sup_{(t,x)\in\overline{\cO}_m} v^{\bar \lambda}(t,x)\leq(2 C_3+C_2 + \bar\lambda r M_1)^2+\eta.
\end{align*}
By the arbitrariness of $\eta$ and since \eqref{eq:dradu<gradv} holds for any $\lambda\in (0,\lambda_m]$, taking $\lambda_m=\bar \lambda$ we have
\begin{align*}
\sup_{(t,x)\in\overline{\cO}_m} |\nabla u(t,x)|^2_d\leq(2 C_3+C_2 + \bar\lambda r M_1)^2+\bar \lambda M_1.
\end{align*}
Hence, the proposition holds with $M_3=\! ((2 C_3\!+\!C_2\! +\! \bar \lambda rM_1)^2\!+\!\bar \lambda M_1)^{1/2}$ independent of $\varepsilon$ and $\delta$.
\end{proof}
\begin{proof}[{\bf Proof of} \eqref{eq:bnduseproflate}]
Recalling $\hat w^n=u-w^n$ we first notice 
\begin{align}\label{eq:wnrewritu}
|\nabla w^n|^2_d=&\,\left(|\nabla u|^2_d-2\langle\nabla u ,\nabla \hat{w}^n\rangle+|\nabla \hat{w}^n|^2_d\right)\\
\leq&\,|\nabla u|^2_d+\|\nabla \hat{w}^n \|_{C^0(\overline{\cO}_m)}\left(2\|\nabla u \|_{C^0(\overline{\cO}_m)}+\|\nabla \hat{w}^n \|_{C^0(\overline{\cO}_m)}\right).\notag
\end{align}
The first term on the left-hand side of \eqref{eq:bnduseproflate} is positive. For the second one, notice
\begin{align}
\langle\nabla w^n,\nabla(g_m-u)\rangle=&\langle\nabla (u-\hat w^n),\nabla(g_m-u)\rangle\\
=&\,\langle \nabla u,\nabla g_m\rangle-|\nabla u|^2_d-\langle \nabla \hat{w}^n,\nabla (g_m-u)\rangle\\
\le &\, \langle \nabla \hat{w}^n,\nabla (u-g_m)\rangle\le \|\hat{w}^n\|_{C^{0,1,\gamma}(\overline{\cO}_m)}( \|\nabla u\|_{C^0(\overline{\cO}_m)}+\|\nabla g_m\|_{C^0(\overline{\cO}_m)}),
\end{align}
where the first inequality holds by \eqref{eq:assgradu1} in $U_{\lambda,\eta}\cup(\{0\}\times V_{\lambda,\eta})$. Since $0\le \chi'_n\le 2$, then
\begin{align}\label{eq:grd00}
\tfrac{1}{\delta}\chi'_n(g_m-u)\langle\nabla w^n,\nabla(g_m-u)\rangle\le \tfrac{2}{\delta}\|\hat{w}^n\|_{C^{0,1,\gamma}(\overline{\cO}_m)}( \|\nabla u\|_{C^0(\overline{\cO}_m)}+\|\nabla g_m\|_{C^0(\overline{\cO}_m)}),
\end{align}
and this term can be collected into $R_n$ in \eqref{eq:bnduseproflate}.
	
Let us look at the third and fourth term on the left-hand side of \eqref{eq:bnduseproflate}. For the former we have 
\begin{align}\label{eq:nabhm}
\langle\nabla w^n,\nabla h_m\rangle\le |\nabla w^n|_d|\nabla h_m|_d\le \tfrac{1}{2}|\nabla w^n|^2_d+\tfrac{1}{2}|\nabla h_m|^2_d.
\end{align}
For the latter, recalling \eqref{eq:infgeneder} we have
\begin{align}\label{eq:wLw}
\sum_{k=1}^d\! w^n_{x_k}\cL_{x_k}w^n
&=\tfrac{1}{2}\!\sum_{i,j=1}^d\! \langle \nabla w^n,\nabla a_{ij}\rangle w^n_{x_i x_j}\!+\!\sum_{i=1}^d\!\langle \nabla w^n,\nabla b_i\rangle w^n_{x_i}\\
&\le \tfrac{d^2}{2} A_m |\nabla w^n|_d|D^2 w^n|_{d\times d}\!+\!d A_m |\nabla w^n|^2_d,\notag
\end{align}
where we used Cauchy-Schwarz inequality and set 
\begin{align}\label{eq:defnsupab}
A_m\coloneqq \max_{i,j}\Big(\big\|\nabla a_{ij}\big\|_{C^0(\overline B_m)}+ \big\|\nabla b_i\big\|_{C^0(\overline{B}_m)}\Big).
\end{align}
Using that $d^2 A_m|\nabla w^n|_d|D^2 w^n|_{d\times d}\le \theta^{-1} d^4A^2_m|\nabla w^n|^2_d+\theta|D^2 w^n|^2_{d\times d}$, with $\theta=\theta_{B_m}$ as in \eqref{eq:defnthetaEC}, and combining \eqref{eq:nabhm}, \eqref{eq:wLw} and \eqref{eq:wnrewritu}, we have
\begin{align}\label{eq:bndgrad5a1}
&\langle\nabla w^n,\nabla h_m\rangle+\sum_{k=1}^d w^n_{x_k}\cL_{x_k}w^n
\le\tfrac{1}{2} C_1 |\nabla w^n|^2_d + \tfrac{1}{2}\theta| D^2 w^n |^2_{d\times d}+\tfrac{1}{2} C_2\\
&\leq  \tfrac{1}{2}\big[C_1 |\nabla u|^2_d \!+\! \theta| D^2 w^n |^2_{d\times d}\!+\!C_2\!+\!C_1\|\nabla \hat{w}^n \|_{C^0(\overline{\cO}_m)}\big(2\|\nabla u \|_{C^0(\overline{\cO}_m)}\!+\!\|\nabla \hat{w}^n \|_{C^0(\overline{\cO}_m)}\big)\big],\notag
\end{align}
with
\begin{align}\label{eq:defnC2.1}
C_1=C_1(m)\coloneqq 1+d^4A^2_m\theta^{-1}+2dA_m\quad\text{ and }\quad C_2=C_2(m)\coloneqq  \|\nabla h_m\|^2_{C^0(\overline{\cO}_m)}.
\end{align}
The expression involving $\|\nabla \hat w^n\|_{C^0}$ can be collected into $R_n$ in \eqref{eq:bnduseproflate}.

It remains to find an upper bound for $\lambda\left(ru +\psi_\varepsilon(\zeta_n)-\frac{1}{\delta}(g_m-u)^+-h_m\right)$. Since $h_m\ge 0$ and $(g_m-u)^+\ge 0$ and taking $M_1= M_1(m)>0$ as in Remark \ref{rem:locbndu} we have
\begin{align}\label{eq:grd01}
\lambda\big(ru +\psi_\varepsilon(\zeta_n)-\tfrac{1}{\delta}(g_m-u)^+-h_m\big)\le \lambda \big(r M_1+\psi_\varepsilon(\zeta_n)\big).
\end{align}
By convexity of $\psi_\varepsilon$ and since $\psi_\varepsilon(0)=0$, we have
\begin{align}\label{eq:psicnvxzeta}
\psi_\varepsilon(\zeta_n)\leq&\,\psi_\varepsilon'(\zeta_n)(|\nabla u|^2_d-f_m^2)\leq \psi_\varepsilon'(\zeta_n)|\nabla u|^2_d
\leq\psi_\varepsilon'(\bar\zeta_n)|\nabla u|^2_d+\big|\psi_\varepsilon'(\zeta_n)-\psi_\varepsilon'(\bar\zeta_n)\big||\nabla u|^2_d\\
\leq&\,\psi_\varepsilon'(\bar\zeta_n)|\nabla u|^2_d+\big\|\psi_\varepsilon'(\zeta_n)-\psi_\varepsilon'(\bar\zeta_n)\big\|_{C^0(\overline{\cO}_m)}\|\nabla u\|^2_{C^0(\overline{\cO}_m)}.
\end{align}
Combining \eqref{eq:grd00}, \eqref{eq:bndgrad5a1}, \eqref{eq:grd01} and \eqref{eq:psicnvxzeta} we obtain \eqref{eq:bnduseproflate}.
\end{proof}

\begin{proof}[{\bf Proof of} \eqref{eq:claimtoprove}]
By Cauchy-Schwarz inequality and recalling that $\hat{w}^n=u-w^n$ we have
\begin{align}\label{eq:claimtoprove1}
&-2\langle \nabla w^n,\nabla(|\nabla u^n|^2_d- f_m^2)\rangle\\
&\leq-2\langle \nabla w^n,\nabla|\nabla u^n|^2_d\rangle+2|\nabla w^n|_d|\nabla f_m^2|_d\notag\\
&\leq-2\langle \nabla w^n,\nabla|\nabla u^n|^2_d\rangle+2|\nabla u|_d|\nabla f_m^2|_d+2\|\nabla \hat{w}^n\|_{C^0(\overline{\cO}_m)}\|\nabla f^2_m \|_{C^0(\overline{\cO}_m)}.
\end{align}
The first term on the right-hand side of \eqref{eq:claimtoprove1} can be written as 
\begin{align}
\left\langle\nabla w^n, \nabla |\nabla u^n|^2_d\right\rangle=&\left\langle\nabla u, \nabla |\nabla u^n|^2_d\right\rangle-\left\langle\nabla \hat w^n, \nabla |\nabla u^n|^2_d\right\rangle.
\end{align}
The $k$-th entry of the vector $\nabla|\nabla u^n|^2_d$ reads $\big(|\nabla u^n|^2_d\big)_{x_k}=2\langle\nabla u^n,\nabla u^n_{x_k}\rangle$ and therefore
\begin{align*}
\big|\nabla|\nabla u^n|^2_d\big|^2_d\le&\, 4\sum_{k=1}^d\big|\langle\nabla u^n,\nabla u^n_{x_k}\rangle\big|^2
\le 4\sum_{k=1}^d|\nabla u^n|^2_d\,|\nabla u^n_{x_k}|^2_d\\
\le&\,4|\nabla u^n|^2_d\,|D^2 u^n|^2_{d\times d}\le 2\Big(\|\nabla u^n\|^4_{C^0(\overline \cO_m)}+\|D^2 u^n\|^4_{C^0(\overline \cO_m)}\Big).
\end{align*}
Thus, using $\sqrt{a^2+b^2}\le |a|+|b|$, we have
\begin{align}\label{eq:grd000}
\left\langle\nabla w^n, \nabla |\nabla u^n|^2_d\right\rangle\ge&\, \left\langle\nabla u, \nabla |\nabla u^n|^2_d\right\rangle\\
&-\sqrt{2}\,\|\nabla \hat w^n\|_{C^0(\overline\cO_m)}\Big(\|\nabla u^n\|^2_{C^0(\overline \cO_m)}+\|D^2 u^n\|^2_{C^0(\overline \cO_m)}\Big).
\end{align}
For the term $\left\langle\nabla u, \nabla |\nabla u^n|^2_d\right\rangle$ we argue as follows:
\begin{align}\label{eq:2ndparticlaim}
u_{x_k}\partial_{x_k} |\nabla u^n|^2_d=&\,2u_{x_k}\langle\nabla u^n,\nabla u_{x_k}^n\rangle\notag\\
=&\,2u_{x_k}\langle\nabla u^n-\nabla u,\nabla u_{x_k}^n\rangle+2u_{x_k}\langle\nabla u,\nabla u_{x_k}^n-\nabla u_{x_k}\rangle+2u_{x_k}\langle\nabla u,\nabla u_{x_k}\rangle\\
\geq&\, -2\|u^n-u \|_{C^{1,2,\gamma}(\overline{\cO}_m)}\|\nabla u \|_{C^0(\overline{\cO}_m)}\left(\|D^2u^n \|_{C^0(\overline{\cO}_m)}+\|\nabla u \|_{C^0(\overline{\cO}_m)} \right)\notag\\
&\,+2u_{x_k}\langle\nabla u,\nabla u_{x_k}\rangle.\notag
\end{align}
For the last term above, recall that all expressions are evaluated at $(t^\lambda_n,x^\lambda_n)$, which is a stationary point for $v^{\lambda,n}$ in the spatial coordinates. Thus, $2\langle\nabla w^n,\nabla w^n_{x_i}\rangle=\lambda u_{x_i}$ by \eqref{eq:bndgrd+var} and we get
\begin{align}\label{eq:3rdparticlaim}
2u_{x_k}\langle\nabla u,\nabla u_{x_k}\rangle=&\,2u_{x_k}\langle\nabla (w^n+\hat{w}^n),\nabla (w^n_{x_k}+\hat{w}^n_{x_k})\rangle\notag\\
=&\,2u_{x_k}\langle\nabla w^n,\nabla w^n_{x_k}\rangle+2u_{x_k}\langle \nabla w^n,\nabla \hat{w}^n_{x_k}\rangle+2u_{x_k}\langle \nabla \hat{w}^n,\nabla (w^n_{x_k}+\hat{w}^n_{x_k})\rangle\\
\geq&\,\lambda u_{x_k}^2-2\|\hat{w}^n \|_{C^{1,2,\gamma}(\overline{\cO}_m)}\|\nabla u \|_{C^0(\overline{\cO}_m)}\left(\|\nabla w^n \|_{C^0(\overline{\cO}_m)}+\|D^2u \|_{C^0(\overline{\cO}_m)}\right).\notag
\end{align}
Summing over all $k$'s in \eqref{eq:2ndparticlaim} we have 
\begin{align}\label{eq:grd001}
\langle \nabla u,\nabla|\nabla u^n|^2_d\rangle\ge& \lambda|\nabla u|^2_d-2 d\|\hat{w}^n \|_{C^{1,2,\gamma}(\overline{\cO}_m)}\|\nabla u \|_{C^0(\overline{\cO}_m)}\left(\|\nabla w^n \|_{C^0(\overline{\cO}_m)}+\|D^2 u \|_{C^0(\overline{\cO}_m)}\right)\\
&-2d\|u^n-u \|_{C^{1,2,\gamma}(\overline{\cO}_m)}\|\nabla u \|_{C^0(\overline{\cO}_m)}\left(\|D^2u^n \|_{C^0(\overline{\cO}_m)}+\|\nabla u \|_{C^0(\overline{\cO}_m)} \right)\notag.
\end{align}
Plugging this expression back into \eqref{eq:grd000} we get 
\begin{align}\label{eq:grd003}
&\langle \nabla w^n, \nabla|\nabla u^n|^2_d\rangle\\
&\ge \lambda|\nabla u|^2_d-2 d\|\hat{w}^n \|_{C^{1,2,\gamma}(\overline{\cO}_m)}\|\nabla u \|_{C^0(\overline{\cO}_m)}\left(\|\nabla w^n \|_{C^0(\overline{\cO}_m)}+\|D^2 u \|_{C^0(\overline{\cO}_m)}\right)\notag\\
&\quad-2d\|u^n-u \|_{C^{1,2,\gamma}(\overline{\cO}_m)}\|\nabla u \|_{C^0(\overline{\cO}_m)}\left(\|D^2u^n \|_{C^0(\overline{\cO}_m)}+\|\nabla u \|_{C^0(\overline{\cO}_m)} \right)\notag\\
&\quad-\sqrt{2}\,\|\nabla \hat w^n\|_{C^0(\overline\cO_m)}\Big(\|\nabla u^n\|^2_{C^0(\overline \cO_m)}+\|D^2 u^n\|^2_{C^0(\overline \cO_m)}\Big).\notag
\end{align}

Since $w^n\to u$ and $u^n\to u$ in $C^{1,2,\gamma}(\overline\cO_m)$ as $n\to\infty$ {for all $\gamma\in(0,\alpha)$} (Lemma \ref{lem:stability}), then all the terms in \eqref{eq:claimtoprove1} and \eqref{eq:grd003} depending on $\hat w^n$ and $u^n-u$ can be collected in a remainder $\tilde R_n=\tilde R_n(m,\eps,\delta)$ that goes to zero as $n\to\infty$. So, \eqref{eq:claimtoprove1} and \eqref{eq:grd003} give us
\begin{align*}
-2\langle\nabla w^n,\nabla (|\nabla  u^n|^2_d-f^2_m)\rangle\leq-2\lambda|\nabla u|^2_d+2|\nabla u|_d|\nabla f_m^2|_d+\tilde{R}_n,
\end{align*}
as claimed in \eqref{eq:claimtoprove}. 
\end{proof}

\subsection{Solution to the penalised problem}\label{sec:sol-pen}
In this section we prove that Problem \ref{prb:penprob} admits a unique solution. The proof is based on a fixed point argument that requires the a priori estimates derived in Sections \ref{sec:growth} and \ref{sec:grd} as well as the next bound. 

\begin{proposition}\label{prp:hldrnrme}
Let $u^{\varepsilon,\delta}_m $ be a solution of Problem \ref{prb:penprob}. For any $\beta\in(0,1)$ there is $M_4=M_4(m,\varepsilon,\delta,\beta)$ such that
\begin{align}\label{eq:uc01}
\|u^{\varepsilon,\delta}_m \|_{C^{0,1,\beta}(\overline{\cO}_m)}\leq M_4.
\end{align}
\end{proposition}
\begin{proof}
Let $u=u^{\varepsilon,\delta}_m$ for simplicity. Then, $u$ can be seen as the unique solution $\varphi$ of the linear PDE 
\begin{align*}
\partial_t\varphi+\cL\varphi-r\varphi=-h_m-\tfrac{1}{\delta}(g_m-u)^++\psi_\varepsilon(|\nabla u|^2_d-f^2_m),\quad \text{on $\cO_m$}
\end{align*}
with boundary conditions $\varphi(t,x)=0$ for $x\in\partial B_m$ and $\varphi(T,x)=g_m(T,x)$ for all $x\in B_m$. The theory of strong solutions for linear parabolic PDEs (see \cite[Thm.\ 2.6.5 and Rem.\ 2.6.4]{bensoussan2011applications}) gives 
\begin{align}\label{eq:W12pnorm}
\|u \|_{W^{1,2,p}(\cO_m)}\leq C\left( \left\|h_m+\tfrac{1}{\delta}(g_m-u )^+-\psi_\varepsilon(|\nabla u |_d^2-f_m^2)\right\|_{L^p(\cO_m)}+\|g_m\|_{W^{1,2,p}(\cO_m)}\right)
\end{align}
for any $p\in(1,\infty)$, with $C$ a constant independent of $\varepsilon$ and $\delta$. Denoting $|\cO_m|$ the volume of the set $\cO_m$, thanks to Proposition \ref{prop:locgradp} and $u\ge 0$ we have
\begin{align}\label{eq:bndW12loc}
\|u\|_{W^{1,2,p}(\cO_m)} \leq C|\cO_m|^{\frac{1}{p}}\left( \left\| h_m+\tfrac{1}{\delta}g_m\right\|_{C^{0}(\overline{\cO}_m)}+\tfrac{1}{\varepsilon}M_3^2\right)+C\|g_m\|_{W^{1,2,p}(\cO_m)}<\infty,
\end{align}
having also used that $\psi_\eps(x)\le \tfrac1\eps x $ for all $x\ge 0$.

Since $p$ is arbitrary, then Sobolev embedding \eqref{eq:inclW12C01} guarantees that for any $\beta\in(0,1)$ there exists a constant $M_4=M_4(m,\varepsilon,\delta,\beta)$ such that \eqref{eq:uc01} holds.
\end{proof}

\begin{theorem}\label{thm:exisolpenprb}
There exists a unique solution $u^{\varepsilon,\delta}_m$ of Problem \ref{prb:penprob}.
\end{theorem}
\begin{proof}
Uniqueness is by Corollary \ref{cor:unique-pen}. Existence will be proved {refining} arguments from \cite[Prop.\ 1.2]{kelbert2019hjb}. 
Fix $\varepsilon,\delta\in(0,1)$ and $m\in\N$. Let us define a subset of $C^{0,1,\alpha}(\overline{\cO}_m)$ as
\begin{align*}
C^{0,1,\alpha}_*(\overline{\cO}_m)\coloneqq\left\{\varphi\in C^{0,1,\alpha}(\overline{\cO}_m)\left| \begin{aligned}\varphi=0 \text{ on }\{T\}\times\partial B_m \text{ and }\\
\nabla\varphi= {\bf 0} \text{ on }\{T\}\times \partial B_m\quad\;\,\end{aligned}\right\}\right..
\end{align*} 
The space $C^{0,1,\alpha}_*(\overline{\cO}_m)$ is closed under the norm $\|\cdot\|_{C^{0,1,\alpha}(\overline{\cO}_m)}$. Hence $(C^{0,1,\alpha}_*(\overline{\cO}_m),\|\cdot\|_{C^{0,1,\alpha}(\overline{\cO}_m)})$ is a Banach space. For $C^{1,2,\alpha}_*(\overline{\cO}_m)\coloneqq ( C^{1,2,\alpha}\cap C^{0,1,\alpha}_*)(\overline{\cO}_m)$, also $(C^{1,2,\alpha}_*(\overline{\cO}_m),\|\cdot\|_{C^{1,2,\alpha}(\overline{\cO}_m)})$ is a Banach space.

Given $\varphi\in C^{0,1,\alpha}_*(\overline{\cO}_m)$ we consider the linear partial differential equation for $w=w^\varphi$:
\begin{align}\label{eq:pdeoperP}
\begin{cases}
\partial_tw+\mathcal{L}w-rw=-h_m-\frac{1}{\delta}\left( g_m-\varphi\right)^++\psi_\varepsilon\left(|\nabla \varphi|_d^2-f_m^2\right), &\text{ on } \cO_m, \\
w(t,x)=g_m(t,x), &\text{$(t,x)\in\partial_P \cO_m$}.
\end{cases}
\end{align}
For $x\in\partial B_m$ the compatibility condition
\[
\lim_{s\uparrow T}(\partial_tg_m\!+\!\mathcal{L}g_m\!-\!rg_m)(s,x)\!=\!\big[\!-h_m\!-\!\tfrac{1}{\delta}\left( g_m\!-\!\varphi\right)^+\!+\!\psi_\varepsilon\left(|\nabla \varphi|_d^2\!-\!f_m^2\right)\big](T,x),
\]
holds, with both sides of the expression being equal to zero. Indeed, on the left-hand side, properties of the cut-off function $\xi_{m-1}\in C^\infty_c(B_m)$ guarantee $g_m=\partial_{x_i}g_m=\partial_{x_i,x_j}g_m=0$ on $[0,T]\times\partial B_m$, hence also $\partial_t g_m=0$. On the right-hand side of the equation, we use that $h_m=\varphi=\partial_{x_i}\varphi=0$ on $\{T\}\times\partial B_m$.
Therefore \eqref{eq:pdeoperP} admits a unique solution in $C^{1,2,\alpha}(\overline{\cO}_m)$ by \cite[Thm.\ 3.3.7]{friedman2008partial}. The boundary condition of the PDE implies $w= 0$ and $\nabla w= {\bf0}$ on $\{T\}\times \partial B_m$. Hence $w\in C^{1,2,\alpha}_*(\overline{\cO}_m)$. 

Define the operator $\Gamma:C^{0,1,\alpha}_*(\overline{\cO}_m)\to C^{0,1,\alpha}_*(\overline{\cO}_m)$ that maps $\varphi$ to the solution of the PDE \eqref{eq:pdeoperP}, i.e., $\Gamma[\varphi]=w^\varphi$. Next, we show that $\Gamma$ has a fixed point by Schaefer's fixed point theorem (stated in Appendix for completeness). So Problem \ref{prb:penprob} admits a solution. 

We have $\Gamma[\varphi]\in C^{1,2,\alpha}_*(\overline{\cO}_m)\Longrightarrow \Gamma[\varphi]\in C^{0,1,\beta}_*(\overline{\cO}_m)$ for all $\beta\in(0,1)$ by \eqref{eq:inclW12C01}. 
We must prove that $\Gamma$ is continuous and compact in $C^{0,1,\alpha}_*(\overline\cO_m)$. 
Consider a sequence $(\varphi_n)_{n\in\N}\subset C^{0,1,\alpha}_*(\overline{\cO}_m)$ such that $\varphi_n\to \varphi$ in $C^{0,1,\alpha}_*(\overline{\cO}_m)$. Let $F:\R\times\R^d\to\R$ be defined as
\begin{align*}
F(q,p)\coloneqq \psi_{\varepsilon}\left(|p|^2_d-f^2_m\right)-\tfrac{1}{\delta}(g_m-q)^+.
\end{align*}
Clearly $|\varphi|$ and $|\nabla\varphi|_d$ are bounded on $\overline \cO_m$. Since $F$ is locally Lipschitz, then standard estimates for H\"older norms allow to prove $F(\varphi_n,\nabla \varphi_n)\to F(\varphi,\nabla \varphi)$ in {$C^{\gamma}(\overline{\cO}_m)$} as $n\to\infty$ for any $\gamma\in(0,\alpha)$. By Lemma \ref{lem:approxpenprob} in Appendix, $\Gamma[\varphi_n]\to\Gamma[\varphi]$ in $C^{1,2,\gamma'}_*(\overline\cO_m)$ for any $\gamma'\in(0,\gamma)$, as $n\to\infty$. Sobolev embedding \eqref{eq:inclW12C01} implies $\Gamma[\varphi_n]\to\Gamma[\varphi]$ in $C^{0,1,\beta}_*(\overline\cO_m)$ for any $\beta\in(0,1)$, hence continuity of $\Gamma$. 

For compactness, notice that $\Gamma$ maps bounded sets of $C^{0,1,\alpha}_*(\overline{\cO}_m)$ into bounded sets of $C^{1,2,\alpha}_*(\overline{\cO}_m)$ by \cite[Thms.\ 3.2.6 and 3.3.7]{friedman2008partial}. Since bounded sets in $C^{1,2,\alpha}_*(\overline{\cO}_m)$ are bounded in $C^{0,1,\beta}_*(\overline{\cO}_m)$ for all $\beta\in(0,1)$ (see \eqref{eq:inclW12C01}), then $C^{1,2,\alpha}_*(\overline{\cO}_m)\hookrightarrow C^{0,1,\alpha}_*(\overline{\cO}_m)$ is a compact embedding. 

It remains to prove that the set
\begin{align}
\cB\coloneqq \Big\{\varphi\in C^{0,1,\alpha}_*(\overline{\cO}_m)\Big|\, \rho \Gamma[\varphi]=\varphi\text{ for some }\rho\in[0,1] \Big\}
\end{align}
is bounded in $C^{0,1,\alpha}_*(\overline{\cO}_m)$. If $\rho=0$, then $\varphi=0$. If $\rho \Gamma[\varphi]=\varphi$ for some $\rho\in(0,1]$, then $\varphi$ satisfies
\begin{align}\label{eq:penpderho}
\begin{cases}\partial_t\varphi+\mathcal{L}\varphi-r\varphi=\rho\big[-h_m-\frac{1}{\delta}\big( g_m-\varphi\big)^++\psi_\varepsilon\big(|\nabla \varphi|_d^2-f_m^2\big)\big], &\text{ on } \cO_m, \\
\varphi=\rho g_m, &\text{ on }\partial_P\cO_m.
\end{cases}
\end{align}
The PDE in \eqref{eq:penpderho} is the same as the one in \eqref{eq:penprob} but with $h_m$, $\psi_\varepsilon$, $g_m$, $\frac{1}{\delta}$ replaced by $\rho h_m$, $\rho\psi_\varepsilon$, $\rho g_m$, $\frac{\rho}{\delta}$. 
Then, all the results from this section and Proposition \ref{prp:hldrnrme} apply to $\varphi$. In particular, $\|\varphi\|_{C^{0,1,\alpha}(\overline{\cO}_m)}\leq M_4$ uniformly for all $\rho\in(0,1]$, with $M_4$ as in Proposition \ref{prp:hldrnrme}.
Finally, Schaefer's fixed point theorem guarantees existence of the solution of Problem \ref{prb:penprob}, for every triple $(\varepsilon,\delta,m)$.
\end{proof}

\section{Penalised Problem on Unbounded Domain and Further Estimates}\label{sec:penalised-b}
We refine our a priori estimates on the solution of Problem \ref{prb:penprob}. First we shall make all bounds independent of $m$ so that we can construct a solution to a penalised problem on unbounded domain as $m\to\infty$. Then we shall find bounds independent of $\varepsilon$ and $\delta$ so as to pass to the limit for $\eps,\delta\downarrow 0$ and obtain a solution to Problem \ref{prb:varineq}. 

\subsection{Estimates independent of $m$}
First we bound $\nabla u^{\eps,\delta}_m$ independently of $m$ on each compact.
\begin{proposition}\label{prop:gradbndU}
Fix $m_0\!\in\!\N$ and $q\!\geq\! m_0\!+\!2$. Let $u^{\varepsilon,\delta}_q$ be the unique solution of Problem \ref{prb:penprob} on $\cO_q$. 
Then, there is $N_1=N_1(m_0)$ independent of $\varepsilon$, $\delta$ and $q$, such that 
\begin{align}\label{eq:gradbndU}
\sup_{(t,x)\in\overline{\cO}_{m_0}}|\nabla u^{\varepsilon,\delta}_q(t,x)|_d\leq N_1.
\end{align}
\end{proposition}
\begin{proof}
For notational simplicity set $u=u^{\varepsilon,\delta}_q$, $\xi=\xi_{m_0}$ and $\|\cdot\|_{0}=\|\cdot\|_{C^0(\overline{\cO}_{m_0+1})}$.
Since $q\geq m_0+2$, we have $f_q=f$, $g_q=g$ and $h_q=h$ on $\overline{\cO}_{m_0+1}$. {Let $\lambda_{0}\in(0,\infty)$ be a constant depending on $m_0$ but independent of $\eps,\delta, q$, which will be chosen later}.
Let $v^\lambda\in C^{0,1,\alpha}(\overline\cO_{m_0+1})$ be defined as
\begin{align}\label{eq:vlambda}
v^\lambda(t,x)\coloneqq \xi(x)\left|\nabla u(t,x)\right|^2_d-\lambda u(t,x),\quad\text{ for }(t,x)\in\overline{\cO}_{m_0+1}
\end{align}
{for some $\lambda\in(0,\lambda_0]$}. We will use later that
\begin{align}\label{eq:vlgrdu}
\sup_{(t,x)\in\overline{\cO}_{m_0}}|\nabla u(t,x)|^2_d\leq \sup_{(t,x)\in\overline{\cO}_{m_0+1}}\xi(x)|\nabla u(t,x)|^2_d\le \sup_{(t,x)\in\overline{\cO}_{m_0+1}}v^\lambda(t,x)+{\lambda_0}\|u\|_0, 
\end{align}
where we also notice that $\|u\|_0\le M_1$ with $M_1= M_1(m_0\!+\!1)$ as in Remark \ref{rem:locbndu}.

Let $(t^\lambda,x^\lambda)\in\overline\cO_{m_0+1}$ be a maximum point for $v^\lambda$. If $x^\lambda\in\partial B_{m_0+1}$ then $\xi(x^\lambda)=0$ and $v(t^\lambda,x^\lambda)=-\lambda u(t^\lambda,x^\lambda)\leq 0$. 
If $t^\lambda=T$, then $v^\lambda(T,x^\lambda)\leq \xi(x^\lambda)|\nabla g(T,x^\lambda)|^2_d\le \|f\|^2_{0}$ (see Assumption \ref{ass:gen2}). Thus, in both cases \eqref{eq:gradbndU} holds with $N_1\ge (\|f\|^2_0+\lambda_0 M_1)^{1/2}$. 
Defining $\Lambda_{m_0+1}$ as in \eqref{eq:Lambdam} but with $\lambda_0$ instead of $\lambda_m$, if $\Lambda_{m_0+1}\neq \varnothing$ the bound holds taking $\lambda\in\Lambda_{m_0+1}$. It remains to consider the case $\Lambda_{m_0+1}=\varnothing$, so that $(t^\lambda,x^\lambda)\in\cO_{m_0+1}$ for all $\lambda\in{(0,\lambda_0]}$.

As in the proof of Proposition \ref{prop:locgradp} we use the smooth approximation $w^n$ of $u$, obtained from \eqref{eq:PDEsmth2}. Analogously, let us define $v^{\lambda,n}\coloneqq \xi\left|\nabla w^n\right|^2_d-\lambda u$ in $\overline{\cO}_{m_0+1}$ and
let $(t^\lambda_n,x^\lambda_n)_{n\in\N}$ be a sequence converging to $(t^\lambda,x^\lambda)$ with $(t^\lambda_n,x^\lambda_n)\in\argmax_{\overline{\cO}_{m_0+1}}v^{\lambda,n}$. For any $\eta>0$ there exists a neighbourhood $U_{\lambda,\eta}\cup\big(\{0\}\times V_{\lambda,\eta}\big)$ of $(t^\lambda,x^\lambda)$ such that 
\begin{align}\label{eq:vlambda2}
v^\lambda(t,x)>v^\lambda(t^\lambda,x^\lambda)-\eta,\qquad \text{for all $(t,x)\in U_{\lambda,\eta}\cup\big(\{0\}\times V_{\lambda,\eta}\big)$}
\end{align} 
and $(t^\lambda_n,x^\lambda_n)\in U_{\lambda,\eta}\cup\big(\{0\}\times V_{\lambda,\eta}\big)$ for sufficiently large $n$'s (with the convention $V_{\lambda,\eta}=\varnothing$ if $t^\lambda\neq 0$).

Taking derivatives of $v^{\lambda,n}$ we have, for  $1\leq i,j\leq d$,
\begin{align*}
v_t^{\lambda,n}=&\,2\xi\langle \nabla w^n,\nabla w^n_t\rangle-\lambda u_t;\notag\\
v_{x_i}^{\lambda,n}=&\,2\xi\langle \nabla w^n,\nabla w^n_{x_i}\rangle-\lambda u_{x_i}+\xi_{x_i}|\nabla w^n|^2_d;\\ 
v^{\lambda,n}_{x_ix_j}=&\,2\xi\langle \nabla w^n_{x_i},\nabla w^n_{x_j}\rangle+2\xi\langle \nabla w^n,\nabla w^n_{x_ix_j}\rangle\\
&\,-\lambda u_{x_ix_j}+\xi_{x_i x_j}|\nabla w^n|^2_d+2\left(\xi_{x_i}\langle \nabla w^n,\nabla w^n_{x_j}\rangle+\xi_{x_j} \langle \nabla w^n,\nabla w^n_{x_i}\rangle\right).\notag
\end{align*}
Since $(t^\lambda_n,x^\lambda_n)\in\cO_{m_0+1}$ then \eqref{eq:MaxP} gives $(\partial_tv^{\lambda,n}+\cL v^{\lambda,n})(t^\lambda_n,x^\lambda_n)\le 0$.
Substituting the expressions for the derivatives of $v^{\lambda,n}$, some tedious but straightforward calculations and symmetry of $a_{ij}$ give 
\begin{align}\label{eq:2firstbndgrad}
0\geq&\,2\xi \langle\nabla w^n,(\partial_t\! +\!\mathcal{L})(\nabla w^n)\rangle\!-\!\lambda(\partial_t u\! +\! \mathcal{L}u)\!+\!\xi\sum_{i,j=1}^d a_{ij}\langle\nabla w^n_{x_i},\nabla w^n_{x_j}\rangle\\
&+\!(\mathcal{L}\xi)|\nabla w^n|^2_d\!+\!2\sum_{i,j=1}^da_{ij}\xi_{x_i}\langle \nabla w^n,\nabla w^n_{x_j}\rangle,\notag
\end{align}
where $(\partial_t +\mathcal{L})(\nabla w^n)$ denotes the vector with entries $(\partial_t +\mathcal{L}) w^n_{x_k}$ for $k=1,\ldots,d$. Here we omit the dependence on $(t^\lambda_n,x^\lambda_n)$ for notational convenience.

The expressions for $(\partial_t+\cL)w^n_{x_k}$ and $(\partial_t+\cL)u$ are the same as in \eqref{eq:bndgrad3} and \eqref{eq:penprob}, respectively. A lower bound for $\sum a_{ij}\langle \nabla w^n_{x_i},\nabla w^n_{x_j}\rangle$ was also obtained in \eqref{eq:bndgrad2} but with $\theta$ therein replaced by $\theta=\theta_{B_{m_0+1}}$. Thus, from \eqref{eq:2firstbndgrad} we get 
\begin{align}\label{eq:1stabndunif}
0\geq&\,\xi\theta|D^2w^n|^2_{d\times d}\\
&+\!2\xi\Big[r|\nabla w^n|^2\!+\!\psi'_\eps(\bar\zeta_n)\langle \nabla w^n,\nabla (|\nabla u^n|^2_d\!-\!f^2)\rangle\!-\!\tfrac{1}{\delta}\chi'_n(g\!-\!u)\langle \nabla w^n,\nabla(g\!-\!u)\rangle\\
&\qquad\quad -\langle\nabla w^n, \nabla h\rangle-\sum_{k=1}^d w^n_{x_k}\cL_{x_k}w^n\Big]-\lambda\big(ru+\psi_\varepsilon(\zeta_n)-h-\tfrac{1}{\delta}(g-u)^+\big)\\ 
&\,+(\mathcal{L}\xi)|\nabla w^n|^2_d+2\sum_{i,j=1}^da_{ij}\xi_{x_i}\langle \nabla w^n,\nabla w^n_{x_j}\rangle\notag,
\end{align}
where $\bar\zeta_n= (|\nabla u^n|_d^2-f^2)(t^\lambda_n,x^\lambda_n)-\tfrac{1}{n}$ and $\zeta_n= (|\nabla u|_d^2-f^2)(t^\lambda_n,x^\lambda_n)$.
Up to a factor $\xi$, the expression above is the analogue of \eqref{eq:bndgrad5} but with two additional terms. As in \eqref{eq:bnduseproflate} we have
\begin{align}\label{eq:b000}
&2\xi\Big[r|\nabla w^n|^2_d-\tfrac{1}{\delta}\chi_n'(g-u)\langle \nabla w^n,\nabla(g-u)\rangle-\langle\nabla w^n,\nabla h\rangle -\sum_{k=1}^d w^n_{x_k}\mathcal{L}_{x_k} w^n\Big]\\
&-\lambda\big(r u +\psi_\varepsilon(\zeta_n)-\tfrac{1}{\delta}(g-u)^+-h\big)\\
&\geq-C_1\xi |\nabla u|^2_d - \tfrac{\theta}{2}\xi| D^2 w^n |^2_{d\times d}- \xi C_2-\lambda rM_1-\lambda\psi_\varepsilon'(\bar\zeta_n)|\nabla u|^2_d-R_n,
\end{align}
where $R_n\to 0$ as $n\to\infty$ uniformly on $\overline{\cO}_{m_0+1}$. We notice that, differently to \eqref{eq:bnduseproflate}, we have a factor $1/2$ multiplying $| D^2 w^n |^2_{d\times d}$. That of course is obtained by adjusting the constant $C_1= (1+2d^4A^2_{m_0+1}\theta^{-1}+2dA_{m_0+1})$ (see \eqref{eq:defnC2.1}) with $A_{m_0+1}$ as in \eqref{eq:defnsupab} and $C_2=\|\nabla h\|_{0}$. 

The last term on the right-hand side of \eqref{eq:1stabndunif} can be easily bounded. Set $\bar a_0\coloneqq \max_{i,j}\|a_{ij}\|_0$, which is finite by continuity of $a_{ij}$, and recall that $|\nabla\xi|_d^2\leq C_0\xi$ by \eqref{rem:cutoffXI}. Then 
\begin{align}\label{eq:mixbndxiD2}
&2\sum_{i,j=1}^d\!a_{ij}\xi_{x_i}\langle \nabla w^n,\nabla w^n_{x_j}\rangle
\geq-\!2\bar a_0 d\!\sum_{j=1}^d\!|\nabla \xi|_d|\nabla w^n|_d|\nabla w^n_{x_j}|_d\\
&\qquad\geq\,-2\bar a_0 d^2 \sqrt{C_0 \xi} |\nabla w^n|_d|D^2 w^n|_{d\times d}\geq\,- C_3|\nabla w^n|^2_d - \xi\frac{\theta}{2}| D^2 w^n |^2_{d\times d}\notag,
\end{align}
where the final inequality is by $|ab|\leq p a^2+\frac{b^2}{p}$ with $a=2\bar a_0 d^2\sqrt{C_0}|\nabla w^n|_d$, $b=\sqrt{\xi}| D^2 w^n |_{d\times d}$ and $p=2/\theta$, and setting $C_3= 8\theta^{-1}\bar a_0^2d^4C_0$. 
It is also easy to check that $\|\cL\xi\|_0\le \kappa$ for some $\kappa=\kappa(\|b\|_0,\|\sigma\|_0)>0$, because the derivatives of $\xi$ are bounded independently of $m_0$.  

Plugging \eqref{eq:b000} and \eqref{eq:mixbndxiD2} into \eqref{eq:1stabndunif} and setting $C_4=C_3+\kappa$ we have
\begin{align}\label{eq:b001}
0\ge &2\xi\psi'_\eps(\bar \zeta_n)\langle\nabla w^n,\nabla\big(|\nabla u^n|^2_d-f^2\big)\rangle-\xi C_1|\nabla u|^2_d-C_4|\nabla w^n|^2_d\\
&-\lambda\psi'_\eps(\bar \zeta_n)|\nabla u|^2_d-\xi C_2-\lambda r M_1-R_n.\notag
\end{align}
Using \eqref{eq:wnrewritu} we have $\xi C_1|\nabla u|^2_d+C_4|\nabla w^n|^2_d\leq C_5 |\nabla u|^2_d+\tilde{R}_n$, where $\tilde R_n\to 0$ as $n\to\infty$, uniformly on $\overline{\cO}_{m_0+1}$ and $C_5=C_1+C_4$. 
It remains to find a lower bound for $\xi\langle\nabla w^n,\nabla\big(|\nabla u^n|^2_d-f^2\big)\rangle$. By the same arguments as in \eqref{eq:claimtoprove1}, the bounds in \eqref{eq:grd000} and \eqref{eq:2ndparticlaim} continue to hold, up to the inclusion of the multiplicative factor $\xi$. We have an additional term in the final expression in \eqref{eq:3rdparticlaim}, because now $v_{x_k}^{\lambda,n}=0$ gives $2\xi\langle \nabla w^n,\nabla w^n_{x_k}\rangle=\lambda u_{x_k}-\xi_{x_k}|\nabla w^n|^2_d$. So, the extra term appearing in \eqref{eq:3rdparticlaim} reads $-u_{x_k}\xi_{x_k}|\nabla w^n|^2_d$ and, similarly to \eqref{eq:wnrewritu}, we get
\[
-u_{x_k}\xi_{x_k}|\nabla w^n|^2_d\geq -u_{x_k}\xi_{x_k}|\nabla u|^2_d-|u_{x_k}||\xi_{x_k}|\|\nabla \hat{w}^n \|_{0}\left(2\|\nabla u \|_{0}+\|\nabla \hat{w}^n \|_{0}\right).
\]
In summary, we have
\begin{align}\label{eq:b002}
\xi\langle\nabla w^n,\nabla\big(|\nabla u^n|^2_d-f^2\big)\rangle\ge \lambda|\nabla u|^2_d-\xi|\nabla u|_d|\nabla f^2|_d-|\nabla u|^3_d|\nabla \xi|_d-\xi \hat R_n,
\end{align}
where $\hat R_n\to 0$ as $n\to\infty$, uniformly on $\overline{\cO}_{m_0+1}$.

Substituting \eqref{eq:b002} into \eqref{eq:b001} and grouping together all terms that vanish as $n\to\infty$ we obtain
\begin{align*}
0\ge 2\psi'_\eps(\bar\zeta_n)\big[\tfrac{\lambda}{2}|\nabla u|^2_d-\xi|\nabla u|_d|\nabla f^2|_d-|\nabla u|^3_d|\nabla \xi|_d\big]-C_5|\nabla u|^2_d-C_2-\lambda r M_1 -\bar R_n,
\end{align*}
where $\bar R_n\to0$ as $n\to\infty$. Using that $\psi'_\eps(\bar \zeta_n)\ge 1$ by the analogue of \eqref{eq:assgradfm}, for sufficiently large $n$, and multiplying the above expression by $-1$ we arrive at
\begin{align*}
0\le \psi'_\eps(\bar \zeta_n)\big[-\lambda|\nabla u|^2_d+2\xi|\nabla u|_d|\nabla f^2|_d+2|\nabla u|^3_d|\nabla \xi|_d+C_5|\nabla u|^2_d+C_2+\lambda r M_1 +\bar R_n\big].
\end{align*}
Using $2\xi|\nabla u|_d|\nabla f^2|_d\le |\nabla u|^2_d+|\nabla f^2|^2_d $ and $|\nabla f^2|_d\le \|\nabla f^2\|_{0} $, the above inequality leads to 
\[
\big(\lambda-1-2|\nabla \xi|_d|\nabla u|_d-C_5\big)|\nabla u|^2_d\le \|\nabla f^2\|_{0}^2+C_2+\lambda r M_1+\bar R_n. 
\]
Then, recalling that $|\nabla \xi|^2_d\le C_0\xi $ (see \eqref{rem:cutoffXI}) and setting $\bar \lambda=2+2\sqrt{C_0}\big\|\sqrt{\xi}|\nabla u|_d\big\|_{0}+C_5$, we obtain
\begin{align}\label{eq:b003}
|\nabla u|^2_d(t^{\bar \lambda}_n,x^{\bar \lambda}_n)\le \|\nabla f\|_{0}^2+C_2+\bar \lambda r M_1+\bar R_n
\end{align}

The parameter $\bar\lambda$ is bounded from above independently of $\nabla u$ as follows. Let $c>0$ be a constant that varies from one expression to the next, independent of $\bar \lambda$, $\eps$, $\delta$, but depending on $d$ and the $C^0(\overline\cO_{m_0+1})$-norms of $b$, $\sigma$, $g$, $h$, $f^2$, and their spatial gradient. From \eqref{eq:vlambda}, \eqref{eq:vlgrdu} and \eqref{eq:vlambda2}, we obtain
\begin{align*}
\big\|\xi|\nabla u|^2_d\big\|_0\le |\nabla u(t^{\bar \lambda}_n,x^{\bar \lambda}_n)|^2_d\!+\!\bar \lambda M_1\!+\!\eta\le c(1\!+\!\bar \lambda\! +\!\bar R_n)\!+\!\eta\le c(1\!+\!\big\|\sqrt{\xi}|\nabla u|_d\big\|_{0} \!+\!\bar R_n)\!+\!\eta, 
\end{align*}
where the second inequality uses \eqref{eq:b003} and the final one the definition of $\bar \lambda$.
Since $\|\xi|\nabla u|^2_d\|_0=\|\sqrt{\xi}|\nabla u|_d\|_0^2$, then $\|\sqrt{\xi}|\nabla u|_d\|_0\!\le\! \max\{1,c(1\!+\!\bar R_n)\}\!+\!\eta$.
As $n\uparrow\infty$ and $\eta\downarrow 0$ we get \eqref{eq:gradbndU} from \eqref{eq:vlgrdu} and \eqref{eq:b003}, choosing $\lambda_0=2+2(1+ c) \sqrt{C_0}+C_5$, thanks to the bound on $\bar \lambda$.
\end{proof}

The bound in Proposition \ref{prp:hldrnrme} depends on $m$. The next result instead provides a uniform bound on any compact $\overline\cO_{m_0}$ for $m_0<m$. This can be achieved thanks to \eqref{eq:gradbndU}.
\begin{proposition}\label{prp:W12pboundtemp}
Fix $m_0\in\N$ and $q\geq m_0+3$ and let $u^{\varepsilon,\delta}_q $ be the unique solution of Problem \ref{prb:penprob} on $\cO_q$. For any $p\in(d+2,\infty)$ and $\beta=1-(d+2)/p$ there is $N_2=N_2(m_0,\varepsilon,\delta,p)$ such that
\begin{align}\label{eq:C01bound}
\|u^{\varepsilon,\delta}_q\|_{W^{1,2,p}(\cO_{m_0})}+\|u^{\varepsilon,\delta}_q \|_{C^{0,1,\beta}(\overline{\cO}_{m_0})}\leq N_2.
\end{align} 
\end{proposition}
\begin{proof}
Define $\varphi(t,x)\coloneqq \xi_{m_0}(x)u^{\varepsilon,\delta}_q(t,x)$. Since $u^{\varepsilon,\delta}_q$ solves \eqref{eq:penprob} and $f_q=f$, $g_q=g$ and $h_q=h$ on $\cO_{m_0+1}$, then $\varphi$ solves
\[
\partial_t\varphi+\cL\varphi-r\varphi = \xi_{m_0}\big[-h-\tfrac{1}{\delta}(g-u^{\varepsilon,\delta}_q)^++\psi_\varepsilon\big(|\nabla u^{\varepsilon,\delta}_q|^2_d-f^2\big)\big] + Q,\quad\text{ on }\cO_{m_0+1},
\]
where $Q(t,x)=u^{\varepsilon,\delta}_q(t,x)(\cL\xi_{m_0})(x)+2\langle a(x)\nabla\xi_{m_0}(x),\nabla u^{\eps,\delta}_q(t,x)\rangle$, and with boundary conditions $\varphi(t,x)=0$ for $x\in\partial B_{m_0+1}$ and $\varphi(T,x)=\xi_{m_0}(x)g(T,x)$ for $x\in B_{m_0+1}$. As in \eqref{eq:W12pnorm} we have
\begin{align}\label{eq:W12norm2}
&\|\varphi \|_{W^{1,2,p}(\cO_{m_0+1})}\\
&\leq C\Big(\big\|\xi_{m_0}\big[h\!+\!\tfrac{1}{\delta}{(g\!-\!u^{\eps,\delta}_q)^+}\!-\!\psi_\varepsilon(|\nabla u^{\varepsilon,\delta}_q |_d^2\!-\!f^2)\big]\!+\!Q\big\|_{L^p(\cO_{m_0+1})}\!+\!\|\xi_{m_0}g\|_{W^{1,2,p}(\cO_{m_0+1})}\Big)\notag
\end{align}
for any $p\in(1,\infty)$ and with $C>0$ independent of $\varepsilon$, $\delta$ and $q$. Denoting $| \cO_{m_0+1}|$ the volume of $\cO_{m_0+1}$ and using Proposition \ref{prop:gradbndU} we obtain 
\begin{align}\label{eq:W12pu}
&\|u^{\varepsilon,\delta}_q\|_{W^{1,2,p}(\cO_{m_0})}\le \|\varphi \|_{W^{1,2,p}(\cO_{m_0+1})}\\
& \leq C|\cO_{m_0}|^{\frac{1}{p}}\left( \| h+\tfrac{1}{\delta}g +Q\|_{C^0(\overline\cO_{m_0+1})}+\tfrac{1}{\varepsilon}(N_1)^2\right)+C\|\xi_{m_0}g\|_{W^{1,2,p}(\cO_{m_0+1})},\notag
\end{align}
where the first inequality is due to $u^{\eps,\delta}_{q}=\varphi$ on $\cO_{m_0}$. Since $Q$ is bounded on $\overline \cO_{m_0+1}$ independently of $q$, the estimate above and Sobolev embedding \eqref{eq:inclW12C01} give us \eqref{eq:C01bound}.
\end{proof}

\subsection{Penalised problem on unbounded domain}

Combining the results obtained so far we can prove existence and uniqueness of the solution to a penalised problem on $\R^{d+1}_{0,T}$. 

\begin{problem}\label{prb:penprobRd}
Find $u=u^{\eps,\delta}$ with $u^{\varepsilon,\delta}\in (C^{1,2,\alpha}_{Loc}\cap W^{1,2,p}_{\ell oc})(\R^{d+1}_{0,T})$, for any $p\in(1,\infty)$ and $\alpha\in(0,1)$ as in Assumption \ref{ass:gen2}, that solves:
\begin{align}\label{eq:pdeinRdeps}
\begin{cases}\partial_tu+\mathcal{L}u-ru=-h-\frac{1}{\delta}\left( g-u\right)^++\psi_\varepsilon\left(|\nabla u|_d^2-f^2\right), &\text{on $[0,T)\!\times\!\R^d$}, \\
u(T,x)=g(T,x), &\text{for all $x\in\R^d$}.
\end{cases}
\end{align}\hfill $\blacksquare$
\end{problem}

\begin{theorem}\label{thm:highreguued}
There exists a solution $u^{\varepsilon,\delta}$ of Problem \ref{prb:penprobRd}.
\end{theorem}
\begin{proof}
Fix $n$ and take $m>n+3$. From Proposition \ref{prp:W12pboundtemp} we know that for any $\beta\in(0,1)$ and $p\in(1,\infty)$, the norms $\|u^{\varepsilon,\delta}_m\|_{W^{1,2,p}(\cO_{n})}$ and $\|u^{\varepsilon,\delta}_m\|_{C^{0,1,\beta}(\overline{\cO}_{n})}$ are bounded by a constant independent of $m$. By weak compactness in $W^{1,2,p}$ and Ascoli-Arzel\'a's theorem we can then extract a sequence $(u^{\eps,\delta}_{m^n_k})_{k\in\N}$ and there exists a function $u^{\eps,\delta;n}\in W^{1,2,p}(\cO_{n})$ (both possibly depending on the choice of $\cO_n$) such that, as $k\to\infty$ (and $m^n_k\to\infty$), we obtain
\begin{align}\label{eq:weakconv}
\begin{array}{l}
u^{\varepsilon,\delta}_{m^n_k} \to u^{\varepsilon,\delta;n}\quad\text{and}\quad\nabla u^{\varepsilon,\delta}_{m^n_k} \to \nabla u^{\varepsilon,\delta;n}\quad\text{ in } {C^{\alpha}(\overline{\cO}_{n})}, \\[+5pt]
\partial_t u^{\varepsilon,\delta}_{m^n_k} \to \partial_t u^{\varepsilon,\delta;n}\quad\text{and}\quad D^2u^{\varepsilon,\delta}_{m^n_k} \to D^2u^{\varepsilon,\delta;n}\quad\text{weakly in }L^p(\cO_{n}).
\end{array}
\end{align}
Since the sequence $(u^{\eps,\delta}_{m^n_k})_{k\in\N}$ is bounded also in the $W^{1,2,p}(\cO_{n+1})$-norm (perhaps by a larger constant), then up to selecting a further subsequence we have convergence as in \eqref{eq:weakconv} but with $n$ replaced by $n+1$.
Therefore $u^{\eps,\delta;n}=u^{\eps,\delta;n+1}$ on $\cO_{n}$. Using that $\overline\cO_n\uparrow \R^{d+1}_{0,T}$ as $n\to\infty$ and iterating the extraction of further subsequences (if needed), we can uniquely define a limit function $u^{\varepsilon,\delta}\in (C^{0,1,\alpha}_{\ell oc}\cap W^{1,2,p}_{\ell oc})(\R^{d+1}_{0,T})$. 

Fix $n$ and take $m>n$. Multiply the PDE solved by $u^{\eps,\delta}_m$ (see \eqref{eq:penprob}) by a test function supported on $\overline \cO_n$. Then passing to the limit along the subsequence constructed above, it is standard procedure to show that $u^{\eps,\delta}$ satisfies the first equation in \eqref{eq:pdeinRdeps} in the a.e.\ sense on $\cO_n$ thanks to locally uniform convergence on compacts of $u^{\eps,\delta}_m$ and $\nabla u^{\eps,\delta}_m$ and the weak convergence of $\partial_t u^{\eps,\delta}_m$ and $D^2 u^{\eps,\delta}_m$. Since, this can be done for any $\cO_n$ and $u^{\eps,\delta}(T,\cdot)=g(T,\cdot)$, then $u^{\eps,\delta}$ solves \eqref{eq:pdeinRdeps} in the a.e.\ sense (it is a {\em strong} solution). 
It now remains to prove it is actually a classical solution. 

Fix an arbitrary open bounded domain $\cO\subset\R^{d+1}_{0,T}$ with smooth parabolic boundary $\partial_P\cO$. Let $v\in C^{1,2,\alpha}_{Loc}(\cO)$ be the unique classical solution of the boundary value problem  
\begin{align}\label{eq:strnsolpde2}
\begin{cases}
\partial_tv+\mathcal{L}v-rv=-h-\tfrac{1}{\delta}\left(g-u^{\varepsilon,\delta}\right)^++\psi_\varepsilon\left(|\nabla u^{\varepsilon,\delta}|_d^2-f^2\right), &\text{ on } \cO, \\
v(t,x)=u^{\eps,\delta}(t,x), &\text{for $(t,x)\in\partial_P\cO$}.
\end{cases}
\end{align}
Existence and uniqueness of such $v$ is guaranteed by \cite[Thm.\ 3.4.9]{friedman2008partial} because 
\[
-h-\tfrac{1}{\delta}\big(g-u^{\varepsilon,\delta}\big)^++\psi_\varepsilon\big(|\nabla u^{\varepsilon,\delta}|_d^2-f^2\big)\in C^{\alpha}(\overline\cO),
\] 
and $\cL$ is uniformly elliptic on $\overline{\cO}$ with continuously differentiable coefficients (Assumption \ref{ass:gen1}). Since $v$ is also a strong solution, then $v-u^{\eps,\delta}\in W^{1,2,p}(\cO)$ is a strong solution of $\partial_tw+\cL w-rw=0$ in $\cO$ with $w=0$ on $\partial_P\cO$. It follows that $\|v-u^{\eps,\delta}\|_{W^{1,2,p}(\cO)}=0$ by the same estimate as in \eqref{eq:W12pnorm}. By arbitrariness of $\cO$ we can choose a $C^{1,2,\alpha}_{L oc}$-representative of $u^{\eps,\delta}$, as claimed.
\end{proof}

We now give a probabilistic representation for  $u^{\eps,\delta}$ analogue of \eqref{eq:probrap} but on unbounded domain. 
For $(n,\nu)\in\cA^\circ_t$ and $w\in\cT^\delta_t$ let us denote by $\cJ^{\eps,\delta}_{t,x}(n,\nu,w)$ a payoff analogue of \eqref{eq:Jpen} but with $\rho_{m}$, {$g_m$, $h_m$} replaced by $T-t$, {$g$, $h$, respectively,} and with the Hamiltonian $H^\eps_m$ replaced by
\begin{align}\label{eq:Heps}
H^{\varepsilon}(t,x,y)\coloneqq \sup_{p\in\mathbb{R}^d}\left\{\langle y, p\rangle-\psi_\eps\left(|p|^2_d- f^2(t,x)\right)\right\}.
\end{align}
Notice that
\begin{align}\label{rmk:Hnoninc}
t\mapsto H^{\varepsilon}(t,x,y)\:\text{is non-increasing for all $(x,y)\in\R^d\times\R^d$},
\end{align}
because $t\mapsto f(t,x)$ is non-increasing by Assumption \ref{ass:gen2}. Moreover, taking $p=\eps y/2$ in $H^{\varepsilon}$ gives
\begin{align}\label{eq:lowbnHeps}
H^{\varepsilon}(t,x,y)\geq \frac{\eps}{2}|y|_d^2-\psi_\eps\left({\tfrac{\eps^2}{4}|y|^2_d-f^2(t,x)}\right)\geq \frac{\eps}{2}|y|_d^2-\psi_\eps\left({\tfrac{\eps^2}{4}|y|^2_d}\right)\geq \frac{\eps}{4}|y|_d^2.
\end{align}
\begin{proposition}\label{lem:prbraprRD}
Let $u^{\varepsilon,\delta}$ be a solution of Problem  \ref{prb:penprobRd}. Then 
\begin{align}\label{eq:valueunbdd}
u^{\varepsilon,\delta}(t,x)=\inf_{(n,\nu)\in\cA^{\circ}_t}\sup_{w\in\cT^\delta_t}\cJ^{\varepsilon,\delta}_{t,x}(n,\nu,w)=\sup_{w\in \cT^\delta_t}\inf_{(n,\nu)\in\cA^{\circ}_t}\cJ^{\varepsilon,\delta}_{t,x}(n,\nu,w).
\end{align}
\end{proposition}
\begin{proof}
Fix $[(n,\nu),w]\in\cA^\circ_t\times\cT^\delta_t$. An application of Dynkin's formula to 
$R^w_{\rho_m}u^{\eps,\delta}(t+\rho_{m},X^{[n,\nu]}_{\rho_{m}})$
combined with \eqref{eq:pdeinRdeps} gives (recall $u^{\eps,\delta}\in (W^{1,2,p}_{\ell oc}\cap C^{0,1,\beta}_{\ell oc})(\R^{d+1}_{0,T})$)
\begin{align}\label{eq:penver}
u^{\eps,\delta}(t,x)=\E_x\biggr[&\, R^w_{\rho_{m}}u^{\eps,\delta}\big(t+\rho_{m},X^{[n,\nu]}_{\rho_{m}}\big)\\
&+\!\int_0^{\rho_{m}}\!R^w_s\Big(h\!+\!\tfrac{1}{\delta}\big(g\!-\!u^{\eps,\delta}\big)^+\!-\!\psi_\eps\big(|\nabla u^{\eps,\delta}|^2_d\!-\!f^2\big) \Big)\big(t\!+\!s,X^{[n,\nu]}_s\big)\,\ud s\notag \\
&\,+\!\int_0^{\rho_{m}}\!R^w_s\Big(w_su^{\eps,\delta}\big(t\!+\!s,X_s^{[n,\nu]}\big)\!-\!\big\langle n_s\dot{\nu}_s,\nabla u^{\eps,\delta}\big(t\!+\!s,X_s^{[n,\nu]}\big)\big\rangle\Big)\,\ud s\biggr].\notag
\end{align}
By definition of the Hamiltonian $H^\eps$ in \eqref{eq:Heps} we have
\begin{align}\label{eq:penver3}
u^{\eps,\delta}(t,x)\le&\,\E_x\Big[R^w_{\rho_{m}}u^{\eps,\delta}\big(t+\rho_{m},X^{[n,\nu]}_{\rho_{m}}\big)\Big]\\
&\,+\E_x\bigg[\int_0^{\rho_{m}}\!R^w_s\Big(h\!+\!\tfrac{1}{\delta}\big(g\!-\!u^{\eps,\delta}\big)^+\!+\!w_s u^{\eps,\delta}\!+\!H^\eps(\cdot,n_s\dot{\nu}_s)\Big)\big(t\!+\!s,X^{[n,\nu]}_s\big)\,\ud s\biggr].\notag
\end{align}
Letting $m\uparrow \infty$ we have $\rho_{m}\uparrow T-t$, $\P_x$-a.s. We can take the limit inside the second expectation by monotone convergence as all the terms under the integral are positive. By Lemma \ref{lem:polygrow} the term under the first expectation has quadratic growth in $X^{[n,\nu]}$. Thanks to standard estimates for SDEs (see \cite[Thm.\ 2.5.10]{krylov1980controlled}) there is a constant $c>0$ independent of $m$ such that
\[
\E_x\Big[\sup_{0\le s\le T-t}\big|X^{[n,\nu]}_s\big|^2_d\Big]\le c\big(1+|x|^2_d+\E_x[|\nu_{T-t}|^2]\big).
\] 
Then, dominated convergence and $u^{\eps,\delta}(T,\,\cdot\,)=g(T,\,\cdot\,)$ give us
\begin{align}\label{eq:penver2}
u^{\eps,\delta}(t,x)\!\le\! \E_x\bigg[R^w_{T-t}g\big(T,X^{[n,\nu]}_{T-t}\big)\!+\!\int_0^{T-t}\!\!\!\!R^w_s\Big(h\!+\!\tfrac{1}{\delta}\big(g\!-\!u^{\eps,\delta}\big)^+\!+\!w_s u^{\eps,\delta}\!+\!H^\eps(\cdot,n_s\dot{\nu}_s)\Big)\big(t\!+\!s,X^{[n,\nu]}_s\big)\,\ud s\bigg].
\end{align}
By arguments as in the proof of Proposition \ref{prp:probrap1}, with $w^*\in \cT^{\delta}_t$ defined as in \eqref{eq:dfnoptw*} but with $u$ and $g_m$ replaced by $u^{\eps,\delta}$ and $g$, respectively, we obtain $u^{\eps,\delta}(t,x)\le \cJ_{t,x}^{\varepsilon,\delta}(n,\nu,w^*)$. Therefore 
\begin{align}\label{eq:late0}
u^{\eps,\delta}(t,x)\leq \sup_{w\in\cT^\delta_t} \inf_{(n,\nu)\in\cA^{\circ}_t}\cJ_{t,x}^{\varepsilon,\delta}(n,\nu,w).
\end{align}

As in Proposition \ref{prp:probrap1}, for the reverse inequality 
we set $X^*=X^{[n^*,\nu^*]}$ and denote
\begin{equation}\label{eq:optcntr2}
\begin{split}
n_s^*\coloneqq &\!\begin{cases} -\frac{\nabla u^{\eps,\delta}(t+s,X_s^*)}{|\nabla u^{\eps,\delta}(t+s,X_s^*)|_d}, &\quad\text{if } \nabla u^{\eps,\delta}(t+s,X_s^*)\neq \mathbf{0},\\[+5pt]
\text{{any unit vector}}, &\quad\text{if }\nabla u^{\eps,\delta}(t+s,X_s^*)=\mathbf{0},
\end{cases}\\
\dot{\nu}_s^*\coloneqq &\, 2\psi_{\varepsilon}'\left(|\nabla u^{\eps,\delta}(t+s,X_s^*)|_d^2- f^2(t+s,X_s^*)\right)|\nabla u^{\eps,\delta}(t+s,X_s^*)|_d.
\end{split}
\end{equation}
We claim here and will prove later that $(n^*,\nu^*)\in\cA^\circ_t$ and the SDE for $X^*$ admits a unique non-exploding strong solution. 
For $(n^*,\nu^*)$ equality holds in \eqref{eq:penver3}. As $m\uparrow \infty$ Fatou's lemma gives $u^{\eps,\delta}(t,x)\ge \cJ_{t,x}(n^*,\nu^*,w)$, hence
\begin{align}\label{eq:late1}
u^{\eps,\delta}(t,x)\ge \inf_{(n,\nu)\in\cA^{\circ}_t}\sup_{w\in\cT^\delta_t}\cJ_{t,x}(n,\nu,w).
\end{align}
Combining \eqref{eq:late0} and \eqref{eq:late1} we conclude.

It remains to check that $(n^*,\nu^*)\in\cA^\circ_t$ and $X^*$ is non-exploding. We use an argument from \cite[Lemma 13.7]{soner1989regularity}. Let $\zeta_m=\inf\{s\ge0: |X^*_s|_d\ge m\}$. On the random time-interval $[0,\zeta_m\wedge(T\!-\!t)]$ the process $X^*$ is well-defined and the pair $(n^*,\nu^*)$ is adapted because $u^{\eps,\delta}\in C^{1,2,\alpha}_{Loc}(\cO_m)\cap C^{0,1,\alpha}(\overline\cO_m)$.  Notice that $\zeta_{k}\le \zeta_{k+1}$ and it may occur $\zeta_\infty:=\lim_{k\to\infty}\zeta_k< T-t$ with positive probability. {Moreover $\rho_m=\zeta_m\wedge(T\!-\!t)$ in \eqref{eq:penver} and let us} take $w\equiv0$ therein. By construction, for $s\in[0,\zeta_m\wedge(T\!-\!t)]$
\begin{align*}
-\big\langle n^*_s\dot{\nu^*}_s,\nabla u^{\eps,\delta}(t+s,X_s^{*})\big\rangle-\psi_{\varepsilon}\left(|\nabla u^{\eps,\delta}|_d^2- f^2\right)(t+s,X_s^{*})= H^{\eps}(t+s,X_s^{*},n^*_s\dot{\nu^*}_s).
\end{align*}
Then, by positivity of all remaining terms in \eqref{eq:penver} 
\[
u^{\eps,\delta}(t,x)\ge \E_{x}\Big[\int_0^{{\zeta_m\wedge(T-t)}}e^{-rs}H^{\eps}(t+s,X_s^{*},n^*_s\dot{\nu}^*_s)\ud s\Big].
\]
By positivity of $H^\eps$ and monotone convergence, we can let $m\uparrow \infty$ and preserve the inequality while the integral in time extends to $\zeta_\infty\wedge (T-t)$. Combining with \eqref{eq:lowbnHeps} and Lemma \ref{lem:polygrow} we have
\begin{align}\label{eq:soner}
\frac{\eps}{4}\E_{x}\Big[\int_0^{\zeta_\infty\wedge (T-t)}e^{-rs}|\dot{\nu}^*_s|^2\,\ud s\Big]\le u^{\eps,\delta}(t,x)\le K_3(1+|x|^2_d).
\end{align}

Since $\dot{\nu}^*\ge 0$ then $s\mapsto\nu^*_s$ is non-decreasing and
\begin{align}
|\nu^*_{\zeta_m\wedge (T-t)}|^2=2\int_0^{\zeta_m\wedge (T-t)}\nu_s^*\dot{\nu}_s^*\,\ud s\leq \int_0^{\zeta_m\wedge (T-t)}|\nu_s^*|^2\,\ud s+\int_0^{\zeta_m\wedge (T-t)}|\dot{\nu}_s^*|^2\,\ud s,
\end{align}
where we used $2ab\le a^2+b^2$. By Gronwall's lemma and taking expectations we obtain 
\begin{align}
\E_x\big[|\nu^*_{\zeta_m\wedge (T-t)}|^2\big]\leq  e^T\E_x\Big[\int_0^{\zeta_m\wedge (T-t)}|\dot{\nu}_s^*|^2\,\ud s\Big]\leq e^{T(1+r)}\E_x\Big[\int_0^{\zeta_m\wedge (T-t)} e^{-rs}|\dot{\nu}_s^*|^2\,\ud s\Big].
\end{align}
Combining with \eqref{eq:soner} and letting $m\to\infty$, Fatou's lemma gives us 
\begin{align}\label{eq:Zinftybnd}
\E_x\big[|\nu^*_{\zeta_\infty\wedge (T-t)}|^2\big]\le 4e^{T(1+r)}\eps^{-1}K_3(1+|x|^2_d).
\end{align}
Linear growth of $(b,\sigma)$ and well-posedness of $X^*_{s\wedge\zeta_m}$ give, by Markov inequality and standard bounds,
\begin{align}
\P_x(\zeta_m < T\!-\!t)\!\le& \frac{1}{m^2}\E_x\Big[\sup_{s\in[0,\zeta_m\wedge (T-t)]}|X^{*}_s|^2_d\Big]\\
\leq&\frac{C}{m^2}\Big(1\!+\!|x|^2_d\!+\!\E_x\big[\big|\nu_{\zeta_m\wedge (T-t)}^*\big|^2\big]\Big)\le \frac{C}{m^2}c(\eps)(1\!+\!|x|^2_d),
\end{align}
where $C>0$ depends only on $T$ and $D_1$ from Assumption \ref{ass:gen1}, and $c(\eps)$ depends on the constants from \eqref{eq:Zinftybnd}. Since $\zeta_m\uparrow\zeta_\infty$, then $\P_x(\zeta_m<T-t)\downarrow \P_x(\zeta_\infty\le T-t)$ as $m\to\infty$ and by taking limits in the expression above we conclude $\P_x(\zeta_\infty\le T-t)=0$. Thus, $X^*_{s}$ is well-defined for all $s\in[0,T-t]$ and $\E_x[|\nu^*_{T-t}|^2]<\infty$, by \eqref{eq:Zinftybnd} implying $(n^*,\nu^*)\in\cA^\circ_t$ as claimed.
\end{proof}

Proposition \ref{lem:prbraprRD} implies that Problem \ref{prb:penprobRd} admits a {\em unique} solution and that the triple $[(n^*,\nu^*), w^*]$ is optimal in \eqref{eq:valueunbdd}.
By arguments as in the proof of Proposition \ref{cor:prbrap2edm} we also obtain the next result.
\begin{proposition}\label{cor:prbrap2ed}
Let $u^{\varepsilon,\delta}$ be the unique solution of Problem \ref{prb:penprobRd}. Then 
\begin{align}\label{eq:prbrap2d12}
u^{\varepsilon,\delta}(t,x)=\inf_{(n,\nu)\in\cA^{\circ}_t}\E_{x}\biggr[& e^{-(r+\delta^{-1})(T-t)}g(T,X_{T-t}^{[n,\nu]})\notag\\
&+\!\int_0^{T-t}e^{-(r+\delta^{-1})s}\big[h+\tfrac{1}{\delta}g\vee u^{\varepsilon,\delta}+H^{\varepsilon}(\,\cdot\,,n_s\dot{\nu}_s)\big](t+s,X^{[n,\nu]}_s)\,\ud s\biggr],
\end{align}
and the pair $(n^*,\nu^*)$ from \eqref{eq:optcntr2} is optimal.
\end{proposition}

\subsection{Refined estimates independent of $\eps$ and $\delta$} 
Here we develop bounds for the penalty terms in the PDE of Problem \ref{prb:penprobRd} which are independent of $\eps$ and $\delta$. 
\begin{lemma}\label{lem:bndobstpen2}
For $K_2$ as in \eqref{defn:infTheta} we have 
\begin{align}\label{eq:obstcpenbnd1}
\tfrac{1}{\delta}\big\|\big(g-u^{\varepsilon,\delta}\big)^+\big\|_{\infty}\le K_2.
\end{align}
\end{lemma}
\begin{proof}
For any $(n,\nu)\in\cA^\circ_t$, the function $g$ has the same probabilistic representation as in \eqref{eq:gmprobrep} but with $g_m$, $\rho_{m}$ and $w$ replaced by $g$, $T-t$ and $1/\delta$, respectively. Combining that with the expression for $u^{\eps,\delta}$ in Proposition \ref{cor:prbrap2ed}, and recalling $\Theta$ defined in \eqref{defn:infTheta}, we get
\begin{align*}
&u^{\eps,\delta}(t,x)-g(t,x)\\
&=\inf_{(n,\nu)\in \cA^{\circ}_t}\E_{x}\biggr[\int_0^{T-t} \!\!e^{-(r+\delta^{-1})s}\Big[\Theta+\langle n_s\dot{\nu}_s,\nabla g\rangle\!+\!\tfrac{1}{\delta}g\vee u^{\eps,\delta}\!-\!\tfrac{1}{\delta}g\!+\! H^{\varepsilon}(\,\cdot\,,n_s\dot{\nu}_s)\Big](t+s,X_s^{[n,\nu]})\,\ud s\biggr].\notag
\end{align*}
As in \eqref{rem:trickhamg}, $[\langle n_s\dot{\nu}_s, \nabla g\rangle+H^{\varepsilon}(\cdot,n_s\dot{\nu}_s)](t+s,X_s^{[n,\nu]}) \geq0$, and observing that $g\vee u^{\eps,\delta}-g\geq0$ we get
\begin{align*}
&u^{\eps,\delta}(t,x)-g(t,x)\geq \inf_{(n,\nu)\in \cA^{\circ}_t}\E_{x}\biggr[\int_0^{T-t}\!\! e^{-(r+\delta^{-1})s} \Theta\big(t+s,X_s^{[n,\nu]}\big)\,\ud s\biggr]\ge -K_2\frac{\delta}{r\delta+1},
\end{align*}
where $K_2$ was defined in \eqref{defn:infTheta}. The above implies $\tfrac{1}{\delta}(g(t,x)-u^{\eps,\delta}(t,x))^+\leq K_2$ as needed.
\end{proof}
Next we give an upper bound on $\partial_tu^{\varepsilon,\delta}$. In the lemma below we understand 
\begin{align}\label{eq:lftderivT}
\partial_t u^{\varepsilon,\delta}(T,x)\coloneqq \lim_{s\to 0}\frac{u^{\varepsilon,\delta}(T,x)-u^{\varepsilon,\delta}(T-s,x)}{s}.
\end{align}
\begin{lemma}\label{lem:bndtimder}
There is $K_4>0$ only depending on $K_0$ and $K_2$ from Assumption \ref{ass:gen2} such that
\begin{align}\label{eq:tder}
\partial_t u^{\varepsilon,\delta}(t,x) \leq K_4,\quad \text{ for }(t,x)\in \R^{d+1}_{0,T}.
\end{align}
\end{lemma}
\begin{proof}
Let $u=u^{\varepsilon,\delta}$ for simplicity and take $T\ge t_2>t_1\ge 0$. {Let $w^{(2)}\in \cT^\delta_{t_2}$ be optimal} for the value function $u(t_2,x)$ and let $(n^{(1)},\nu^{(1)})\in\cA^\circ_{t_1}$ be {optimal} for the value function $u(t_1,x)$. Set $X^{(1)}=X^{[n^{(1)},\nu^{(1)}]}$ and notice that $w_s^{(1)}\coloneqq w^{(2)}_s\mathds{1}_{\{s\leq T-t_2\}}$ lies in $\cT^\delta_{t_1}$ and $(n^{(1)},\nu^{(1)})$ restricted to $[0,T-t_2]$ lies in $\cA^\circ_{t_2}$. To simplify notation let us also set $[\Delta_{t_2,t_1}g](s,x)=g(t_2+s,x)-g(t_1+s,x)$ and analogously for $[\Delta_{t_2,t_1}h](s,x)$ and $[\Delta_{t_2,t_1}H^\eps](s,x,y)$. Then
\begin{align}\label{eq:DD}
&u(t_2,x)\!-\!u(t_1,x)\!\leq\! \cJ_{t_2,x}^{\varepsilon,\delta}(n^{(1)},\nu^{(1)},w^{(2)})\!-\!\cJ_{t_1,x}^{\varepsilon,\delta}(n^{(1)},\nu^{(1)},w^{(1)}) \notag\\
&\le \E_x\biggr[R^{w^{(2)}}_{T-t_2}g(T,X^{(1)}_{T-t_2})\!-\!R^{w^{(1)}}_{T-t_1}g(T,X^{(1)}_{T-t_1})\!-\!\int_{T-t_2}^{T-t_1}\!\!R^{w^{(1)}}_s\Big(h\! +\! H^{\varepsilon}(\cdot, n_s^{(1)}\dot{\nu}^{(1)}_s) \Big)(t_1\!+\!s,X^{(1)}_s)\,\ud s\notag\\
&\qquad\:\:+\!\int_{0}^{T-t_2}\!\!R^{w^{(2)}}_s\Big([\Delta_{t_2,t_1}h]\!+\![\Delta_{t_2,t_1}H^\eps](\cdot,n_s^{(1)}\dot{\nu}^{(1)}_s)\!+\!w^{(2)}_s[\Delta_{t_2,t_1}g]\Big)(s,X^{(1)}_s)\ud s\biggr].
\end{align}
By \eqref{rmk:Hnoninc} the Hamiltonian $H^{\varepsilon}$ is non-increasing in time, so $[\Delta_{t_2,t_1}H^\eps]\le 0$. Next, we apply Dynkin's formula to $R^{w^{(1)}}_{T-t_1}g(T,X^{(1)}_{T-t_1})$ on the time interval $[T-t_2,T-t_1]$, to obtain
\begin{align}\label{eq:Dg12}
\E_x\Big[R^{w^{(1)}}_{T-t_1}g(T,X^{(1)}_{T-t_1})\Big]=&\,\E_x\bigg[R^{w^{(1)}}_{T-t_2}g(T\!-\!(t_2\!-\!t_1),X^{(1)}_{T-t_2})\notag\\
&\quad\quad+\int_{T-t_2}^{T-t_1}\!\!R^{w^{(1)}}_s\big(\partial_t g\!+\!\cL g\!-\!rg\!+\langle\nabla g,n^{(1)}_s\dot{\nu}^{(1)}_s\rangle\big)(t_1\!+\!s,X^{(1)}_s)\ud s\biggr].
\end{align}
Let us plug \eqref{eq:Dg12} into \eqref{eq:DD}, recall that $R^{w^{(1)}}_{T-t_2}=R^{w^{(2)}}_{T-t_2}$ and use \eqref{rem:trickhamg} with $f,g$ replacing $f_m,g_m$:
\begin{align*}
&u(t_2,x)-u(t_1,x)\\
&\leq \E_x\biggr[R^{w^{(2)}}_{T-t_2}\left(g(T,X^{(1)}_{T-t_2})\!-\!g(T\!-\!(t_2\!-\!t_1),X^{(1)}_{T-t_2})\right)\!-\!\int_{T-t_2}^{T-t_1}\!R^{w^{(1)}}_s\Theta(t_1\!+\!s, X^{(1)}_s) \,\ud s\\
&\quad\qquad+\int_{0}^{T-t_2}R^{w^{(2)}}_s\Big([\Delta_{t_2,t_1}h](s,X^{(1)}_s)+w^{(2)}_s [\Delta_{t_2,t_1}g](s,X^{(1)}_s)\Big)\,\ud s\biggr],
\end{align*}
where we recall $\Theta=\partial_t g+\cL g-rg+h$.
Thanks to condition \eqref{eq:lipstimehg} on $h$ and $g$ and \eqref{defn:infTheta} on $\Theta$
\begin{align*}
u(t_2,x)-u(t_1,x)\leq \big(K_0(1+T)+K_2\big)(t_2-t_1),
\end{align*}
by evaluating explicitly $\int_0^{T-t_2}w^{(2)}_s R^{w^{(2)}}_s\ud s$.
Then, the claim holds with $K_4= K_0(1+T)+K_2$. 
\end{proof}

Our next goal is to find a uniform bound for the penalty term involving $\psi_\eps$. Here we are inspired by estimates obtained under different assumptions in \cite{zhu1992generalized} (see also \cite{hynd2010partial} and \cite{kelbert2019hjb}).
\begin{lemma}\label{thm:gradpenbnd}
There is $M_5=M_5(m)$, independent of $\eps$ and $\delta$, such that
\begin{align}\label{eq:gradpeneq}
\big\|\psi_\varepsilon\big(|\nabla u^{\varepsilon,\delta}|^2_d-f^2\big)\big\|_{C^0(\overline \cO_m)}\le M_5.
\end{align}
\end{lemma}
\begin{proof}
For notational simplicity we set $u=u^{\varepsilon,\delta}$ and $\xi=\xi_m$. Let 
\begin{align}\label{eq:defnvpsi}
v(t,x)\coloneqq \xi(x)\psi_\varepsilon\!\left(|\nabla u(t,x)|^2_d-f^2(t,x)\right),\quad\text{ for }(t,x)\in \overline{\cO}_{m+1}.
\end{align}
Since $v$ is continuous, then it attains a maximum on $\overline \cO_{m+1}$. If such maximum is attained at a point $(t^*,x^*)\in\partial_P \cO_{m+1}$ then $v(t^*,x^*)=0$ because either $x^*\in \partial B_{m+1}$ and $\xi(x^*)=0$ or $t^*=T$ and $|\nabla u(t^*,x^*)|_d=|\nabla g(T,x^*)|_d\leq f(T,x^*)$ by \eqref{eq:grdg<f}. Thus, suppose the maximum is attained in $\cO_{m+1}$. 

We argue similarly to the proof of Proposition \ref{prop:locgradp}. For any $\eta>0$ there exists a neighbourhood $U_\eta\cup\big(\{0\}\!\times\!V_\eta\big)$ of $(t^*,x^*)$ such that $v(t,x)>v(t^*,x^*)-\eta$ for all $(t,x)\in U_\eta\cup\big(\{0\}\!\times\!V_\eta\big)$. {With no loss of generality, there is $S<T$ and $B$ an open ball with $\overline B\subset B_{m+1}$, so that $U_\eta\cup\big(\{0\}\!\times\!V_\eta\big)\subset \cO_{S,B}$, where $\cO_{S,B}=[0,S)\times B$}.
Let $w^n$ be the solution of a PDE as in \eqref{eq:pdeinRdeps} but with $\nabla u^{\eps,\delta}$ and $(\cdot)^+$ on the right-hand side of that equation replaced by smooth approximations $\nabla u^n$ and $\chi_n$, and the function $f^2$ in the argument of $\psi_\eps$ replaced by $f^2+\frac{1}{n}$. By arguments analogous to those in Lemma \ref{lem:stability}, $w^n\in C^{1,3,\alpha}_{Loc}(\cO_{m+1})$ and $w^n\to u$ in $C^{1,2,\beta}(\overline\cO_{S,B})$ and in $(C^{0,1,\beta}\cap W^{1,2,p})(\overline\cO_m)$ as $n\to \infty$ for all $\beta\in(0,\alpha)$ (see Remark \ref{rem:stab}).
Define
\begin{align}\label{eq:defnxipsi}
v^n(t,x)\coloneqq \xi(x)\psi_\varepsilon\!\left(|\nabla w^n(t,x)|^2_d-f^2(t,x)-\tfrac{1}{n}\right),\quad\text{ for }(t,x)\in \overline{\cO}_{m+1}.
\end{align}
We have that $v^n$ belongs to $C^{1,2,\alpha}_{Loc}(\cO_{m+1})\cap C^{0,1,\alpha}(\overline{\cO}_{m+1})$, it is non-negative and it is equal to zero for $x\in \partial B_{m+1}$. Moreover $v^n\to v$ in $C^{0,1,\gamma}(\overline{\cO}_{m+1})$ for all $\gamma\in(0,\alpha)$. 

Let $(t^*_n,x^*_n)_{n\in\N}$ be such that $(t^*_n,x^*_n)\in\argmax_{\overline{\cO}_{m+1}}v^n $ and, with no loss of generality, assume $(t^*_n,x^*_n)\to(t^*,x^*)$. We also assume $|\nabla g(t^*,x^*)|_d- |\nabla u(t^*,x^*)|_d< 0$,
as otherwise $f(t^*,x^*)\geq |\nabla g(t^*,x^*)|_d\geq |\nabla u(t^*,x^*)|_d$ (cf.\ \eqref{eq:grdg<f}) implies $0\leq v(t,x)\leq v(t^*,x^*)=0$. By uniform convergence of $\nabla w^n$ to $\nabla u$, we can also assume $|\nabla g|_d- |\nabla w^n|_d\leq 0$ on $U_\eta\cup(\{0\}\times V_\eta)$ for all $n\in\N$. 

We denote $\bar \zeta_n\coloneqq (|\nabla w^n|^2_d-f^2-\tfrac{1}{n})(t^*_n,x^*_n)$ and taking derivatives of $v^n$ we obtain
\begin{align}\label{eq:diffvnpnlty}
v^n_t=&\,\xi\psi_\varepsilon'(\bar\zeta_n)\left(|\nabla w^n|^2_d-f^2\right)_t \notag\\
v^n_{x_i}=&\,\xi_{x_i} \psi_\varepsilon(\bar\zeta_n) + \xi\psi_\varepsilon'(\bar\zeta_n)\left(|\nabla w^n|^2_d-f^2\right)_{x_i}\notag \\
v^n_{x_i x_j}=&\,\xi_{x_ix_j} \psi_\varepsilon(\bar\zeta_n) + \psi_\varepsilon'(\bar\zeta_n)\left(\xi_{x_i}\left(|\nabla w^n|^2_d-f^2\right)_{x_j}+\xi_{x_j}\left(|\nabla w^n|^2_d-f^2\right)_{x_i}\right)\\
&\,+\xi\psi_\varepsilon''(\bar\zeta_n)\left(|\nabla w^n|^2_d-f^2\right)_{x_i}\left(|\nabla w^n|^2_d-f^2\right)_{x_j} \notag\\
&\,+\xi\psi_\varepsilon'(\bar\zeta_n)\left(2\langle \nabla w_{x_j}^n,\nabla w_{x_i}^n\rangle+2\langle \nabla w^n,\nabla w_{x_ix_j}^n\rangle -\left(f^2\right)_{x_ix_j}\right).\notag
\end{align}
By \eqref{eq:MaxP} we have $0\geq (\partial_tv^n+\mathcal{L}v^n)(t^*_n,x^*_n)$. Since $(t^*_n,x^*_n)$ is fixed, we omit it from the calculations that follow, for notational simplicity. Then using \eqref{eq:diffvnpnlty} and symmetry of $a_{ij}$ 
\begin{align}\label{eq:1bndpnlty}
0\geq &\,(\mathcal{L}\xi)\psi_\varepsilon(\bar\zeta_n)- \xi\psi_\varepsilon'(\bar\zeta_n)\big(\partial_t(f^2)+\mathcal{L}(f^2)\big)+2\xi \psi_\varepsilon'(\bar\zeta_n)\left\langle\nabla w^n,(\partial_t +\cL)(\nabla w^n)\right\rangle \\[+4pt]
&+\tfrac{1}{2}\xi\psi_\varepsilon''(\bar\zeta_n)\big\langle a\nabla(|\nabla w^n|^2_d-f^2),\nabla(|\nabla w^n|^2_d-f^2)\big\rangle\\
&+\psi_\varepsilon'(\bar\zeta_n)\big\langle a\nabla \xi,\nabla (|\nabla w^n|^2_d-f^2) \big\rangle+\xi\psi_\varepsilon'(\bar\zeta_n)\sum_{i,j=1}^d a_{ij} \big\langle \nabla w^n_{x_i},\nabla w^n_{x_j} \big\rangle,\notag
\end{align}
where $(\partial_t +\cL)(\nabla w^n)$ is the vector with entries $(\partial_t +\cL)w^n_{x_k}$ for $k=1,\ldots,d$.

Recall that $\psi_\eps$ is non-decreasing and convex, so $\psi'_\eps,\psi''_\eps\ge 0$. Then, the last term on the right-hand side above is bounded from below by $\xi\psi'_\eps(\bar\zeta_n)\theta\big|D^2 w^n\big|^2_{d\times d}$ as in \eqref{eq:bndgrad2} with $\theta$ equal to $\theta_{B_{m+1}}$. Uniform ellipticity \eqref{eq:defnthetaEC} on $B_{m+1}$ also gives
\begin{align*}
\tfrac{1}{2}\xi\psi_\varepsilon''(\bar\zeta_n)\big\langle a\nabla(|\nabla w^n|^2_d-f^2),\nabla(|\nabla w^n|^2_d-f^2)\big\rangle\geq 0.
\end{align*}
Set $\bar a_m\coloneqq \max_{i,j}\|a_{ij}\|_{C^0(\overline B_{m+1})}$ and recall $|\nabla \xi|^2_d\le C_0\xi$ (see \eqref{rem:cutoffXI}). Then
\begin{align*}
&\big\langle a\nabla \xi,\nabla (|\nabla w^n|^2_d-f^2) \big\rangle\ge -\bar a_m d^2|\nabla \xi|_d\Big(2|\nabla w^n|_d|D^2 w^n|_{d\times d}+|\nabla f^2|_d\Big)\\
&\geq-\xi\frac{\theta}{4}|D^2w^n|_{d\times d}^2-\frac{16}{\theta}\bar a_m^2 d^4 C_0|\nabla w^n|_d^2 -\bar a_m d^2\sqrt{C_0\xi}|\nabla f^2|_d,\notag
\end{align*}
where we used $|ab|\le pa^2+b^2/p$ with $p=\frac{4}{\theta}$, $b=\sqrt{\xi}|D^2w^n|_{d\times d}$ and $a=2\bar a_md^2\sqrt{C_0}|\nabla w^n|_d$. Since $\nabla w^n\to\nabla u$ uniformly on $\overline \cO_{m+1}$, then by Proposition \ref{prop:gradbndU} we can assume $|\nabla w^n|_d\!\le\! 1\!+\!N_1$ and obtain 
\begin{align}\label{eq:2bndmixxiD2}
&\big\langle a\nabla \xi,\nabla (|\nabla w^n|^2_d-f^2) \big\rangle\ge -\xi\frac{\theta}{4}|D^2w^n|_{d\times d}^2-C_1,
\end{align}
with $C_1=16d^4\bar a_m^2 C_0\theta^{-1}(1+N_1)^2+\bar a_m d^2 \sqrt{C_0\xi}\|\nabla f^2\|_{C^0(\overline\cO_{m+1})}$. Since $\cL\xi$ and $(\partial_t+\cL)f^2$ are continuous on $\overline \cO_{m+1}$ we have $|\cL\xi|+|(\partial_t+\cL)f^2|\leq C_2$ on $\overline{\cO}_{m+1}$. 
Similarly to the first inequality in \eqref{eq:psicnvxzeta}, $\psi_\varepsilon(\bar\zeta_n)\leq \psi_\varepsilon'(\bar\zeta_n)|\nabla w^n|_d^2\leq \psi_\varepsilon'(\bar\zeta_n)(1+N_1)^2$. Thus 
\[
\psi_\varepsilon(\bar\zeta_n)|\cL\xi|\le \psi_\varepsilon'(\bar\zeta_n)(1+N_1)^2 C_2=:\psi_\varepsilon'(\bar\zeta_n)C_3, 
\]
where $C_3=C_3(m)>0$.

We claim that, for any $\lambda>0$ there are constants $C_4=C_4(m)>0$ and $\kappa_{\delta,m}>0$ such that
\begin{align}\label{eq:bad00}
\xi\langle\nabla w^n,(\partial_t +\cL)(\nabla w^n)\rangle\ge -\tfrac{\theta}{8}\xi|D^2 w^n|^2_{d\times d} -\xi\lambda\theta\psi^2_\eps( \zeta_n)- C_4(1+\lambda^{-1})-\kappa_{\delta,m}R_n,
\end{align}
where ${\zeta_n}\coloneqq (|\nabla u^n|_d^2-f^2-\tfrac{1}{n})(t^*_n,x^*_n)$ and $R_n$ is independent of $(t^*_n,x^*_n)$ and such that $R_n\to 0$ as $n\to\infty$. The claim is proven separately at the end of this proof, for the sake of readability.
Plugging all the above estimates into \eqref{eq:1bndpnlty} and factoring out $\psi'_\eps(\bar \zeta_n)\ge 1$ gives us
\begin{align*}
0\le -\tfrac{\theta}{2}\xi|D^2w^n|^2_{d\times d}+2\lambda\theta\xi\psi^2_\eps(\zeta_n)+C_1+C_2+C_3+2C_4(1+\lambda^{-1})+2\kappa_{\delta,m}R_n.
\end{align*}
Letting $C_5=C_5(m)>0$ be a suitable constant the expression simplifies to
\begin{align*}
\tfrac{\theta}{2}\xi|D^2w^n|^2_{d\times d} \le 2\lambda\theta\xi\psi^2_\eps(\zeta_n)+C_5(1+\lambda^{-1})+2\kappa_{\delta,m}R_n.
\end{align*}
We want to bound $\xi|D^2w^n|_{d\times d}$ by $\xi\psi_\eps(\zeta_n)$. So we multiply both sides of the inequality above by $\xi$, take square root and use $\sqrt{a+b}\le \sqrt{a}+\sqrt{b}$ for $a,b\ge 0$ and $|\xi|\le 1$. That gives
\begin{align}\label{eq:bad01}
\xi|D^2w^n|_{d\times d} \le 2\sqrt{\lambda}\xi\psi_\eps(\zeta_n)+\sqrt{2\theta^{-1}C_5(1+\lambda^{-1})}+\sqrt{4\theta^{-1}\kappa_{\delta,m}R_n}.
\end{align}

Recall that $w^n$ solves 
\begin{align*}
\partial_tw^n +\mathcal{L}w^n-r w^n=\psi_\varepsilon(\zeta_n)- h-\tfrac{1}{\delta}\chi_n(g-u),\quad\text{on $[0,T)\times\R^d$}.
\end{align*}
Multiplying by $\xi$ we can express $\xi\psi_\eps(\zeta_n)$ in terms of the remaining functions in the equation above. Since {$(t^*_n,x^*_n)\in\cO_{S,B}$}, $w^n\to u$ in {$C^{1,2,\beta}(\overline \cO_{S,B})$} and $\chi_n(g-u)\to(g-u)^+$ uniformly on compacts, we can assume with no loss of generality {that on $\overline\cO_{S,B}$ the following hold:} $\frac{1}{\delta}\chi_n(g-u)\leq(1+ K_2)$ by \eqref{eq:obstcpenbnd1}, $|\nabla w^n|_d\le (1+N_1)$ by \eqref{eq:gradbndU}, $\partial_t w^n\le \frac{1}{2}+K_4$ by \eqref{eq:tder} and $r w^n\ge -\frac{1}{2}$ because $u\ge 0$. The coefficients $a$ and $b$ in $\cL$ are bounded on $\overline{B}_{m+1}$ by a constant $A_{m+1}$ (slightly abusing notation). Thus, 
\begin{align*}
\xi\psi_\varepsilon(\zeta_n)= &\,\xi\partial_t w^n -\xi rw^n+ \xi \mathcal{L}w^n+\xi h+\xi\tfrac{1}{\delta}\chi_n(g-u)\notag\\
\leq &\,1+K_4 + A_{m+1}(1+N_1)+ \tfrac{1}{2} A_{m+1}\xi|D^2w^n|_{d\times d}+\|h\|_{C^0(\overline\cO_{m+1})}+(1+K_2).
\end{align*}
Substituting \eqref{eq:bad01} and grouping together the constants we obtain, for some $C_6=C_6(m)>0$, 
\begin{align*}
\Big(1-\sqrt{\lambda} A_{m+1}\Big) \xi\psi_\varepsilon(\zeta_n)&\,\leq C_6\sqrt{1+\lambda^{-1}+ \kappa_{\delta,m}R_n}.
\end{align*}
Then, choosing $\lambda=(4A^2_{m+1})^{-1}$ and recalling that all expressions are evaluated at $(t^*_n,x^*_n)$ we obtain
\begin{align*}
\xi(x^*_n)\psi_\varepsilon\left(|\nabla u^n(t^*_n,x^*_n)|^2_d-f^2(t^*_n,x^*_n)-\tfrac{1}{n}\right)\leq 2C_6\sqrt{1+4A^2_{m+1}+ \kappa_{\delta,m}R_n}.
\end{align*}
Taking limits as $n\to\infty$, using that $(t^*_n,x^*_n)\to (t^*,x^*)$, $R_n\to 0$ and $\nabla u^n\to\nabla u$ (uniformly on compacts), we have
\begin{align*}
\xi(x^*)\psi_\varepsilon\left(|\nabla u(t^*,x^*)|^2_d-f^2(t^*,x^*)\right)\leq 2C_6\sqrt{1+4A^2_{m+1}}=:M_5.
\end{align*}
Recalling the definition of $v$ in \eqref{eq:defnvpsi} we can conclude:
\begin{align*}
\big\|\psi_\varepsilon\left(|\nabla u|^2_d-f^2\right)\big\|_{C^0(\overline{\cO}_{m})}\leq&\,\sup_{(t,x)\in\overline{\cO}_{m+1}}v(t,x)=v(t^*,x^*)\le M_5.
\end{align*}
\end{proof}
\begin{proof}[{\bf Proof of \eqref{eq:bad00}}]
Recall that $u=u^{\eps,\delta}$ and that $w^n$ solves
\begin{align*}
\partial_t w^n+\mathcal{L}w^n-r w^n=-h-\tfrac{1}{\delta}\chi_n\big(g-u^{\eps,\delta}\big)+\psi_{\varepsilon}(|\nabla u^n|^2_d-f^2-\tfrac{1}{n}),\quad  \text{ on }[0,T)\times\R^d.
\end{align*}
Differentiating with respect to $x_k$, multiplying by $\xi$ and evaluating at $(t^*_n,x^*_n)$ we get
\begin{align}\label{eq:pdesmotdif}
\xi(\partial_t w^n_{x_k}+\mathcal{L} w^n_{x_k})=&\,-\xi\mathcal{L}_{x_k} w^n+\xi rw^n_{x_k}-\xi h_{x_k}-\xi\tfrac{1}{\delta}\chi_n'(g-u)(g-u)_{x_k}\\
&\,+\xi\psi_{\varepsilon}'(\zeta_n)(|\nabla u^n|^2_d-f^2-\tfrac{1}{n})_{x_k},
\end{align}
where $\zeta_n=(|\nabla u^n|^2_d-f^2-\tfrac{1}{n})(t^*_n,x^*_n)$. We subtract $v_{x_k}^n=\xi_{x_k}\psi_{\varepsilon}(\bar\zeta_n)+\xi\psi_{\varepsilon}'(\bar\zeta_n)(|\nabla w^n|^2_d-f^2)_{x_k}$ from both sides of \eqref{eq:pdesmotdif},
and we add and subtract $\xi_{x_k}\psi_\varepsilon(\zeta_n)$ on the right-hand side of \eqref{eq:pdesmotdif}. Then
\begin{align}\label{eq:pdesmthdiff}
\xi(\partial_t w^n_{x_k}+\mathcal{L} w^n_{x_k})-v^n_{x_k}=&\,-\xi\mathcal{L}_{x_k} w^n+\xi rw^n_{x_k}-\xi h_{x_k}-\xi\tfrac{1}{\delta}\chi_n'(g-u)(g-u)_{x_k}\\
&\,- \xi_{x_k}\psi_\varepsilon(\zeta_n)+{P_{n,k}},\notag
\end{align}
where
\begin{align*}
{P_{n,k}}=&\, \xi_{x_k}\Big(\psi_\varepsilon(\zeta_n)-\psi_\varepsilon(\bar\zeta_n)\Big) + \xi\Big(\psi_\varepsilon'(\zeta_n)\big(|\nabla u^n|^2_d-f^2\big)_{x_k} -\psi_\varepsilon'(\bar\zeta_n)\big(|\nabla w^n|^2_d-f^2\big)_{x_k} \Big).
\end{align*}
Recall that $(t^*_n,x^*_n)\in\cO_{S,B}$ is a stationary point for $v^n$ in the spatial coordinates, so $v^n_{x_k}=0$ in \eqref{eq:pdesmthdiff} for each $1\leq k\leq d$. Then
\begin{align}\label{eq:2bndpnlty21}
\xi\langle \,\nabla w^n,(\partial_t+\cL)(\nabla w^n)\rangle=&\xi\Big(\!-\!\sum_{k=1}^dw^n_{x_k}\mathcal{L}_{x_k} w^n\!+\!r|\nabla w^n|^2_d\!-\!\langle\nabla w^n,\nabla h\!+\!\tfrac{1}{\delta}\chi'_n(g-u)\nabla(g-u)\rangle\Big)\notag\\
&-\!\psi_\eps(\zeta_n)\langle\nabla w^n,\nabla \xi\rangle\!+\!\sum_{k=1}^d w^n_{x_k}P_{n,k}.
\end{align}
{Since $(t^*_n,x^*_n)\in\cO_{S,B}$, then denoting $\|\cdot\|_{S,B}=\|\cdot\|_{C^0(\overline\cO_{S,B})}$ we have} 
\begin{align*}
\small{\sum}_{k=1}^dw^n_{x_k}P_{n,k}\geq&\,-\|\psi_\varepsilon(\zeta_n)-\psi_\varepsilon(\bar\zeta_n)\|_{S,B}\|\nabla w^n\|_{S,B}\|\nabla \xi\|_{S,B}\\
&\, -\|\psi_\varepsilon'(\zeta_n)-\psi_\varepsilon'(\bar\zeta_n)\|_{S,B}\|\nabla w^n\|_{S,B}\|\nabla \left(|\nabla w^n|^2_d-f^2\right)\|_{S,B}\\
&\,-\|\psi_\eps'(\zeta_n)\|_{S,B}\|\nabla w^n\|_{S,B}\|\nabla(|\nabla u^n|^2_d-|\nabla w^n|^2_d)\|_{S,B}=:\tilde R_n,
\end{align*}
where $\tilde R_n\to 0$ as $n\to\infty$ thanks to $C^{1,2,\beta}(\overline\cO_{S,B})$-convergence of $w^n$ and $u^n$ to $u$, for $\beta\in(0,\alpha)$.
By Cauchy-Schwarz inequality, recalling that $0\le \chi'_n(\cdot)\le 2$ and using arguments as in \eqref{eq:grd00} we have 
\begin{align*}
&\langle\nabla w^n,\nabla h\rangle\leq \|\nabla w^n\|_{C^0(\overline\cO_{m+1})}\|\nabla h\|_{C^0(\overline\cO_{m+1})}\leq (N_1+1)\|\nabla h\|_{C^0(\overline\cO_{m+1})}, \\
&\tfrac{1}{\delta}\chi'_n(g-u)\langle\nabla w^n,\nabla(g-u)\rangle\le \tfrac{2}{\delta}\|\nabla\hat w^n\|_{S,B}\big(\|\nabla g\|_{C^0(\overline\cO_{m+1})}+\|\nabla u\|_{C^0(\overline\cO_{m+1})}\big)=:\kappa_{\delta,m}\tilde R'_n,
\end{align*}
where $N_1=N_1(m+1)$ is as in Proposition \ref{prop:gradbndU} and $\hat w^n=u-w^n$. Recall that {$\|\nabla\hat w^n\|_{S,B}\to 0$} as $n\to\infty$, therefore $\tilde R'_n\to 0$ too. For the penultimate term on the right-hand side of \eqref{eq:2bndpnlty21}, recalling $|\nabla \xi|_d\le\sqrt{C_0\xi}$ (see \eqref{rem:cutoffXI}), we have
\begin{align}\label{eq:psiscprdxid}
\psi_\varepsilon(\zeta_n)\langle \nabla w^n,\nabla \xi\rangle \leq&\,|\nabla w^n|_d|\nabla \xi|_d\psi_\varepsilon(\zeta_n)\leq (N_1+1)\sqrt{C_0\xi}\psi_\varepsilon(\zeta_n)\leq\frac{(N_1+1)^2C_0}{\lambda \theta}+\xi \lambda\theta\psi_\varepsilon^2(\zeta_n),
\end{align}
where in the last inequality we used $ab\leq \frac{a^2}{p}+p b^2$ with $a=(N_1+1)\sqrt{C_0}$, $b=\sqrt{\xi}{\psi_\varepsilon(\zeta_n)}$, and $p=\lambda\theta$, with $\lambda$ a constant to be chosen later and $\theta=\theta_{B_{m+1}}$ as in \eqref{eq:defnthetaEC}. 
For the first term on the right-hand side of \eqref{eq:2bndpnlty21} we argue as in \eqref{eq:wLw} and obtain
\begin{align}\label{eq:theta8}
\sum_{k=1}^dw^n_{x_k} \mathcal{L}_{x_k} w^n\le \tfrac{\theta}{8}|D^2 w^n|^2_{d\times d}+C_1(N_1+1)^2,
\end{align}
where $C_1\coloneqq 8d^4A^2_{m+1}\theta^{-1}+2dA_{m+1}$, the constant $A_{m+1}$ is defined as in \eqref{eq:defnsupab} and, differently from \eqref{eq:bndgrad5a1}, we use $ab\leq pa^2+\frac{b^2}{p}$ with $p=\frac{\theta}{8}$, $a=|D^2 w^n|_{d\times d}$ and $b=\frac{d^2}{2}A_{m+1}|\nabla w^n|_d$.

Combining these bounds we get
\begin{align*}
\xi\langle \nabla w^n,(\partial_t+\cL)(\nabla w^n)\rangle
\geq-\xi \tfrac{\theta}{8}|D^2 w^n|^2_{d\times d} -\xi\lambda\theta\psi^2_\eps(\zeta_n)- C_4(1+\lambda^{-1})-\kappa_{\delta,m}R_n,
\end{align*}
where we define $C_4\coloneqq C_1(N_1+1)^2+(N_1+1)^2C_0\theta^{-1}+(N_1+1)\|\nabla h\|_{C^0(\overline\cO_{m+1})}$ and we collect $\tilde R_n$ and $\kappa_{\delta,m}\tilde R_n'$ in $\kappa_{\delta,m}R_n$ with an abuse of notation.
\end{proof}

The bounds on the penalty terms in the PDE for $u^{\eps,\delta}$ enable the next estimate. 
\begin{theorem}\label{thm:W12pbound}
For any $p\in(1,\infty)$, there is $M_6=M_6(m,p)$ such that 
\begin{align}\label{eq:boundW12p}
\|u^{\varepsilon,\delta}\|_{W^{1,2,p}(\cO_{m})}\leq M_6, \quad\text{for all $\eps,\delta\in(0,1)$}.
\end{align} 
\end{theorem}
\begin{proof}
The proof repeats the exact same arguments as in the proof of Proposition \ref{prp:W12pboundtemp} but applied to $\varphi=\xi_{m+1}u^{\eps,\delta}$ rather than to $\varphi=\xi_{m_0}u^{\eps,\delta}_q$. In addition, we use Lemmas \ref{lem:bndobstpen2} and \ref{thm:gradpenbnd}
to obtain the upper bound for $\|\varphi\|_{W^{1,2,p}(\cO_{m+1})}$ as in \eqref{eq:W12norm2}, which is therefore independent of $\eps,\delta$. 
\end{proof}

\section{The Variational Inequality}\label{sec:final}

In this section, we finally prove our main result, i.e., Theorem \ref{thm:usolvar}. First we prove that Problem \ref{prb:varineq} admits a solution (Theorem \ref{thm:limfuncum}), then we prove that such solution is the value function of our game (Theorem \ref{thm:game}) and it is the maximal solution for Problem \ref{prb:varineq}. 

\begin{theorem}\label{thm:limfuncum} 
There exists a solution $u$ of Problem \ref{prb:varineq}.
\end{theorem}
\begin{proof}
Let $(\eps_k)_{k\in\N}$ be a decreasing sequence with $\eps_k\to 0$. Fix $k,m\in\N$. Thanks to Theorem \ref{thm:W12pbound} and the compact embedding of $W^{1,2,p}(\cO_m)$ into $C^{0,1,\beta}(\overline{\cO}_m)$ for $\beta=1-(d+2)/p$, we can extract a sequence $(u^{\eps_k,\delta^m_{k,j}})_{j\in\N}$ converging to a limit $u^{\eps_k;[m]}$ (possibly depending on $\cO_m$) as $j\to\infty$, in the following sense: 
\begin{equation}\label{eq:conver}
\begin{array}{l}
u^{\eps_k,\delta^m_{k,j}}\to u^{\eps_k;[m]}\quad\text{and}\quad \nabla u^{\eps_k,\delta^m_{k,j}}\to \nabla u^{\eps_k;[m]}\quad\text{in } {C^{\alpha}}(\overline \cO_m),\\[+4pt]
\partial_t u^{\eps_k,\delta^m_{k,j}}\to \partial_t u^{\eps_k;[m]}\quad\text{and}\quad D^2 u^{\eps_k,\delta^m_{k,j}}\to D^2 u^{\eps_k;[m]}\quad\text{weakly in } L^p(\cO_m).
\end{array}
\end{equation}
Up to selecting further subsequences (if needed), we find analogous limits on $\cO_{m+1}\subset \cO_{m+2}\subset \ldots$ so that $u^{\eps_k;[m]}=u^{\eps_k;[m+1]}$ on $\overline\cO_m$, $u^{\eps_k;[m+1]}=u^{\eps_k;[m+2]}$ on $\overline\cO_{m+1}$ and so on. Since $\overline\cO_m\uparrow \R^{d+1}_{0,T}$ as $m\to\infty$, iterating this procedure we can  define a 
limit function $u^{\eps_k}$ on $\R^{d+1}_{0,T}$.

The sequence $(u^{\eps_k})_{k\in\N}$ satisfies the same bound as in \eqref{eq:boundW12p}. Therefore, by the same argument as above we can extract a further converging subsequence, which we denote still by $(u^{\eps_k})_{k\in\N}$ with an abuse of notation. That is, there is a function $u$ on $\R^{d+1}_{0,T}$ such that for any $m\in\N$
\begin{equation*}
\begin{array}{l}
u^{\eps_k}\to u\quad\text{and}\quad \nabla u^{\eps_k}\to \nabla u\quad\text{in } {C^{\alpha}}(\overline \cO_m),\\[+4pt]
\partial_t u^{\eps_k}\to \partial_t u\quad\text{and}\quad D^2 u^{\eps_k}\to D^2 u\quad\text{weakly in } L^p(\cO_m).
\end{array}
\end{equation*}
Finally, we can extract a diagonal subsequence $(u^{\eps_i,\delta_i})_{i\in\N}$ that converges to $u$ locally on $\R^{d+1}_{0,T}$ in the sense above as $ (\eps_i,\delta_i)\to 0$, simultaneously. 

Next we prove that the limit function $u$ is solution of Problem \ref{prb:varineq}. By construction, $u\in W^{1,2,p}_{\ell oc}(\R^{d+1}_{0,T})$. Thanks to \eqref{eq:obstcpenbnd1}, \eqref{eq:gradpeneq} {and $C^{0,1,\alpha}$-convergence on compacts}, in the limit as $\eps,\delta\to 0$ we obtain
\begin{align}\label{eq:twopenbnd}
g(t,x)-u(t,x)\leq 0\quad\text{and}\quad|\nabla u(t,x)|_d-f(t,x)\leq0,\quad \text{for all $(t,x)\in\R^{d+1}_{0,T}$}.
\end{align}
Fix $(\bar t,\bar x)\in\cC$, i.e., $u(\bar t,\bar x)>g(\bar t,\bar x)$. By continuity of $u$ and $g$ there is an open neighbourhood $\cO$ of $(\bar t,\bar x)$ such that $u(t,x)>g(t,x)$ for all $(t,x)\in\cO$. Uniform convergence on compacts of $u^{\eps_i,\delta_i}$ to $u$ also guarantees that $u^{\eps_i,\delta_i}>g$ on $\cO$, for sufficiently large $i$'s. Then, for large $i$'s \eqref{eq:pdeinRdeps} reads
\begin{align*}
\partial_tu^{\varepsilon_i,\delta_i}+\mathcal{L}u^{\varepsilon_i,\delta_i}-ru^{\varepsilon_i,\delta_i}=-h+\psi_\varepsilon\big(|\nabla u^{\varepsilon_i,\delta_i}|_d^2-f^2\big)\ge -h,\quad\text{on $\cO$}.
\end{align*}
Multiplying the equation above by $\phi\in C^\infty_c(\cO)$, $\phi\ge 0$ and letting $i\to\infty$ we obtain the second equation in \eqref{eq:inipde}. Analogously, let $(\bar t,\bar x)\in\cI$, i.e., $|\nabla u(\bar t,\bar x)|_d<f(\bar t, \bar x)$. Then, by continuity of $\nabla u$ and uniform convergence on compacts of $\nabla u^{\eps_i,\delta_i}\to \nabla u$ we find an open neighbourhood $\cO$ such that $|\nabla u|_d<f$ on $\cO$ and $|\nabla u^{\eps_i,\delta_i}|_d<f$ on $\cO$ for sufficiently large $i$. In such neighbourhood \eqref{eq:pdeinRdeps} reads
\begin{align*}
\partial_t u^{\varepsilon_i,\delta_i}+\mathcal{L}u^{\varepsilon_i,\delta_i}-ru^{\varepsilon_i,\delta_i}=-h-\tfrac{1}{\delta}\big(g-u^{\eps_i,\delta_i}\big)^+\le -h,
\end{align*}
and by the same argument as above, using test functions, we can pass to the limit and obtain the third equation in \eqref{eq:inipde}. The case in which $(\bar t,\bar x)\in\cI\cap\cC$ is now obvious and the first equation in \eqref{eq:inipde} also holds {for all $(t,x)$ by standard PDE theory (e.g., as in the proof of Theorem \ref{thm:highreguued})}. Finally, the terminal condition is trivially satisfied since $u^{\eps,\delta}(T,x)=g(T,x)$ for all $(\eps,\delta)\in(0,1)^2$.

Notice that $u$ has at most quadratic growth by Lemma \ref{lem:polygrow}.
\end{proof}
We are now in the position to prove our main result.
\begin{theorem}\label{thm:game}
The game in \eqref{eq:lowuppvfnc} admits a value $v$ which is also the maximal solution of Problem \ref{prb:varineq}. Moreover, $\tau_*$ defined in \eqref{eq:taustar} is optimal for the stopper.
\end{theorem}
\begin{proof}
Fix $(t,x)\in\R^{d+1}_{0,T}$, let $[(n,\nu),\tau]\in\cA_t\times \cT_t$ and redefine $\rho_m$ as in \eqref{eq:rom} but with c\`adl\`ag controls rather than continuous ones. By regularity of $u^{\eps,\delta}$ (Theorem \ref{thm:highreguued}), letting $\rho_m^k\!=\!\rho_m\wedge\!(T\!-t\!-k^{-1})^+$ It\^o's formula applies to 
$e^{-r(\tau\wedge\rho^k_{m})}u^{\varepsilon,\delta}(t+{\tau\wedge\rho^k_{m}},X_{{\tau\wedge\rho^k_{m}}}^{[n,\nu]})$. Using that $u^{\eps,\delta}$ solves Problem \ref{prb:penprobRd}, taking expectations and letting $k\uparrow \infty$ we obtain 
\begin{align}\label{eq:itopensol}
u^{\varepsilon,\delta}(t,x)=\E_x\biggr[&e^{-r(\tau\wedge\rho_{m})}u^{\varepsilon,\delta}(t+\tau\wedge\rho_{m},X_{\tau\wedge\rho_{m}}^{[n,\nu]})\\
&+\! \int_{0}^{\tau\wedge\rho_{m}}\!\!e^{-rs}\big[h\!+\!\tfrac{1}{\delta}\big(g\!-\!u^{\varepsilon,\delta}\big)^+\!-\!\psi_{\varepsilon}\big(|\nabla u^{\varepsilon,\delta}|^2_d\!-\!f^2\big)\big]\big(t\!+\!s,X_{s-}^{[n,\nu]}\big)\,\ud s \\
& -\!\int_{0}^{\tau\wedge\rho_{m}}e^{-rs}\langle\nabla u^{\varepsilon,\delta}(t+s,X_{s-}^{[n,\nu]}), n_s\rangle \,\ud \nu_s^c \notag\\
& -\! \sum_{0\leq s\leq\tau\wedge\rho_{m}}e^{-rs}\int_{0}^{\Delta\nu_s}\langle\nabla u^{\varepsilon,\delta}(t\!+\!s,X^{[n,\nu]}_{s-}+\lambda n_s),n_s\rangle\, \ud\lambda \biggr].\notag
\end{align}

We want to take limits as $\eps,\delta\to 0$ and pass the limits under expectations. To do that we notice that $X^{[n,\nu]}_{s-}\in\overline B_m$ for all $s\in[0,\rho_{m}]$, $\P_x$-a.s. Then the terms under the integral with respect to `$\ud s$' are bounded thanks to Assumption \ref{ass:gen2}, Lemma \ref{lem:bndobstpen2} and Lemma \ref{thm:gradpenbnd}. 
Since $\nabla u^{\eps,\delta}$ is also bounded by $N_1(m)$ (Proposition \ref{prop:gradbndU} with $m_0$ therein replaced by $m$), the integrals with respect to the control are bounded by $N_1(m) \nu_{T-t}$, which is square integrable by definition of $\cA_t$. This reasoning does not apply to the final jump in the sum, because that jump could bring the process $X^{[n,\nu]}$ outside the ball $B_m$ if $\rho_m\le\tau$. It is therefore convenient to isolate that term from the sum and rewrite
\begin{align}\label{eq:disintegration}
\begin{aligned}
&u^{\varepsilon,\delta}(t+\tau\wedge\rho_{m},X_{\tau\wedge\rho_{m}}^{[n,\nu]})-\int_{0}^{\Delta\nu_{\tau\wedge\rho_{m}}}\langle\nabla u^{\varepsilon,\delta}(t\!+\!\tau\wedge\rho_{m},X^{[n,\nu]}_{\tau\wedge\rho_{m}-}+\lambda n_{\tau\wedge\rho_{m}}),n_{\tau\wedge\rho_{m}}\rangle\, \ud\lambda \\
&\,=u^{\varepsilon,\delta}(t+\tau\wedge\rho_{m},X_{\tau\wedge\rho_{m}-}^{[n,\nu]}).
\end{aligned}
\end{align}
Finally, we recall that $u^{\eps,\delta}$ has quadratic growth by Lemma \ref{lem:polygrow} and notice that
\begin{align}\label{eq:extSDE}
\E_x\Big[\sup_{0\le s\le T-t}\big|X^{[n,\nu]}_{s}\big|^2_d\Big]\le c\big(1+|x|^2_d+\E_x[|\nu_{T-t}|^2]\big),
\end{align}
by standard estimates for SDEs \cite[Thm.\ 2.5.10]{krylov1980controlled}, {with $c>0$ independent of $m$, $\eps$, $\delta$}. Then we are allowed to use dominated convergence and it remains to evaluate the limit. 

For $\P_x$-a.e.\ $\omega\in\Omega$ there is a compact $\cK_\omega\subset\R^d$ such that $X^{[n,\nu]}_{s}(\omega)\in\cK_\omega$ for all $s\in[0,\rho_{m}(\omega)]$, by right-continuity of the process and the fact that $\nu$ is square integrable. Then, uniform convergence of $(u^{\eps,\delta},\nabla u^{\eps,\delta})$ to $(u,\nabla u)$ on compacts implies
\begin{equation}
\begin{array}{l}
\lim_{\eps,\delta\to 0}\langle \nabla u^{\eps,\delta}(t+s,X^{[n,\nu]}_{s-}),n_s\rangle(\omega)=\langle \nabla u(t+s,X^{[n,\nu]}_{s-}),n_s\rangle(\omega),\\ [+5pt]
\lim_{\eps,\delta\to 0}\langle \nabla u^{\eps,\delta}(t+s,X^{[n,\nu]}_{s-}+\lambda n_s),n_s\rangle(\omega)=\langle \nabla u(t+s,X^{[n,\nu]}_{s-}+\lambda n_s),n_s\rangle(\omega),
\end{array}
\end{equation}
for all $s\in[0,\rho_{m}(\omega)]$ and all $\lambda\in[0,\Delta\nu_s(\omega)]$, for $\P_x$-a.e.\ $\omega\in\Omega$.

Let us now choose $\tau=\tau_*=\inf\{s\geq0\,|\, u(t\!+\!s,X_s^{[n,\nu]} )\leq g(t\!+\!s,X_s^{[n,\nu]})\}$ (notice that $\tau_*=\tau_*(n,\nu)$). Fix $\omega\in\Omega$ outside of a $\P_x$-null set and recall that $s\mapsto \nu_s(\omega)$ has at most countably many jumps on any bounded interval. Then, for a.e.\ $0\le s<\tau_*(\omega)\wedge\rho_m(\omega)$ there is $q_{s,\omega}>0$ such that 
\[
\big(u-g\big)\big(t\!+\!s,X^{[n,\nu]}_{s-}(\omega)\big)=\big(u-g\big)\big(t\!+\!s,X^{[n,\nu]}_{s}(\omega)\big)\ge q_{s,\omega}.
\]
Therefore, the pointwise convergence of $u^{\eps,\delta}$ to $u$ implies that for a.e.\ $0\le s<\tau_*(\omega)\wedge\rho_m(\omega)$  
\[
\limsup_{\eps,\delta\to0}\,\tfrac1\delta\big(g-u^{\eps,\delta}\big)^+\big(t\!+\!s,X^{[n,\nu]}_{s-}(\omega)\big)=0.
\]

Using the observations above, combined with \eqref{eq:disintegration}, dominated convergence and $\psi_\eps\ge 0$ we obtain
\begin{align*}
u(t,x)\leq&\, \E_x\biggr[e^{-r(\tau_{*}\wedge\rho_{m})}u(t\!+\!\tau_{*}\wedge\rho_{m},X_{\tau_{*}\wedge\rho_{m}}^{[n,\nu]})\!+\!\int_{0}^{\tau_{*}\wedge\rho_{m}}\!\!e^{-rs}h(t+s,X_s^{[n,\nu]})\,\ud s \\
&-\!\int_{0}^{\tau_{*}\wedge\rho_{m}}\!\!e^{-rs}\langle\nabla u(t\!+\!s,X_{s-}^{[n,\nu]}), n_s\rangle \,\ud \nu_s^c\!-\! \sum_{0\leq s\leq\tau_{*}\wedge\rho_{m}}\!\!\!e^{-rs}\!\!\int_{0}^{\Delta\nu_{s}}\!\!\langle\nabla u(t\!+\!s,X_{s-}^{[n,\nu]}\!+\!\lambda n_s),n_s\rangle\, \ud\lambda \biggr].
\end{align*}

By $|\nabla u|_d\leq f$ and the definition of $\tau_{*}\wedge\rho_{m}$ we have
\begin{align*}
u(t,x)\leq &\,\E_x\biggr[e^{-r\tau_{*}}g(t\!+\!\tau_{*},X_{\tau_{*}}^{[n,\nu]})\mathds{1}_{\{\tau_{*}\leq \rho_{m}\}}+ e^{-r\rho_{m}}u(t\!+\!\rho_{m},X_{\rho_{m}}^{[n,\nu]})\mathds{1}_{\{\tau_{*}>\rho_{m}\}}\biggr]\\
&+\E_x\biggr[\int_{0}^{\tau_{*}\wedge\rho_{m}}e^{-rs}h(t+s,X_s^{[n,\nu]})\,\ud s +\int_{[0,\tau_{*}\wedge\rho_{m}]}e^{-rs}f(t+s,X_s^{[n,\nu]})\circ\,\ud \nu_s \biggr].
\end{align*}
Now we let $m\to\infty$. Clearly $\rho_{m}\uparrow T-t$, $\P_x$-a.s.\ by \eqref{eq:extSDE}. Since the bound in \eqref{eq:extSDE} is independent of $m$ and functions $u$ and $g$ have at most quadratic growth we can apply the dominated convergence theorem to pass the limit inside the first expectation. We can also take the limit inside the second expectation by monotone convergence as all terms under the integral are non-negative. For $\P_x$-a.e.\ $\omega\in\Omega$ we have $X^{[n,\nu]}_{s}(\omega)\in\cK_\omega$ for all $s\in[0,T-t]$. Then 
\begin{align*}
&\lim_{m\to\infty}e^{-r\rho_{m}(\omega)}u\big(t+\rho_{m}(\omega),X_{\rho_{m}}^{[n,\nu]}(\omega)\big)\mathds{1}_{\{\tau_{*}>\rho_{m}\}}(\omega)\\
&= e^{-r(T-t)}u\big(T,X_{T-t}^{[n,\nu]}(\omega)\big)\mathds{1}_{\{\tau_{*}\ge T-t\}}(\omega)=e^{-r(T-t)}g\big(T,X_{T-t}^{[n,\nu]}(\omega)\big)\mathds{1}_{\{\tau_{*}= T-t\}}(\omega),\quad\text{$\P_x$-a.e.\ $\omega\in\Omega$},
\end{align*}
because $\tau_{*}(\omega)\le T-t$ and $u$ is uniformly continuous on $[0,T]\times\cK_\omega$. Hence, for $m\to\infty$ we obtain 
\begin{align}\label{eq:tauetaopt}
u(t,x)\le\cJ_{t,x}(n,\nu,\tau_{*}).
\end{align} 
By arbitrariness of $(n,\nu)\in\cA_t$ and sub-optimality of $\tau_{*}$ we have $u(t,x)\leq \underline{v}(t,x)$ by definition of lower value. 

Next we prove $u\geq \overline{v}$. Since $\frac{1}{\delta}(g-u^{\varepsilon,\delta})^+\ge 0$ that term can de dropped from \eqref{eq:itopensol} to obtain a lower bound for $u^{\eps,\delta}$. 
We let $\delta\to 0$ in \eqref{eq:itopensol} along the sequence constructed in \eqref{eq:conver} while keeping $\eps$ fixed. As above, dominated convergence applies, and thanks to $u^\eps\ge g$ ({Lemma \ref{lem:bndobstpen2}}) we obtain
\begin{align}\label{eq:itopensol2}
u^{\varepsilon}(t,x)\ge\!\E_x\biggr[&e^{-r(\tau\wedge\rho_{m})}g(t\!+\!\tau\wedge\rho_{m},X_{\tau\wedge\rho_{m}}^{[n,\nu]})\!+\! \int_{0}^{\tau\wedge\rho_{m}}\!\!e^{-rs}\big[h\!-\!\psi_{\varepsilon}\big(|\nabla u^{\varepsilon}|^2_d\!-\!f^2\big)\big]\big(t\!+\!s,X_{s-}^{[n,\nu]}\big)\,\ud s \notag\\
& -\!\int_{0}^{\tau\wedge\rho_{m}}\!\!\!e^{-rs}\langle\nabla u^{\varepsilon}(t\!+\!s,X_{s-}^{[n,\nu]}), n_s\rangle \ud \nu_s^c\\
& -\! \sum_{0\leq s\leq\tau\wedge\rho_{m}}\!\!\!e^{-rs}\!\!\int_{0}^{\Delta\nu_s}\langle\nabla u^{\varepsilon}(t\!+\!s,X^{[n,\nu]}_{s-}+n_s\lambda),n_s\rangle\, \ud\lambda \biggr].\notag
\end{align}
We can now choose a control pair $(n,\nu)=(n^\eps,\nu^\eps)$ defined as in \eqref{eq:optcntr2} but with $u^{\eps,\delta}$ therein replaced by $u^{\eps}$. 
Although $\nabla u^\eps(t,\cdot)$ is not Lipschitz, it can be shown by standard localisation procedure and the use of \cite[Thm.\ 1]{veretennikov1981strong} that the associated controlled SDE admits a unique, non-exploding, strong solution $X^\eps=X^{[n^\eps,\nu^\eps]}$ on $[0,T\!-\!t]$ (the proof is given in Appendix \ref{app:Xeps} for completeness).

By construction, the pair $(n^{\eps},\nu^{\eps})$ satisfies
\begin{align}\label{eq:nnueps}
-\big\langle n^{\eps}_s\dot{\nu}^{\eps}_s,\nabla u^{\eps}(t+s,X_s^{\eps})\big\rangle-\psi_{\varepsilon}\left(|\nabla u^{\eps}|_d^2- f^2\right)(t+s,X_s^{\eps})= H^{\eps}(t+s,X_s^{\eps},n^{\eps}_s\dot{\nu}^{\eps}_s),
\end{align}
with $H^\eps$ as in \eqref{eq:Heps}. Then, from \eqref{eq:itopensol2} we obtain 
\begin{align*}
u^{\varepsilon}(t,x)\geq\E_x\biggr[&\,e^{-r(\tau\wedge\rho_{m})}g(t+\tau\wedge\rho_{m},X^{\eps}_{\tau\wedge\rho_{m}})+ \int_{0}^{\tau\wedge\rho_{m}}e^{-rs}\big(h\!+\!H^{\varepsilon}(\cdot,n_s^{\eps}\dot{\nu}_s^{\eps})\big)(t+s,X_s^{\eps})\ud s\biggr].
\end{align*}
Taking $p=f(t+s,X_s^{\eps})n_s^{\eps}$ in the Hamiltonian, letting $m\to\infty$ and using Fatou's lemma, we obtain 
\begin{align*}
u^{\varepsilon}(t,x)\geq\E_x\biggr[&\,e^{-r\tau}g\big(t+\tau,X_{\tau}^{\eps}\big)+ \int_{0}^{\tau}e^{-rs}h(t+s,X_s^{\eps})\ud s +\int_{[0,{\tau}]}e^{-rs}f(t+s,X_s^{\eps})\circ \ud \nu_s^{\eps}\biggr],
\end{align*}
where we notice that $\nu^{\eps}$ is absolutely continuous so that the final integral is obvious. 

By arbitrariness of $\tau$ we can take supremum over all stopping times. Since $(n^\eps,\nu^\eps)\in\cA_t$, we can also take infimum over all admissible controls and continue with the same direction of inequalities. That is, $u^\eps(t,x)\ge \overline v(t,x)$. Finally, letting $\eps\to 0$ we obtain $u(t,x)\ge\overline v(t,x)$, as needed. Since $u\le \underline v$ was proven above, we conclude $u=v=\overline v=\underline v$. 

Now that we know that $u=v$ is the value of the game, \eqref{eq:tauetaopt} also yields optimality of $\tau_*$. Notice that, in particular, $\tau_*=\tau_*(n,\nu)$ is best response against any admissible pair $(n,\nu)$.

It remains to prove that $v$ is indeed the maximal solution of Problem \ref{prb:varineq}. Let $w$ be another solution of Problem \ref{prb:varineq}.  
By standard mollification procedure we can construct a sequence $(w_k)_{k\in\N}\subset C^\infty_c(\R^{d+1}_{0,T})$ such that $w_k\to w$ and $\nabla w_k\to\nabla w$ uniformly on compact sets, with $\partial_t w_k \to \partial_t w$ and $D^2w_k\to D^2w$ strongly in $L^p_{\ell oc}(\R^{d+1}_{0,T})$ for all $p\in[1,\infty)$, as $k\to\infty$. For $\eta> 0$, set 
\[
\cC_w^\eta=\{(t,x)\in\R^{d+1}_{0,T}:w(t,x)> g(t,x)+\eta\}.
\]
For fixed $\eta>0$ and $m\in\N$, by the same argument as in the proof of \cite[Thm.\ 4.1, Ch.\ VIII]{fleming2006controlled} 
\begin{align}\label{eq:limsup-w}
\liminf_{k\to\infty}\inf_{(t,x)\in\overline{\cO_m\cap\cC^\eta_w}}\,(\partial_tw_k+\cL w_k-r w_k+h)(t,x)\ge 0.
\end{align}
Pick an arbitrary pair $(n,\nu)\in\cA_t$ and denote 
\[
\zeta_\eta=\inf\{s\ge 0|w(t+s,X^{[n,\nu]}_s)\le g(t+s,X^{[n,\nu]}_s)+\eta\}\wedge(T-t).
\]
By an application of It\^o's formula to $e^{-rs}w_k(t\!+\!s,X^{[n,\nu]}_s)$ we obtain
\begin{align*}
w_k(t,x)=&\, \E_x\biggr[e^{-r(\zeta_{\eta}\wedge\rho_{m})}w_k(t\!+\!\zeta_{\eta}\wedge\rho_{m},X_{\zeta_{\eta}\wedge\rho_{m}}^{[n,\nu]})\!-\!\int_{0}^{\zeta_{\eta}\wedge\rho_{m}}\!\!e^{-rs}(\partial_tw_k\!+\!\cL w_k\!-\!rw_k)(t\!+\!s,X_s^{[n,\nu]})\ud s \\
&-\!\int_{0}^{\zeta_{\eta}\wedge\rho_{m}}\!\!e^{-rs}\langle\nabla w_k(t\!+\!s,X_{s-}^{[n,\nu]}), n_s\rangle \,\ud \nu_s^c\!-\! \sum_{s\leq\zeta_{\eta}\wedge\rho_{m}}\!\!\!e^{-rs}\!\!\int_{0}^{\Delta\nu_{s}}\!\!\langle\nabla w_k(t\!+\!s,X_{s-}^{[n,\nu]}\!+\!\lambda n_s),n_s\rangle\, \ud\lambda \biggr].
\end{align*}
Letting $k\to\infty$, dominated convergence (up to possibly selecting a subsequence) and reverse Fatou's lemma (justified by \eqref{eq:limsup-w}) allow us to pass the limit under expectation. Then, exploiting the uniform convergence of $(w_k,\nabla w_k)$ to $(w,\nabla w)$ on $\overline{\cO_m\cap\cC^\eta_w}$, \eqref{eq:limsup-w}, the definition of $\zeta_\eta\wedge\rho_m$ and the fact that $|\nabla w|_d\le f$ we have
\begin{align*}
w(t,x)\le\,\eta\!+\! \E_x\biggr[&e^{-r\zeta_\eta}g(t\!+\!\zeta_{\eta},X_{\zeta_{\eta}}^{[n,\nu]})\mathds{1}_{\{\zeta_\eta\le\rho_m\}}+e^{-r\rho_m}w(t\!+\!\rho_{m},X_{\rho_{m}}^{[n,\nu]})\mathds{1}_{\{\zeta_\eta>\rho_m\}}\\
&+\!\int_{0}^{\zeta_{\eta}\wedge\rho_{m}}\!\!e^{-rs}h(t\!+\!s,X_s^{[n,\nu]})\ud s\!+\!\int_0^{\zeta_\eta\wedge\rho_m}e^{-rs} f(t\!+\!s,X^{[n,\nu]}_s)\circ\ud \nu_s \biggr].
\end{align*}
Finally, letting $m\to\infty$, the same arguments that lead to \eqref{eq:tauetaopt} give $w(t,x)\le \eta+\cJ_{t,x}(n,\nu,\zeta_\eta)$. Hence, $w(t,x)\le \eta+v(t,x)$ and letting $\eta\to 0$ we conclude.
\end{proof}

\begin{remark}
It is worth noticing that the proof above can be repeated verbatim if we replace $\cA_t$ with $\cA^\circ_t$ everywhere. Thus we conclude that 
\[
v(t,x)=\inf_{(n,\nu)\in\cA^\circ_t}\sup_{\tau\in\cT_t}\cJ_{t,x}(n,\nu,\tau)=\sup_{\tau\in\cT_t}\inf_{(n,\nu)\in\cA^\circ_t}\cJ_{t,x}(n,\nu,\tau).
\]
That is, the game with absolutely continuous controls admits the same value as the game with singular controls. It is however expected, but it will not be proven here, that an optimal control cannot be found in $\cA^\circ_t$ whereas it should be possible to find one in $\cA_t$ in some cases. 
\end{remark}

We make a simple final observation concerning the nature of the optimal control in our game.
\begin{lemma}\label{lem:jumptau}
Let $(t,x)\in\R^{d+1}_{0,T}$. For any $\tau\in\cT_t$ we have
\[
\inf_{(n,\nu)\in\cA_t}\cJ_{t,x}(n,\nu,\tau)=\inf_{(n,\nu)\in\cA^\tau_t}\cJ_{t,x}(n,\nu,\tau),
\]  
where $\cA^\tau_t:=\{(n,\nu)\in\cA_t\, |\, \nu_\tau=\nu_{\tau-},\,\P_x-a.s.\}$.
\end{lemma}
\begin{proof}
In the expression of $\cJ_{t,x}(n,\nu,\tau)$, for any triple $[(n,\nu),\tau]\in\cA_t\times\cT_t$ we have
\begin{align}\label{eq:lemjump}
\begin{aligned}
&e^{-r\tau }g(t+\tau,X^{[n,\nu]}_\tau)+\int_{[0,\tau]}e^{-rs}f(t+s,X^{[n,\nu]}_s)\circ\ud\nu_s\\
&=e^{-r\tau }g(t\!+\!\tau,X^{[n,\nu]}_{\tau-})\!+\!\int_{[0,\tau)}\!\!e^{-rs}f(t\!+\!s,X^{[n,\nu]}_s)\circ\ud\nu_s\\
&\quad+\!e^{-r\tau}\int_{0}^{\Delta\nu_\tau}\!\!\big(\langle\nabla g,n_\tau\rangle\!+\!f\big)(t\!+\!\tau,X^{[n,\nu]}_{\tau-}\!+\!\lambda n_\tau)\ud \lambda\\
&\ge e^{-r\tau }g(t+\tau,X^{[n,\nu]}_{\tau-})+\int_{[0,\tau)}e^{-rs}f(t+s,X^{[n,\nu]}_s)\circ\ud\nu_s,
\end{aligned}
\end{align}
where the final inequality is due to \eqref{eq:grdg<f}. Therefore, the controller attains a lower payoff by avoiding a jump of the control at time $\tau$. That concludes the proof.
\end{proof}

\section{Short remark about methodology}
The methodology in our paper combines analytical and probabilistic estimates. The former are predominant in particular in Proposition \ref{prop:locgradp}, Proposition \ref{prop:gradbndU}, Lemma \ref{thm:gradpenbnd}, which yield bounds on the spatial gradient of the solution of the penalised problem (first on bounded domains and then on unbounded domains), uniformly with respect to the penalisation parameters and the cut-off functions. Those analytical estimates are rather convoluted and do not necessarily offer a clear intuition for generalisations to different setups.

Probabilistic techniques instead are used in Lemma \ref{lem:polygrow} and Lemma \ref{lem:fntbndr} to obtain, respectively, growth estimates on the solution of the penalised problem and bounds on its gradient on $\partial_P\cO_m$. In Lemma \ref{lem:bndobstpen2} and Lemma \ref{lem:bndtimder} quick probabilistic arguments applied to the solution $u^{\eps,\delta}$ of the penalised problem on unbounded domain yield a uniform bound on $\delta^{-1}(g-u^{\eps,\delta})^+$ and an upper bound for the time derivative $\partial_t u^{\eps,\delta}$. Finally, Theorem \ref{thm:game} is proven using probabilistic representations of the functions $u^{\eps,\delta}$ and $u^\eps=\lim_{\delta\to 0}u^{\eps,\delta}$ combined with limiting arguments. That approach is needed because an application of It\^o's formula (or generalisations thereof) to $u(t,X^{[n,\nu]}_t)$, for generic $(n,\nu)$, is not possible due to the lack of sufficient regularity of the maximal solution $u$ of Problem \ref{prb:penprobRd}. 

It seems to us that proofs obtained via probabilistic methods are generally easy to interpret. It would be interesting to obtain fully probabislitic arguments also for the bounds on the spatial gradient, which we currently obtain via analytical methods. This, however, appears a very challenging task and we leave it for future work.


\appendix 
\section{}

\subsection{Cut-off functions}\label{app:xi}
Here we give a construction of the functions in \eqref{rem:cutoffXI}. Let $\xi:\R\to[0,1]$ be defined as
\begin{align*}
\xi(z)\coloneqq 
\begin{cases}
1, & z \leq 0,\\ 
0, &z\geq 1,\\ 
\exp\Big(\frac{1}{z-1}\Big)\Big/\Big[\exp\Big(\frac{1}{z-1}\Big)+\exp\Big(-\frac{1}{z}\Big)\Big], & 0<z<1, 
\end{cases}
\end{align*}
and set $\xi_m(x)=\xi(|x|_d-m)$ for $x\in \R^d$. Then $\xi_m\in C^\infty_c(\R^d)$, $0\le \xi_m\le 1$, $\xi_m=1$ on $\overline B_m$ and $\xi_m=0$ on $\R^d\setminus B_{m+1}$. It is clear that $\nabla\xi_m={\bf 0}$ on $B_m$ and $\R^d\setminus B_{m+1}$. It can also be checked that for $x\in B_{m+1}\setminus B_m$
\begin{align}\label{eq:cutoff0}
\partial_{x_k}\xi_m(x)=\frac{x_k}{|x|_d}\xi'(|x|_d-m),\quad\text{for $k=1,\ldots d$,}
\end{align}
and therefore $|\nabla \xi_m(x)|_d^2=\big[\xi'(|x|_d-m)\big]^2$.
Since $\xi'(z)=-\xi(z)(1-\xi(z))\big(\frac{1}{(z-1)^2}+\frac{1}{z^2}\big)$,
then $|\nabla \xi_m(x)|_d^2\le C_0\xi_m(x)$ for all $x\in\R^d$, for a suitable $C_0>0$ independent of $m$. 

\subsection{Proof of Lemma \ref{lem:stability}}
For existence and uniqueness of the solution to \eqref{eq:PDEsmth2} we invoke \cite[Thm.\ 3.3.7]{friedman2008partial}. Indeed the smoothing of $u^{\eps,\delta}_m$ and $(\,\cdot\,)^+$ guarantees that 
\begin{align}\label{eq:rhs}
h_m+\tfrac{1}{\delta}\chi_n(g_m-u^{\eps,\delta}_m)-\psi_\eps\big(|\nabla u^n|^2_d-f^2_m-\tfrac{1}{n}\big)\in C^{0,1,\alpha}(\overline \cO_m).
\end{align}
Moreover, the compatibility condition
\[
\lim_{s\uparrow T}(\partial_t g_m\!+\!\cL g_m\!-\!r g_m)(s,x)\!=\!\big[\!-\!h_m\!-\!\tfrac{1}{\delta}\chi_n(g_m\!-\!u^{\eps,\delta}_m)\!+\!\psi_\eps\big(|\nabla u^n|^2_d\!-\!f^2_m\!-\!\tfrac{1}{n}\big)\big](T,x), \,\, \text{for $x\in \partial B_m$},
\]
holds with both sides of the equation equal to zero. Indeed, given that $\xi_{m-1}\in C^\infty_c(B_m)$ we have $g_m=\partial_{x_i} g_m=\partial_{x_i x_j} g_m=0$ on $[0,T]\times\partial B_m$. Moreover, $g_m=0$ on $[0,T]\times\partial B_m$ also implies $\partial_t g_m=0$ on $[0,T]\times\partial B_m$. So the left-hand side of the equation is equal to zero. On the right-hand side, for $x\in\partial B_m$ we have $h_m(T,x)=g_m(T,x)=u^{\eps,\delta}_m(T,x)=0$ and
\[
|\nabla u^n(T,x)|^2_d\le |\nabla u^{\eps,\delta}_m (T,x)|^2_d+\tfrac{1}{n}=|\nabla g_m (T,x)|^2_d+\tfrac{1}{n}\le f^2_m(T,x)+ \tfrac{1}{n},
\]
by uniform convergence of $\nabla u^n$ to $\nabla u^{\eps,\delta}_m$ and \eqref{eq:grdgm<fm}. The compatibility condition follows upon recalling $\chi_n(0)=0$ and $\psi_\eps(z)=0$ for $z\le 0$.

The fact that $w^n\in C^{1,3,\alpha}_{L oc}(\cO_m)$ is also consequence of \eqref{eq:rhs} and standard interior estimates for PDEs \cite[Thm.\ 3.5.11+Cor.\ 3.5.1]{friedman2008partial}. Instead, the convergence result $w^n\to u^{\eps,\delta}_m$ in $C^{1,2,\beta}(\overline\cO_m)$, as $n\to\infty$, for $\beta\in(0,\alpha)$, is a special case of Lemma \ref{lem:approxpenprob}.
\hfill$\square$

\begin{lemma}\label{lem:approxpenprob}
Let $F:\R\times\R^d\to\R$ be a Lipschitz continuous function. 
Fix $\phi,\varphi\in C^{0,1,\alpha}(\overline{\cO}_m)$ and let $u$ be a solution in $C^{1,2,\alpha}(\overline{\cO}_m)$ of 
\begin{align}
\begin{cases}\label{eq:pdelimitapp}
\partial_tu+\cL u-ru=-h_m+F(\phi,\nabla \varphi), & \text{ on }\cO_m, \\
u(t,x)=g_m(t,x), &(t,x)\in\partial_P \cO_m.
\end{cases}
\end{align}
Let $(\phi_n)_{n\in\N},(\varphi_n)_{n\in\N}\subseteq C^{0,1,\alpha}(\overline{\cO}_m)$ be such that $\phi_n\to\phi$ and $\varphi_n\to\varphi$ in $C^{0,1,\gamma}(\overline{\cO}_m)$ as $n\to\infty$ for all $\gamma\in(0,\alpha)$. Let $(F_n)_{n\in\N}$ be equi-Lipschitz continuous functions $F_n:\R\times\R^d\to\R$ such that $F_n\to F$ in $C^0_{\ell oc}(\R^{1+d})$. 
Finally, denote by $u_n$ a solution to \eqref{eq:pdelimitapp} in $C^{1,2,\alpha}(\overline\cO_m)$ with $F_n(\phi_n,\nabla \varphi_n)$ instead of $F(\phi,\nabla \varphi)$. 

Then, up to possibly selecting a subsequence,
\begin{align}\label{eq:ctechlm}
\lim_{n\to\infty}\|u_n-u\|_{C^{1,2,\gamma}(\overline{\cO}_m)}=0,\,\quad \text{for all $\gamma\in(0,\alpha)$}.
\end{align}
If $(\varphi_n)_{n\in\N}\subseteq C^{0,2,\alpha}(\overline{\cO}_m)$ and $(F_n)_{n\in\N}\subseteq C^{1,\alpha}(\R^{1+d})$, then $(u_n)_{n\in\N}\subseteq C^{1,3,\alpha}_{L oc} (\cO_m)\cap C^{1,2,\alpha}(\overline\cO_m)$.
\end{lemma}

\begin{proof}
Define $\hat{u}_n\coloneqq u-u_n$. Then $\hat{u}_n$ solves 
\begin{align*}
\begin{cases}
\partial_t\hat{u}_n+\mathcal{L}\hat{u}_n-r\hat{u}_n=F(\phi,\nabla\varphi)-F_n(\phi_n,\nabla\varphi_n),&\text{ on }\cO_m, \\ 
\hat{u}_n(t,x)=0,&(t,x)\in\partial_P \cO_m.
\end{cases}
\end{align*}
By \cite[Thm.\ 3.2.6]{friedman2008partial} we have the estimate 
\begin{align}\label{eq:stab}
\|\hat{u}_n\|_{C^{1,2,\gamma}(\overline{\cO}_m)}\leq K \|F(\phi,\nabla\varphi)-F_n(\phi_n,\nabla\varphi_n)\|_{C^{0,0,\gamma}(\overline{\cO}_m)},
\end{align}
for a constant $K>0$ independent of $n$. Notice that by equi-Lipschitz continuity, the sequence $(F_n)_{n\in\N}$ is compact in any $C^{\beta}(\overline U)$ for $\beta\in(0,1)$ and bounded set $U\subset\R^{d+1}$. Thanks to the convergence of the functions $\phi_n$, $\varphi_n$ and $F_n$ we have that, up to possibly selecting a subsequence, $F_n(\phi_n,\nabla\varphi_n)\to F(\phi,\nabla\varphi)$ in ${C^{\gamma}}(\overline{\cO}_m)$ as $n\to\infty$ for all $\gamma\in(0,\alpha)$. Thus, \eqref{eq:ctechlm} holds.

If we also assume that $(\varphi_n)_{n\in\N}\subseteq C^{0,2,\alpha}(\overline{\cO}_m)$ and $(F_n)_{n\in\N}\subseteq C^{1,\alpha}(\R^{1+d})$, it turns out that $F_n(\phi_n,\varphi_n)\in C^{0,1,\alpha}(\overline\cO_m)$ and since the coefficients of $\cL$ are continuously differentiable then $u_n\in C^{1,3,\alpha}_{L oc}(\cO_m)$ for all $n$ by \cite[Thm.\ 3.5.11+Cor.\ 3.5.1]{friedman2008partial}.
\end{proof}
\begin{remark}\label{rem:stab}
{Thanks to \cite[Thm.\ 2.6.5 and Rem.\ 2.6.4]{bensoussan2011applications} the bound \eqref{eq:stab} can be replaced by 
\[
\|\hat{u}_n\|_{W^{1,2,p}(\cO_m)}\leq K \|F(\phi,\nabla\varphi)-F_n(\phi_n,\nabla\varphi_n)\|_{L^{p}(\cO_m)},\quad p\in(1,\infty).
\]
Hence, stability of solutions of \eqref{eq:pdelimitapp} also holds in $W^{1,2,p}(\cO_m)$, i.e., $\lim_n u_n= u$ in $W^{1,2,p}(\cO_m)$.}
\end{remark}

\subsection{Convergence of the sequence $(t^\lambda_{n_k},x^\lambda_{n_k})_{k\in\N}$ in Proposition \ref{prop:locgradp}}

Here we prove that $(\tilde{t},\tilde{x})\in\argmax_{\overline{\cO}_m} v^\lambda$. Arguing by contradiction let us assume $(\tilde{t},\tilde{x})\notin\argmax_{\overline{\cO}_m} v^\lambda$. Then there exists a $\epsilon>0$ such that $v^\lambda(\tilde{t},\tilde{x})\leq \max_{\overline{\cO}_m}v^\lambda-\epsilon$ and so there exists a neighbourhood $U_\epsilon$ of $(\tilde{t},\tilde{x})$ such that $v^\lambda(t,x)\leq \max_{\overline{\cO}_m}v^\lambda-\frac{\epsilon}{2} $ for all $(t,x)\in \overline{U}_\epsilon$. For all sufficiently large $k$'s we also have $(t_{n_k}^\lambda,x_{n_k}^\lambda)\in U_\epsilon$ and by uniform convergence 
\begin{align}\label{eq:uc}
|v^{\lambda,n_k}-v^\lambda|(t,x)\leq \frac{\epsilon}{4},\quad\text{for $(t,x)\in \overline{\cO}_m$}.
\end{align} 
Hence 
\begin{align}\label{eq:uc2}
\max_{(t,x)\in\overline{\cO}_m}v^{\lambda,n_k}(t,x)=v^{\lambda,n_k}(t_{n_k}^\lambda,x_{n_k}^\lambda)\leq v^\lambda(t_{n_k}^\lambda,x_{n_k}^\lambda)+\frac{\epsilon}{4}\leq\max_{(t,x)\in\overline{\cO}_m}v^\lambda(t,x)-\frac{\epsilon}{4},
\end{align}
where the first equality is by definition of $(t^\lambda_{n_k},x_{n_k}^\lambda)$, the first inequality by \eqref{eq:uc} and the final inequality follows by $(t^\lambda_{n_k},x_{n_k}^\lambda)\in \overline{U}_\epsilon$.

With no loss of generality we can assume $v^\lambda$ and $v^{\lambda,n}$ be positive. Otherwise we apply our argument to $\tilde{v}^\lambda=v^\lambda-\min_{\overline\cO_m} v^\lambda+1$ and the associated sequence $\tilde{v}^{\lambda,n}=v^{\lambda,n}-\min_{\overline\cO_m} v^\lambda+1$. By triangular inequality and positivity of $v^\lambda$ and $v^{\lambda,n}$  we have 
\[
\max_{\overline{\cO}_m}|v^\lambda-v^{\lambda,n_k}|\geq \max_{\overline{\cO}_m}|v^\lambda|-\max_{\overline{\cO}_m}|v^{\lambda,n_k}|= \max_{\overline{\cO}_m}v^\lambda-\max_{\overline{\cO}_m}v^{\lambda,n_k}\geq \frac{\epsilon}{4},
\]
for all $k$'s sufficiently large, where the inequality is due to \eqref{eq:uc2}. This contradicts uniform convergence and therefore $(\tilde{t},\tilde{x})\in\argmax_{\overline{\cO}_m} v^\lambda$ as claimed. 

\subsection{Existence of $X^\eps$ in the proof of Theorem \ref{thm:game}}\label{app:Xeps}
For $m\in\N$ let $(b^m,\sigma^m)$ be functions $\R^d\to \R^d\times\R^{d\times d'}$ that are equal to $(b,\sigma)$ on $B_m$ and extend continuously to be constant on $\R^d\setminus B_m$. Similarly, $\alpha^m:\R^{d+1}_{0,T}\to \R^d$ is defined as 
\[
\alpha^m(t,x)=-2\psi_{\varepsilon}'\left(|\nabla u^{\eps}(t,x)|_d^2- f^2(t,x)\right)\nabla u^{\eps}(t,x),\quad \text{on $B_m$},
\] 
and extended continuously to be constant on $\R^d\setminus B_m$. Since $u^{\eps}\in C^{0,1,\alpha}_{\ell oc}(\R^{d+1}_{0,T})$, then thanks to \cite[Thm.\ 1]{veretennikov1981strong} there exists a unique strong solution of
\begin{align*}
X^{m}_s=x+\int_0^s \big(b^m(X^{m}_u)+\alpha^m(t+u,X^m_u)\big)\ud u+\int_0^s\sigma^m(X^{m}_u)\ud W_u,\quad s\in[0,T-t].
\end{align*}
Notice that $X^m=X^{m;t}$ depends on $t$ via the time-inhomogeneous drift $\alpha^m$ but we omit it for simplicity. Notice also that here we should understand $n^m_s\dot{\nu}^m_s=\alpha^m(t+s,X^m_s)$ and $X^m=X^{m;[n^m,\nu^m]}$. Letting $\zeta_{m,k}= \inf\{s\ge 0:|X^{m}_s|_d\geq k\}$, for any $m\geq k$ we have $X^m_{s\wedge\zeta_{m,k}}=X^k_{s\wedge\zeta_{k,k}}$ for all $s\in[0,T-t]$, $\P_x$-a.s.\ (i.e., the two processes are indistinguishable). Thus, setting $X^\eps_s(\omega):=X^k_{s}(\omega)$ for $s<\zeta_{k,k}(\omega)$ and denoting $\tau_k=\inf\{s\ge 0: |X^\eps_s|_d\ge k\}$ it is clear that, by uniqueness of strong solutions and the definition of the pair $(n^\eps,\nu^\eps)$, the process $X^\eps$ satisfies
\begin{align*}
X^{\eps}_{s\wedge\tau_k}=x+\int_0^{s\wedge\tau_k} \big(b(X^{\eps}_u)+n^\eps_u\dot{\nu}^\eps_u\big)\ud u+\int_0^{s\wedge\tau_k}\sigma(X^{\eps}_u)\ud W_u,\quad s\in[0,T-t].
\end{align*}
By continuity of paths $\tau_k\le \tau_{k+1}$ and $\tau_\infty:=\lim_{k\to\infty}\tau_k$ is well-defined. Moreover, $\nu^\eps$ satisfies the same bound as in \eqref{eq:Zinftybnd} thanks to \eqref{eq:nnueps}. Therefore, linear growth of the coefficients of the SDE and the same arguments as those at the end of the proof of Proposition \ref{lem:prbraprRD} imply that $\P_x(\tau_\infty\le T-t)=0$. Then $X^\eps$ is well-defined on $[0,T-t]$.

\subsection{Shaefer's fixed point theorem}
\begin{theorem}[Thm.\ 9.2.4 in \cite{evans10}]\label{thm:Schaefer}
Suppose $T:\cD\to\cD$ is a continuous and compact mapping on a Banach space $\cD$. Assume further that the set
\begin{align*}
\left\{f\in\cD\,|\,f=\rho T[f]\text{ for some }\rho\in[0,1]\right\}
\end{align*}
is bounded. Then $T$ has a fixed point.
\end{theorem}

\bibliographystyle{plain}
\bibliography{Bibliography}

\end{document}